%% file: cross-ratio_counting.tex
\documentclass[11pt,reqno,a4]{amsart}

\usepackage[english]{babel}
\usepackage{amsthm}
\usepackage{graphicx}
\usepackage{xcolor} 
\usepackage{transparent}
\usepackage{amssymb}
\usepackage{amsmath}
\usepackage[arrow, matrix, curve]{xy}
\usepackage{units} 
\usepackage{MnSymbol} 
\usepackage{here} 
\usepackage{enumerate} 
\usepackage{comment} 
\usepackage{listings} 
\usepackage{pdfpages} 
\usepackage{multirow} 
\usepackage{multicol} 
\usepackage{cancel} 
\usepackage{wrapfig} 
\usepackage{float}
\usepackage{hyperref}
\usepackage{tikz-cd}
\usepackage{enumitem} 

\usetikzlibrary{%
  matrix,%
  calc,%
  arrows%
}

\restylefloat{table}

\hypersetup{
    colorlinks,
    linkcolor={red!70!black},
    citecolor={green!60!black},
    urlcolor={blue!80!black}
}

\newtheoremstyle{dotless}{}{}{\itshape}{}{\bfseries}{}{\newline}{}
\newtheoremstyle{definition}{}{}{}{}{\bfseries}{}{\newline}{}

\theoremstyle{dotless}
\newtheorem{corollary}{Corollary}[section]
\newtheorem{lemma}[corollary]{Lemma}
\newtheorem{theorem}[corollary]{Theorem}
\newtheorem{proposition}[corollary]{Proposition}
\newtheorem*{theorem*}{Theorem}
\theoremstyle{definition}
\newtheorem{construction}[corollary]{Construction}
\newtheorem{definition}[corollary]{Definition}
\newtheorem{remark}[corollary]{Remark}
\newtheorem{example}[corollary]{Example}

\newtheorem{notation}[corollary]{Notation}
\newtheorem{algorithm}[corollary]{Algorithm}

\newcommand{\mult}{\operatorname{mult}}

\newcommand{\ft}{\operatorname{ft}}
\newcommand{\ev}{\operatorname{ev}}
\newcommand{\degree}{\operatorname{deg}}
\newcommand{\val}{\operatorname{val}}
\newcommand{\coker}{\operatorname{coker}}

\newcommand{\changelocaltocdepth}[1]{
  \addtocontents{toc}{\protect\setcounter{tocdepth}{#1}}%
  \setcounter{tocdepth}{#1}%
}

\makeatletter		
\def\l@subsection{\@tocline{2}{0pt}{2.6pc}{5pc}{}}
\makeatother

\changelocaltocdepth{1}
\begin{document}
\pagenumbering{arabic}
\setcounter{tocdepth}{0}

\author{Christoph Goldner}
\address{ Eberhard Karls Universit\"{a}t T\"{u}bingen, Germany}
\email{\href{mailto:christoph.goldner@math.uni-tuebingen.de}{christoph.goldner@math.uni-tuebingen.de}}

\title{Counting tropical rational curves with cross-ratio constraints}

\date{\today}

\begin{abstract}
We enumerate rational curves in toric surfaces passing through points and satisfying cross-ratio constraints using tropical and combinatorial methods. Our starting point is \cite{IlyaCRC}, where a tropical-algebraic correspondence theorem was proved that relates counts of rational curves in toric varieties that satisfy point conditions and cross-ratio constraints to the analogous tropical counts. We proceed in two steps: based on tropical intersection theory we first study tropical cross-ratios and introduce degenerated cross-ratios. Second we provide a lattice path algorithm that produces all tropical curves satisfying such degenerated conditions explicitly. In a special case simpler combinatorial objects, so-called cross-ratio floor diagrams, are introduced which can be used to determine these enumerative numbers as well.
\end{abstract}

\maketitle

\tableofcontents

\changelocaltocdepth{1}
\section*{Introduction}

\textit{Tropical geometry} is a rather young field of mathematics that is intimately connected to algebraic geometry, non-Archimedean analytic geometry and combinatorics. In the past tropical geometry turned out to be a powerful tool to answer enumerative questions. To apply tropical geometry to enumerative questions, so-called correspondence theorems are needed. A correspondence theorem states that an enumerative number equals its tropical counterpart, where in tropical geometry we have to count each tropical object with a suitable multiplicity reflecting the number of classical objects in our counting problem that tropicalize to the given tropical object.
Thus tropical geometry hands us a new approach to enumerative problems: first find a suitable correspondence theorem, then use combinatorics to enumerate the tropical objects in question. A famous example is the following: let $d\in\mathbb{N}_{>0}$ be a degree and assume that points in general position in $\mathbb{P}^2$ are given in  such a way that only finitely many rational plane curves of degree $d$ pass through these points. What is the number $N_d$ of curves passing through these points? For $d\leq 5$, this question can be answered using methods from classical algebraic geometry. In the '90s, Kontsevich presented a recursive formula that can compute $N_d$ for arbitrary $d$ \cite{KontsevichOriginal}. Tropical geometry offers a new approach to compute the numbers $N_d$, and generalizations thereof: in \cite{MikhalkinFundamental}, Mikhalkin pioneered the use of tropical methods in enumerative geometry by proving a correspondence theorem for counts of curves in toric surfaces satisfying point conditions. 

\textit{Moduli spaces} of (stable) curves resp.\ maps to toric surfaces are an important tool in enumerative geometry, both in algebraic and in tropical geometry. Often, an enumerative problem can be expressed as an intersection product on the moduli space parametrizing the objects to be counted.
Gathmann and Markwig started to use tropical moduli space techniques in order to give a tropical proof of Kontsevich's formula in \cite{KontsevichPaper}. 
Both in the original proof of Kontsevich and in this tropical proof, the count of rational plane curves of degree $d$ satisfying point, line and a cross-ratio condition is an essential ingredient.

A \textit{cross-ratio} is a rational number associated to four collinear points. It encodes the relative position of these four points to each other. It is invariant under projective transformations and can therefore be used as a constraint that four points on $\mathbb{P}^1$ should satisfy. So a cross-ratio can be viewed as a condition on elements of the moduli space of $n$-pointed rational stable maps to a toric variety. Tropical cross-ratios were first introduced by Mikhalkin under the name ``tropical double ratio'' in \cite{MikhalkinCRC} and can be thought of as paths of fixed lengths in a tropical curve. More precisely: A plane tropical curve is a $1$-dimensional polyhedral complex (mapped to $\mathbb{R}^2$ satisfying the balancing condition) whose unbounded polyhedra (points on a tropical curve are contracted unbounded polyhedra) are uniquely labeled (see definition \ref{definition:degree}) and a tropical cross-ratio is given by $4$ labels and a length such that forgetting all unbounded polyhedra which are not given in the cross-ratio leaves a tropical curve whose bounded parts' lengths sum up to the given length in the cross-ratio -- see Figure \ref{Example_Introduction}. It is natural to ask: Given point conditions $p_1,\dots,p_n$ and cross-ratio constraints $\lambda_1,\dots,\lambda_l$ in such a way that there are only finitely many rational (tropical) curves of a given degree in a toric surface satisfying them, then
\begin{enumerate}[label=(\arabic*),ref=(\arabic*)]
\item \label{Q1} \textit{How many of these curves are there?}
\item \label{Q2} \textit{Can we construct them?}
\end{enumerate}
These questions motivated the study in this paper. Recall that applying tropical geometry to an enumerative problem happens in two steps: use a correspondence theorem, then use combinatorics. The correspondence theorem we are going to use is provided by Tyomkin in \cite{IlyaCRC}. Our approach to answer questions \ref{Q1} and \ref{Q2} can be subdivided into two steps. The first step is to develop a notion of \textit{degenerated} tropical cross-ratios that helps us to simplify the combinatorics. The second step is to explicitly construct all rational tropical curves that satisfy the given point and degenerated cross-ratio conditions using combinatorial methods. We want to explain these two steps and the methods used more precisely:

\begin{figure}
\centering
\def\svgwidth{300pt}
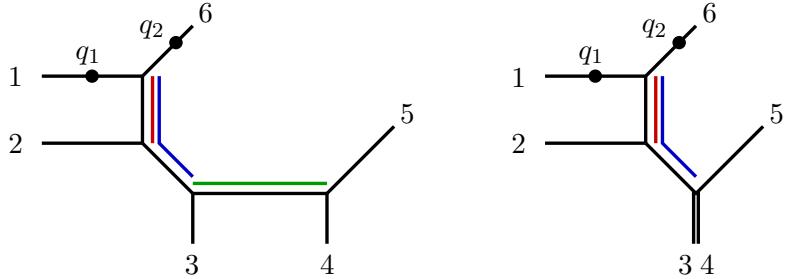
\caption{Left: A degree two plane tropical curve that is fixed by two points $q_1,q_2$ and three cross-ratios: the red length associated to the four labels $(16|23)$, the blue length associated to the four labels $(1q_2|35)$ and the green length associated to the four labels $(23|45)$ are fixed. Right: Degenerating the cross-ratio associated to the green path means shrinking the green path, thus producing a $4$-valent vertex (with two unbounded edges on top of each other).}
\label{Example_Introduction}
\end{figure}

\subsection*{Degenerated cross-ratios}
In section \ref{section:tropical_cross-ratios} a generalization of Mikhalkin's definition of tropical cross-ratios is introduced that allows us to use tropical intersection theory in order to degenerate tropical cross-ratios. If we think of a cross-ratio as a path of fixed length in a tropical curve, then a degenerated cross-ratio is a path of length zero -- see Figure \ref{Example_Introduction}. Obviously, the set of tropical curves satisfying given conditions becomes easier when degenerating the cross-ratios. The difficult part is to determine the multiplicities with which we have to count such curves. These multiplicities have a local description, which we present in theorem \ref{thm:ZSFSSG_section_2} together with the fact that the number of tropical curves satisfying point and cross-ratio conditions stays invariant when degenerating the cross-ratios. The techniques used to prove theorem \ref{thm:ZSFSSG_section_2} are \textit{tropical moduli spaces} and \textit{tropical intersection theory}. 

Moduli spaces of abstract rational tropical curves were studied in \cite{MikhalkinCRC}. They also show up in the study of the tropical Grassmannian as the space of trees \cite{TropicalGrassmannian, AK}.  
It turns out that these tropical moduli spaces are tropicalizations of the corresponding moduli spaces in algebraic geometry in a suitable embedding \cite{GM, Tev}. Tropicalizations of moduli spaces of curves of higher genus (in a toroidal and non-Archimedean setting) were studied by Abramovich, Caporaso and Payne \cite{AbramovichCaporasoPayne}. The theory of rational tropical stable maps was introduced by Gathmann, Kerber and Markwig in \cite{GathmannKerberMarkwig}. Recently, Ranganathan \cite{DhruvStableMapsOne} tropicalized the moduli space of stable rational maps to toric surfaces using logarithmic and non-Archimedean geometry. An excellent overview of the current development concerning compactifications of moduli spaces and tropical moduli spaces can be found in \cite{ICM2018}.

We use tropical intersection theory on moduli spaces of rational stable maps, building on Allermann and Rau \cite{FirstStepsIntersectionTheory, JohannesIntersectionsonTropModuliSpaces}. Katz \cite{Katz2012} related tropical intersection theory to intersection theory on toric varieties studied by Fulton and Sturmfels in \cite{FultonSturmfels}. For matroidal fans (i.e. tropicalizations of linear spaces) Shaw offers in \cite{IntersectionMatroidalFans} a framework of tropical intersection theory. Tropical intersection theory is still an active area of research.

All in all degenerating cross-ratios is a natural approach in the following sense: A tropical curve satisfying non-degenerated conditions can be \textit{degenerated} to a curve that satisfies degenerated conditions itself. This observation allows us to answer question \ref{Q2} if we can construct tropical curves that satisfy degenerated conditions. We offer an algorithm for this construction in section \ref{section:cross-ratio_lattice_path_algorithm}.

\subsection*{Combinatorial methods} Both the lattice path algorithm and floor diagrams are well-known combinatorial tools in tropical geometry. In section \ref{section:cross-ratio_lattice_path_algorithm} we generalize the \textit{lattice path algorithm} to a \textit{cross-ratio lattice path algorithm}. Lattice paths were used in \cite{MikhalkinLatticePaths} and \cite{MikhalkinFundamental} to construct curves satisfying point conditions. Since we want to find tropical curves that satisfy point and degenerated cross-ratio conditions, we need to generalize this approach. There are other generalizations (in particular \cite{rag-rug}) of lattice paths that inspired our definition of cross-ratio lattice paths. The lattice path algorithm can also be extended to determine invariants connected to counts of real curves as well, see \cite{Shustin}. 

 In section \ref{section:duality:tropical_curves_subdivisions} we prove theorem \ref{theorem:same_numbers}, which states that the lattice path algorithm yields the number of tropical curves satisfying point and cross-ratio conditions. Thus theorem \ref{theorem:same_numbers} answers question \ref{Q1}. Moreover, the cross-ratio lattice path algorithm we provide allows us to construct all tropical curves of a given degree that satisfy the given point conditions and the degenerated cross-ratio constraints.

In section \ref{section:special_case:floor_diagrams_cross-ratios} we restrict to curves in Hirzebruch surfaces and impose a restriction to our cross-ratios such that we can use simpler combinatorial objects than the ones we deal with when applying the cross-ratio lattice path algorithm. These simpler combinatorial objects are called \textit{cross-ratio floor diagrams}. They are a generalization of \textit{floor diagrams}. Floor diagrams are graphs that arise from so-called floor decomposed tropical curves by forgetting some information. Floor diagrams were introduced by Mikhalkin and Brugall\'{e} in \cite{MikhalkinBrugalleFloorDiagIntroduction} (and \cite{MikhalkinBrugalle}) to give a combinatorial description of Gromov-Witten invariants of Hirzebruch surfaces. Floor diagrams have also been used to establish polynomiality of the node polynomials \cite{MikhalkinFomin} and to give an algorithm to compute these polynomials in special cases -- see \cite{BlockNodePoly}. Moreover, floor diagrams have been generalized, for example in case of $\Psi$-conditions, see \cite{PsiFloorDiagrams}, or for counts of curves relative to a conic \cite{Bru}.

Theorem \ref{theorem:CR_count=floor_diag_count} states that counting floor diagrams yields the same numbers as counting tropical curves that satisfy point and cross-ratio conditions. Hence floor diagrams offer (besides the cross-ratio lattice path algorithm) another (simpler) way of answering question \ref{Q1}.

\subsection*{Acknowledgements}
The author is indebted to Hannah Markwig for interesting discussions and many suggestions.
The author would like to thank Ilya Tyomkin for interesting discussions.
This work was partially completed during the program ``Tropical Geometry, Amoebas and Polytopes" at the Institute
Mittag-Leffler in spring 2018. The author would like to thank the institute for its hospitality.
The author gratefully acknowledges support by DFG-collaborative research center TRR 195 (INST 248/237-1).

\changelocaltocdepth{2}
\section{Preliminaries}\label{section:preliminaries}

In this preliminary section we give a short introduction to tropical intersection theory and tropical moduli spaces as needed in this paper.
We fix the following conventions: polytopes are convex, and we work over a non-Archimedean closed field of characteristic zero.

\subsection*{Tropical intersection theory}
This subsection summarizes intersection theoretic background from  \cite{Allermann, FirstStepsIntersectionTheory, AllermannHampeRau}. 

\begin{definition}[Normal vectors and balanced fans]
Let $V:=\Gamma\otimes_{\mathbb{Z}}\mathbb{R}$ be the real vector space associated to a given lattice $\Gamma$ and let $X$ be a fan in $V$. The lattice generated by $\operatorname{span}(\kappa)\cap\Gamma$, where $\kappa$ is a cone of $X$, is denoted by $\Gamma_\kappa$. Let $\sigma$ be a cone of $X$ and $\tau$ be a face of $\sigma$ of dimension $\dim(\tau)=\dim(\sigma)-1$ (we write $\tau<\sigma$). A vector $u_{\sigma}\in\Gamma_\sigma$ that generates $\Gamma_\sigma / \Gamma_\tau$ such that $u_\sigma+\tau\subset\sigma$ defines a class $u_{\sigma / \tau}:=[u_\sigma]\in\Gamma_\sigma / \Gamma_\tau$ that does not depend on the choice of $u_\sigma$. This class is called \textit{normal vector of $\sigma$ relative to $\tau$}.

$X$ is a \textit{weighted fan of dimension $k$} if $X$ is of pure dimension $k$ and there are weights on its facets, that is there is a map $\omega_X:X^{(k)}\to\mathbb{Z}$. The number $\omega_X(\sigma)$ is called \textit{weight} of the facet $\sigma$ of $X$. To simplify notation, we write $\omega(\sigma)$ if $X$ is clear. Moreover, a weighted fan $(X,\omega_X)$ of dimension $k$ is called a \textit{balanced} fan of dimension $k$ if
\begin{align*}
\sum_{\sigma \in X^{(k)}, \tau < \sigma} \omega(\sigma)\cdot u_{\sigma / \tau} = 0
\end{align*}
holds in $V/\langle\tau\rangle_{\mathbb{R}}$ for all faces $\tau$ of dimension $\dim(\tau)=\dim(\sigma)-1$.
\end{definition}

\begin{definition}[Group of affine cycles]
Let $V:=\Gamma\otimes_{\mathbb{Z}}\mathbb{R}$ be the real vector space associated to a given lattice $\Gamma$. A \textit{tropical fan (of dimension $k$)} is a balanced fan of dimension $k$. $[(X,\omega_X)]$ denotes the refinement class of a tropical fan $X$ with weights $\omega_X$.
\end{definition}

\begin{definition}[Rational functions]
Let $C$ be an affine $k$-cycle. A \textit{(non-zero) rational function on $C$} is a continuous piecewise linear function $\varphi:|C|\to\mathbb{R}$, i.e. there exists a representative $(X,\omega_X)$ of $C$ such that on each cone $\sigma\in X$ the map $\varphi$ is the restriction of an integer affine linear function. The set of (non-zero) rational functions of $C$ is denoted by $\mathcal{K}^*(C)$. Define $\mathcal{K}(C):=\mathcal{K}^*(C)\cup \{-\infty\}$ such that $(\mathcal{K}(C),\operatorname{max},+)$ is a semifield, where the constant function $-\infty$ is the ``zero" function.
\end{definition}

\begin{definition}[Divisor associated to a rational function]\label{definition:Assoc_Weil_Div}
Let $C$ be an affine $k$-cycle in $V=\Gamma\otimes_{\mathbb{Z}}\mathbb{R}$ and $\varphi\in\mathcal{K}^*(C)$ a rational function on $C$. Let $(X,\omega)$ be a representative of $C$ on whose cones $\varphi$ is affine linear and denote these linear pieces by $\varphi_\sigma$. We denote by $X^{(i)}$ the set of all $i$-dimensional cones of $X$. We define $\operatorname{div}(\varphi):=\varphi\cdot C:= [(\bigcup_{i=0}^{k-1}X^{(i)},\omega_{\varphi})]\in Z^{\textrm{aff}}_{k-1}(C)$, where
\begin{align*}
\omega_\varphi : X^{(k-1)} &\to \mathbb{Z}\\
\tau &\mapsto\sum_{\sigma \in X^{(k)}, \tau < \sigma} \varphi_\sigma(\omega(\sigma)v_{\sigma/\tau})-\varphi_\tau \left( \sum_{\sigma \in X^{(k)}, \tau < \sigma}\omega(\sigma)v_{\sigma/\tau}\right)
\end{align*}
and the $v_{\sigma/\tau}$ are arbitrary representatives of the normal vectors $u_{\sigma/\tau}$. If $D$ is an affine $k$-cycle in $C$, we define $\varphi\cdot D:=\varphi\mid_{|D|}\cdot D$.
\end{definition}

\begin{example}\label{example:pull_back_0_M_0,4}
Let $[(X,\omega_X)]$ be the affine $1$-cycle with representative $(X,\omega_X)$ whose weights are all $1$ and whose $1$-dimensional rays are given by $-e_x,-e_y,e_x+e_y$, where $e_x,e_y$ are the vectors of the standard basis of $\mathbb{R}^2$ such that $X\subset\mathbb{R}^2$. Then
\begin{align*}
\varphi: X &\to \mathbb{R}\\
(x,y)&\mapsto \max(x,y,0)
\end{align*}
is a rational function on $[(X,\omega_X)]$ and $(X,\omega_X)$ is a representative such that $\varphi$ is integer linear affine on each cone. The divisor associated to $\varphi$, namely $\varphi\cdot X$, is given by the $1$-skeleton of $X$ which is just one point (namely $0\in\mathbb{R}^2$) and that point has weight $1$. We calculate this weight as an example: Let $\tau=0\in\mathbb{R}^2$, $\sigma_1=\operatorname{cone}\left( -e_x  \right),\sigma_2=\operatorname{cone}\left(-e_y\right)$ and $\sigma_3=\operatorname{cone}\left(e_x+e_y\right)$ be cones of $X$. Applying definition \ref{definition:Assoc_Weil_Div}, we get
\begin{align*}
\omega_\varphi(\tau)&=\varphi_{\sigma_1}\left(\omega(\sigma_1)v_{\sigma_1/\tau}\right)
+\varphi_{\sigma_2}\left(\omega(\sigma_2)v_{\sigma_2/\tau}\right)
+\varphi_{\sigma_3}\left(\omega(\sigma_3)v_{\sigma_3/\tau}\right)\\
&\quad -\varphi_\tau\left( \omega(\sigma_1)v_{\sigma_1/\tau} + \omega(\sigma_2)v_{\sigma_2/\tau} + \omega(\sigma_3)v_{\sigma_3/\tau} \right)\\
&=\varphi_{\sigma_3}\left(\omega(\sigma_3)v_{\sigma_3/\tau}\right)\\
&=\varphi_{\sigma_3}\left( 1 (e_x+e_y) \right) = 1
\end{align*}
because $\varphi_{\sigma_1},\varphi_{\sigma_2},\varphi_{\tau}\equiv 0$ and $\varphi_{\sigma_3}\left(e_x+e_y \right)=\max(1,1,0)$.
\end{example}

\begin{definition}[Affine intersection product]
Let $C$ be an affine $k$-cycle. The subgroup of globally linear functions in $\mathcal{K}^*(C)$ with respect to $+$ is denoted by $\mathcal{O}^*(C)$. We define the \textit{group of affine Cartier divisors of $C$} to be the quotient group $\operatorname{Div}(C):=\mathcal{K}^*(C)/\mathcal{O}^*(C)$. Let $[\varphi]\in\operatorname{Div}(C)$ be a Cartier divisor. The divisor associated to this function is denoted by $\operatorname{div}([\varphi]):=\operatorname{div}(\varphi)$ and is well-defined. The following bilinear map is called \textit{affine intersection product}
\begin{align*}
\cdot\, : \operatorname{Div}(C)\times Z^{\textrm{aff}}_k(C) &\to Z^{\textrm{aff}}_{k-1}(C)\\
([\varphi],D) &\mapsto [\varphi]\cdot D := \varphi\cdot D.
\end{align*}
\end{definition}

\begin{definition}[Morphisms of fans]
Let $X$ be a fan in $V=\Gamma\otimes_{\mathbb{Z}}\mathbb{R}$ and $Y$ a fan in $V'=\Gamma'\otimes_{\mathbb{Z}}\mathbb{R}$. A \textit{morphism} $f:X\to Y$ is a $\mathbb{Z}$-linear map from $|X|\subseteq V$ to $|Y|\subseteq V'$ induced by a $\mathbb{Z}$-linear map on the lattices. A \textit{morphism of weighted fans} is a morphism of fans. A \textit{morphism of affine cycles} $f:[(X,\omega_X)]\to [(Y,\omega_Y)]$ is a morphism of weighted fans $f:X^*\to Y^*$ and does not depend on the choice of representatives.
\end{definition}

\begin{definition}[Push-forward of affine cycles]
Let $V=\Gamma\otimes_{\mathbb{Z}}\mathbb{R}$ and $V'=\Gamma'\otimes_{\mathbb{Z}}\mathbb{R}$. Let $[X]\in Z^{\textrm{aff}}_m(V)$ and $[Y]\in Z^{\textrm{aff}}_n(V')$ be cycles with representatives $(X,\omega_X)$ and $Y$. Let $f:X\to Y$ be a morphism. Choosing a refinement of $(X,\omega_X)$, the set of cones
\begin{align*}
f_*X:=\{f(\sigma)\mid \sigma\in X \textrm{ contained in a maximal cone of $X$ on which $f$ is injective} \}
\end{align*}
is a tropical fan in $V'$ of dimension $m$ with weights
\begin{align*}
\omega_{f_*X}(\sigma'):=\sum_{\sigma\in X^{(m)}: \, f(\sigma)=\sigma'} \omega_X(\sigma)\cdot |\Gamma_{\sigma'}'/f(\Gamma_\sigma)|
\end{align*}
for all $\sigma'\in f_*X^{(m)}$. The equivalence class of $(f_*X,\omega_{f_*X})$ is uniquely determined by the equivalence class of $(X,\omega_X)$.
For $[(Z,\omega_Z)]\in Z^{\textrm{aff}}_k([X])$ we define
\begin{align*}
f_*[(Z,\omega_Z)]:=[(f_*(Z^*),\omega_{f_*(Z^*)})]\in Z^{\textrm{aff}}_k([Y])
\end{align*}
The map
\begin{align*}
Z^{\textrm{aff}}_k([X])\to Z^{\textrm{aff}}_k([Y]), \, C\mapsto f_*C
\end{align*}
is well-defined, $\mathbb{Z}$-linear and $f_*C$ is called \textit{push-forward of $C$ along $f$}.
\end{definition}

\begin{definition}[Pull-back of Cartier divisors]
Let $C\in Z^{\textrm{aff}}_m(V)$ and $D\in Z^{\textrm{aff}}_n(V')$ be cycles in $V=\Gamma\otimes_{\mathbb{Z}}\mathbb{R}$ and $V'=\Gamma'\otimes_{\mathbb{Z}}\mathbb{R}$. Let $f:C\to D$ be a morphism. The map
\begin{align*}
\operatorname{Div}(D)&\to\operatorname{Div}(C)\\
[h]&\mapsto f^*[h]:=[h\circ f]
\end{align*}
is well-defined, $\mathbb{Z}$-linear and $f^*[h]$ is called \textit{pull-back of $[h]$ along $f$}.
\end{definition}

\begin{definition}[Rational equivalence and Chow groups]
Let $C$ be an (abstract) cycle. Let $R(C):=\{(|C|,h)\mid h \textrm { bounded}\}\subseteq\operatorname{Div}(C)$ be the subgroup of Cartier divisors on $C$ globally given by a bounded rational function and
\begin{align*}
\operatorname{Pic}(C):=\operatorname{Div}(C)/R(C)
\end{align*}
be the \textit{Picard group} of $C$, where pull-backs induced by pull-backs of Cartier divisors are well-defined. Let $D,D'$ be subcycles of $C$. We call $D$ \textit{rationally equivalent to zero on} $C$ if there exists a cycle $C'$ of dimension $\operatorname{dim}(D)+1$, a morphism $f:C'\to C$ and a bounded rational function $h\in R(C')$ such that
\begin{align*}
f_*(h\cdot C')=D.
\end{align*}
We call $D$ and $D'$ \textit{rationally equivalent} (notation: $D \sim D'$) if $D-D'$ is rationally equivalent to zero. The $k$-th \textit{Chow group} of $C$ is defined as
\begin{align*}
A_k(C):=Z_k(C)/\sim,
\end{align*}
where the intersection product $Z_{n-k}(\mathbb{R}^n)\times Z_{n-l}(\mathbb{R}^n)\to Z_{n-k-l}(\mathbb{R}^n)$ induces a well-defined bilinear map (proposition 1.8.10 of \cite{Allermann})
\begin{align*}
A_{n-k}(\mathbb{R}^n)\times A_{n-l}(\mathbb{R}^n) &\to A_{n-k-l}(\mathbb{R}^n)\\
([E],[F]) &\mapsto [E]\cdot [F]:=[E\cdot F].
\end{align*}
\end{definition}

In the following we consider cycles up to rational equivalence. As an example, two arbitrary points in $\mathbb{R}^2$ (viewed as $0$-dimensional cycles, see example \ref{example:pull_back_point}) are rationally equivalent.

\begin{definition}[Degree map]
Let $C$ be a cycle. The map
\begin{align*}
\operatorname{deg}:A_0(C) &\to \mathbb{Z}\\
[\omega_1 P_1+\dots+\omega_r P_r] &\mapsto \sum_{i=1}^r \omega_i
\end{align*}
is a well-defined morphism and for $D\in A_0(C)$ the number $\operatorname{deg}(D)$ is called the \textit{degree of $D$}. 
\end{definition}

\begin{remark}\label{remark:facts_about_rational_equivalence}
The most important facts about rational equivalence that we will use are the following:
\begin{itemize}
\item[(a)]
Pull-backs of rationally equivalent cycles are rationally equivalent.
\item[(b)]
If two $0$-dimensional cycles are rationally equivalent, then their numbers obtained by the degree map are the same.
\item[(c)]
Two cycles in $\mathbb{R}^n$ that only differ by a translation are rationally equivalent.
\end{itemize}
\end{remark}

\subsection*{Tropical moduli spaces}

This subsection collects background on tropical moduli spaces following \cite{GathmannKerberMarkwig}.

\begin{definition}[Moduli space of abstract tropical curves]
An \textit{abstract rational tropical curve} is a metric tree $\Gamma$ with unbounded edges called ends and with $\val(v)\geq 3$ for all vertices $v\in\Gamma$. It is called $n$-\textit{marked tropical curve} $(\Gamma,x_1,\dots,x_n)$ if $\Gamma$ has exactly $n$ ends that are labeled with pairwise different $x_1,\dots,x_n\in\mathbb{N}$. Two $n$-marked tropical curves $(\Gamma,x_1,\dots,x_n)$ and $(\tilde{\Gamma},\tilde{x}_1,\dots,\tilde{x}_n)$ are isomorphic if there is a homeomorphism $\Gamma\to \tilde{\Gamma}$ mapping $x_i$ to $\tilde{x}_i$ for all $i$ and each edge of $\Gamma$ is mapped onto an edge of $\tilde{\Gamma}$ by an affine linear map of slope $\pm 1$. The set $\mathcal{M}_{0,n}$ of all $n$-marked tropical curves up to isomorphism is called \textit{moduli space of $n$-marked tropical curves}. Forgetting all lengths of an $n$-marked tropical curve gives us its \textit{combinatorial type}.
\end{definition}

\begin{remark}[$\mathcal{M}_{0,n}$ is a tropical fan]
We have the \textit{distance map}
\begin{align*}
\operatorname{dist}:\mathcal{M}_{0,n} &\to\mathbb{R}^{n\choose 2}\\
\Gamma &\mapsto (\textrm{length of the path from end $i$ to end $j$ })_{ij}
\end{align*}
and define $v_I$ ($I\subset \{1,\dots ,n\}, |I|\geq 2, |I^C|\geq 2$) to be the image under dist of the $n$-marked tropical curve that has only one bounded edge of length one with markings $I$ on one and markings $I^C$ on the other side. Moreover, the map
\begin{align*}
\phi:\mathbb{R}^n &\to \mathbb{R}^{n\choose 2}\\
a &\mapsto (a_i+a_j)_{ij}
\end{align*}
induces (by abuse of notation) an injective map
\begin{align*}
\operatorname{dist}:\mathcal{M}_{0,n} \to\mathbb{R}^{n\choose 2}/\operatorname{Im}(\phi).
\end{align*}
If we choose
\begin{align*}
\Lambda_n:=\sum_{\substack{I\subset \{1,\dots ,n\}\\ |I|\geq 2}}v_I\mathbb{Z}
\end{align*}
to be the lattice of $\mathbb{R}^{n\choose 2}/\operatorname{Im}(\phi)$, then $\mathcal{M}_{0,n}\subseteq\mathbb{R}^{n\choose 2}/\operatorname{Im}(\phi)$ is a tropical fan of pure dimension $n-3$ with its fan structure given by combinatorial types, and with all weights equal one.
\end{remark}

\begin{figure}
\centering
\def\svgwidth{150pt}
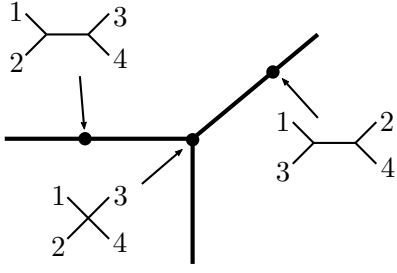
\caption{One way of embedding the moduli space $\mathcal{M}_{0,4}$ into $\mathbb{R}^2$ centered at the origin of $\mathbb{R}^2$. The length of a bounded edge of a tropical curve depicted above is given by the distance of the point in $\mathcal{M}_{0,4}$ parametrizing this curve from the origin of $\mathbb{R}^2$.}
\label{Example_M_0_4}
\end{figure}

\begin{definition}[Degree]\label{definition:degree}
Let $|\Delta|\in\mathbb{N}_{>0}$. A set $\Delta:=\lbrace (v_i,x_i)\rbrace_{i=1,\dots,|\Delta|}$ of tuples is called \textit{degree} if
\begin{itemize}
\item[(1)]
$0\neq v_i\in\mathbb{R}^2$ for all $i=1,\dots,|\Delta|$, and $\langle v_1,\dots,v_{|\Delta|}\rangle=\mathbb{R}^2$, and $\sum_i v_i=0$.
\item[(2)]
$x_i\in\mathbb{N}_{>0}$ for all $i=1,\dots,|\Delta|$, and $x_i\neq x_j$ for all $i\neq j$. An $x_i$ is called \textit{label}.
\end{itemize}
Let $\Sigma$ be a $2$-dimensional lattice polytope in $\mathbb{R}^2$ with facets $E_1,\dots,E_m$ whose lattice lengths are denoted by $|E_1|,\dots,|E_m|$ and let $e_1,\dots,e_m$ be unordered partitions of $E_1,\dots,E_m$, that is $e_i$ is a partition of $E_i$ of some length denoted by $l(e_i)$ for $i=1,\dots,m$. If
\begin{align*}
\lbrace v_i \rbrace_{i=1,\dots,|\Delta|}=\bigcup_{i=1}^{m}\bigcup_{j=1}^{l(e_i)}\lbrace e_{i_j}\cdot\operatorname{pnv}(E_i) \rbrace,
\end{align*}
where $\operatorname{pnv}(E_i)$ is the primitive normal vector of $E_i^\perp$ for $i=1,\dots,m$, then $\Delta$ is said to be associated to a polytope $\Sigma$ with partitions $e_1,\dots,e_m$ and is referred to as $\Delta\left( \Sigma(e_1,\dots,e_m)\right)$.

Important special cases that we use later are the following:
\begin{itemize}
\item
If each entry of each partition $e_i$ is one, then the associated degree is denoted by $\Delta(\Sigma)$.
\item
In case of degree $d$ curves in $\mathbb{P}^2$, the degree $\Delta$ is defined as follows:
Let $\Sigma_d$ be the convex hull of $\lbrace (0,0),(d,0),(d,0)\rbrace\in\mathbb{R}^2$ for some $d\in\mathbb{N}_{>0}$, then $\Delta_d$ is the degree associated to $\Sigma_d$, where the labels are given by: vectors parallel to (and with the same direction as) $(-1,0)\in\mathbb{R}^2$ have labels $1,\dots,d$, vectors parallel to (and with the same direction as) $(0,-1)$ have labels $d+1,\dots,2d$ and vectors parallel to (and with the same direction as) $(1,1)$ have labels $2d+1,\dots,3d$.
\item
In case of degree $(|\alpha|,|\beta|)$ curves of contact orders $\alpha,\beta$ in the first Hirzebruch surface, the degree $\Delta$ is defined as follows:
Let $s\in\mathbb{N}_{>0}$ and $b\in\mathbb{N}$. Let $\alpha=(\alpha_1,\dots)$ be an unordered partition of $b+s$, let $\beta=(\beta_1,\dots)$ be an unordered partition of $b$ and let $\Sigma(\alpha,\beta)$ be the convex hull of $\lbrace (0,0),(s,0),(s,b),(0,b+s)\rbrace\in\mathbb{R}^2$. We associate the degree $\Delta\left(\alpha,\beta\right)$ to the polytope $\Sigma(\alpha,\beta)$, where the partition of the left facet is given by $\alpha$ and the partition of the right facet is given by $\beta$. Moreover, vectors parallel to (and with the same direction as) $(-1,0)\in\mathbb{R}^2$ have labels $1,\dots, l(\alpha)$, vectors parallel to (and with the same direction as) $(1,0)$ have labels $l(\alpha)+1,\dots,l(\alpha)+l(\beta)$.
\end{itemize}

\end{definition}

\begin{definition}[Moduli space of (rational) tropical stable maps to $\mathbb{R}^2$]\label{definition:moduli_stable_maps}
An \textit{$n$-pointed tropical stable map of degree $\Delta$ to $\mathbb{R}^2$} (alternatively: \textit{tropical curve with $n$ points}) is a tuple $(\Gamma,x_1,\dots,x_N,h)$, where $(\Gamma,x_1,\dots,x_N)$ is an $N$-marked tropical curve (with $N=|\Delta|+n$ and $x_{n+1},\dots,x_{N}$ the labels given by $\Delta$) and $h:\Gamma\to\mathbb{R}^2$ such that:
\begin{itemize}
\item[(a)]
Let $e\in\Gamma$ be an edge with length $l(e)\in [0,\infty]$, identify $e$ with $[0,l(e)]$ and denote the vertex of $e$ that is identified with $0\in [0,l(e)]=e$ by $V$. The map $h$ is integer affine linear, i.e. $h\mid_e:t\mapsto tv+a$ with $a\in\mathbb{R}^2$ and $v(e,V):=v\in\mathbb{Z}^2$, where $v(e,V)$ is called \textit{direction vector of $e$ at $V$} and the \textit{weight} of an edge (denoted by $\omega(e)$) is the $\gcd$ of the entries of $v(e,V)$. If $e=x_i\in\Gamma$ is an end, then $v(x_i)$ denotes the direction vector of $x_i$ pointing away from its one vertex it is adjacent to.
\item[(b)]
If $i>n$, then the direction vector $v(x_i)$ of an end labeled with $x_i$ is given by 
$$v(x_i):=v_{i-n},$$ where $v_{i-n}$ is defined by $\Delta$. If $i\leq n$, then the direction vector of the end labeled with $x_i$ is zero. Ends with direction vector zero are called \textit{contracted ends} or \textit{points}.
\item[(c)]
The \textit{balancing condition}
\begin{align*}
\sum_{\substack{e\in\Gamma\textrm{ an edge}, \\ V \textrm{ vertex of }e}}v(e,V)=0
\end{align*}
holds for every vertex $V\in\Gamma$.
\end{itemize}
Two $n$-pointed tropical stable maps of degree $\Delta$, namely $(\Gamma ,x_1,\dots,x_N,h)$ and $(\Gamma' ,x_1', \dots, x_N',h')$, are isomorphic if there is an isomorphism $\varphi$ of their underlying $N$-marked tropical curves such that $h'\circ\varphi=h$.

\noindent The set $\mathcal{M}_{0,n}\left(\mathbb{R}^2,\Delta\right)$ of all $n$-pointed tropical stable maps of degree $\Delta$ up to isomorphism is called \textit{moduli space of $n$-pointed tropical stable maps of degree $\Delta$}.
\end{definition}

\begin{remark}[$\mathcal{M}_{0,n}\left(\mathbb{R}^2,\Delta\right)$ is a fan]\label{remark:identification:stable_maps_abstract_maps}
The map
\begin{align*}
\mathcal{M}_{0,n}\left(\mathbb{R}^2,\Delta\right) &\to \mathcal{M}_{0,N}\times\mathbb{R}^2\\
(\Gamma,x_1,\dots,x_N,h) &\mapsto \left(\left(\Gamma,x_1,\dots,x_N\right),h(x_1)\right)
\end{align*}
with $N=|\Delta|+n$ is bijective and $\mathcal{M}_{0,n}\left(\mathbb{R}^2,\Delta\right)$ is a tropical fan of dimension $|\Delta|-1$, see proposition 4.7 of \cite{GathmannKerberMarkwig}.
\end{remark}

\begin{definition}[Evaluation maps]
For $i=1,\dots,n$, the map
\begin{align*}
\operatorname{ev}_i:\mathcal{M}_{0,n}\left(\mathbb{R}^2,\Delta\right) &\to \mathbb{R}^2\\
(\Gamma,x_1,\dots,x_N,h) &\mapsto h(x_i)
\end{align*}
is called \textit{$i$-th evaluation map}. Under the identification from remark \ref{remark:identification:stable_maps_abstract_maps} the $i$-th evaluation map is a morphism of fans $\operatorname{ev}_i:\mathcal{M}_{0,N}\times\mathbb{R}^2 \to \mathbb{R}^2$, see proposition 4.8 of \cite{GathmannKerberMarkwig}.
\end{definition}

\begin{example}[Pull-back of a point]\label{example:pull_back_point}
A point $p=(p_1,p_2)\in\mathbb{R}^2$ is an intersection product of two rational functions, e.g.\
\begin{align*}
p=\operatorname{max}\{p_1,x\}\cdot\operatorname{max}\{p_2,y\}\cdot\mathbb{R}^2,
\end{align*}
 where $x,y$ are the coordinates in $\mathbb{R}^2$. The pull-back of the point $p$ under $\ev_i$ is defined to be
 $$\ev_i^*(p):= \ev_i^*(\max\{p_1,x\})\cdot \ev_i^*(\max\{p_2,y\})\cdot \mathcal{M}_{0,n}\left(\mathbb{R}^2,\Delta\right).
 $$
\end{example}

\begin{definition}[Forgetful maps]
For $n\geq4$ the map
\begin{align*}
\operatorname{ft}:\mathcal{M}_{0,n}&\to\mathcal{M}_{0,n-1}\\
(\Gamma,x_1,\dots,x_n) &\mapsto (\Gamma',x_1,\dots,x_{n-1})
\end{align*}
where $\Gamma'$ is the stabilization (straighten $2$-valent vertices) of $\Gamma$ after removing its end marked by $x_n$ is called the $n$-th \textit{forgetful map}. Applied recursively, it can be used to forget several ends with markings in $I^C\subset \{x_1,\ldots,x_n\}$, denoted by $\operatorname{ft}_I$, where $I^C$ is the complement of $I\subset \{x_1,\ldots,x_n\}$. With the identification from remark \ref{remark:identification:stable_maps_abstract_maps}, and additionally forgetting the map to the plane, we can also consider 
\begin{align*}
\operatorname{ft}_I:\mathcal{M}_{0,n}\left(\mathbb{R}^2,\Delta\right) &\to\mathcal{M}_{0,|I|}\\
(\Gamma,x_1,\dots,x_n,h) &\mapsto \operatorname{ft}_I(\Gamma,x_i|i\in I).
\end{align*}
\end{definition}
Any forgetful map is a morphism of fans.

\subsection*{Correspondence theorem}
The correspondence theorem of \cite{IlyaCRC} we use states that the number of classical curves satisfying point and cross-ratio conditions and the number of tropical curves satisfying point and tropical cross-ratio conditions are equal. Since different classical curves may tropicalize to the same tropical curve, each tropical curve has to be counted with a \textit{multiplicity}. We recall the definition of these multiplicities. For that we stick to the notation used in \cite{IlyaCRC}, for more details see (4.1) of \cite{IlyaCRC}.

\begin{definition}[Cross-ratios defined by \cite{MikhalkinCRC,IlyaCRC}]\label{definition:cross-ratios_Ilya}
Let $(\Gamma,x_1,\dots,x_N,h)\in \mathcal{M}_{0,n}\left(\mathbb{R}^2,\Delta\right)$. Let $\lbrace \beta_{i_1},\beta_{i_3}\rbrace$ and $\lbrace \beta_{i_2},\beta_{i_4}\rbrace$ be two sets of labels of ends of $\Gamma$ such that $\beta_{i_1},\dots,\beta_{i_4}$ are pairwise different. A bounded edge $\gamma$ of $\Gamma$ \textit{separates} $\beta_{i_1},\beta_{i_2}$ from $\beta_{i_3},\beta_{i_4}$ if $\beta_{i_1},\beta_{i_2}$ belong to one of the two connected components of $\Gamma\backslash \lbrace \gamma\rbrace$ and $\beta_{i_3},\beta_{i_4}$ to another.

The \textit{(tropical) cross-ratio} $\lambda'_i$ of $\lbrace \beta_{i_1},\beta_{i_2}\rbrace$ and $\lbrace \beta_{i_3},\beta_{i_4}\rbrace$ is given by
\begin{align*}
\lambda'_i:=\sum_\gamma \epsilon(\gamma,i)|\gamma|,
\end{align*}
where the sum goes over all bounded edges of $\Gamma$ and $|\gamma|$ is the length of a bounded edge and
\begin{align*}
\epsilon(\gamma,i):=
\begin{cases}
      1, & \text{if $\gamma$ separates the ends $\beta_{i_1},\beta_{i_2}$ from $\beta_{i_3},\beta_{i_4}$,} \\
      -1, & \text{if $\gamma$ separates the ends $\beta_{i_1},\beta_{i_4}$ from $\beta_{i_2},\beta_{i_3}$,} \\
      0, & \text{otherwise.}
\end{cases}
\end{align*}
\end{definition}

\begin{remark}[Cross-ratios and tropicalizations]\label{remark:tropicalized_cross-ratios}
Note that tropical cross-ratios are indeed tropicalizations of classical cross-ratios (see lemma 3.1 of \cite{IlyaCRC}), i.e. given a classical curve that satisfies a classical cross-ratio, then its tropicalization satisfies a tropical cross-ratio which is given by applying the valuation map to the classical cross-ratio. 
\end{remark}

\begin{definition}[Multiplicities]\label{Definition:Ilyas_multiplicities}
Let $C=(\Gamma,x_1,\dots,x_N,h)$ be a tropical curve that satisfies given point conditions $p_1,\dots,p_n$ and tropical cross-ratios $\lambda'_1,\dots,\lambda'_l$.

Let $x_1$ be the end of $\Gamma$ that is contracted to $p_1$ under $h$. We refer to the vertex adjacent to $x_1$ in $\Gamma$ as \textit{root vertex} and orient all edges of $\Gamma$ away from the root vertex. The \textit{head} of a bounded edge $\gamma$ is denoted by $\mathfrak{h}(\gamma)$ and its \textit{tail} by $\mathfrak{t}(\gamma)$. Let $V(\Gamma)$ be the set of vertices of $\Gamma$ and let $E^b(\Gamma)$ be the set of bounded edges of $\Gamma$. We refer to a vertex of $\Gamma$ as $v$ and to a bounded edge of $\Gamma$ as $\gamma$ for now. The vertices adjacent to ends $x_1,\dots,x_N$ are denoted by $v_1,\dots,v_N$ and do not need to be different. Define the complex

\begin{align}\label{eq:complex_Ilya}
\theta:\underbrace{\bigoplus_{v\in V(\Gamma)}\mathbb{Z}^2\oplus\bigoplus_{\gamma\in E^b(\Gamma)} \mathbb{Z}}_{M_1} \overset{B}{\longrightarrow} \underbrace{\bigoplus_{\gamma\in E^b(\Gamma)}\mathbb{Z}^2\oplus \bigoplus_{i=1}^{n}\mathbb{Z}^2 \oplus\bigoplus_{j=1}^{l}\mathbb{Z}}_{M_2} 
\end{align}

\noindent given by the maps (that are defined copywise)

\begin{align*}
1_\gamma\mapsto n_\gamma+\sum_{i=1}^{l}\epsilon(\gamma,i) \textrm{\quad and \quad}
a_v\mapsto \sum_\gamma \tilde{\epsilon}(\gamma,v)a_v+\sum_{i=1}^n\delta(v,v_i)a_v,
\end{align*}
where $a_v$ is the coordinate vector of $h(v)$ and where (see definition \ref{definition:moduli_stable_maps} for the notation of $v(\gamma,\mathfrak{t}(\gamma))$)
\begin{align*}
\tilde{\epsilon}(\gamma,v):=
\begin{cases}
      1, & \text{if}\ v=\mathfrak{t}(\gamma) \\
      -1, & \text{if}\ v=\mathfrak{h}(\gamma) \\
      0, & \text{otherwise}
\end{cases}
\textrm{\quad and \quad}
n_\gamma:=v(\gamma,\mathfrak{t}(\gamma))
\end{align*}
\noindent and
\begin{align*}
\delta(v,v_i):=
\begin{cases}
      1, & \text{if}\ v=v_i \\
      0, & \text{otherwise}.
\end{cases}
\end{align*}
\noindent Let $\theta_\mathbb{Z}$ be the map from above in the complex $\eqref{eq:complex_Ilya}\otimes_\mathbb{Z}\mathbb{Z}$. Finally, we can define the \textit{multiplicity of $C$}
\begin{align*}
m_\mathbb{C}(\Gamma,h):=|\coker\theta_\mathbb{Z}|,
\end{align*}
which is equal to $|\det(B)|$.
\end{definition}

\begin{theorem}[Correspondence Theorem 5.1 of \cite{IlyaCRC}]\label{thm:correspondence_thm_CRC}
Let $\Sigma$ be a $2$-dimensional lattice polytope and let $X_\Sigma$ be its toric variety. Let $q_1,\dots,q_n$ be points in $X_\Sigma$ and let $\mu_1,\dots,\mu_l$ be (classical) cross-ratios. Let these conditions be in general position such that there is only a finite number of rational curves in $X_\Sigma$ that fulfill them. Denote this number by $N^{\textrm{class}}_{0,n}\left( \mu_1,\dots,\mu_l \right)$. Let $p_1,\dots,p_n,\lambda'_1,\dots,\lambda'_l$ be the tropicalizations (see remark \ref{remark:tropicalized_cross-ratios}) of the conditions above. Then
\begin{align*}
N^{\textrm{class}}_{0,n}\left( \mu_1,\dots,\mu_l \right)
=
N_{0,n}\left(\lambda'_1,\dots,\lambda'_l\right)
\end{align*}
holds, where $N_{0,n}\left(\lambda'_1,\dots,\lambda'_l\right)$ is the number of rational tropical curves of degree $\Delta(\Sigma)$ that satisfy the conditions $p_1,\dots,p_n,\lambda'_1,\dots,\lambda'_l$.
\end{theorem}

\begin{example}\label{example:Kontsevich_CR}
When going through the (tropical) proof of Kontsevich's formula \cite{KontsevichPaper}, we can see that it allows us to determine the number of unlabeled tropical curves of degree $\Delta_d$ satisfying point conditions and exactly one cross-ratio constraint which involves exactly two points. In this case \textit{unlabeled} means that non-contracted edges not involved in any cross-ratio condition are not equipped with a label.

In case of $d=3$, Kontsevich's formula yields $40$ unlabeled curves (counted with multiplicity). Moreover, the proof of Kontsevich's formula allows us to actually draw these tropical curves. Hence we can determine the number of labeled curves by putting labels on ends, which yields $1440$ labeled curves.
\end{example}

\section{Tropical cross-ratios}\label{section:tropical_cross-ratios}

In this section we introduce tropical cross-ratios and their degenerations from an intersection theoretic point of view. Given a tropical curve that satisfies degenerated cross-ratios, we express its multiplicity locally.

\begin{definition}[Cross-ratios]\label{Definition:tropical_cross-ratios}
A \textit{(tropical) cross-ratio} $\lambda'$ is an unordered pair of pairs of unordered numbers $\left(\beta_1\beta_2|\beta_3\beta_4\right)$ together with an element in $\mathbb{R}_{>0}$ denoted by $|\lambda'|$, where $\beta_1,\dots,\beta_4$ are pairwise distinct ends of a tropical curve of $\mathcal{M}_{0,n}\left(\mathbb{R}^2,\Delta \right)$. We say that $C\in\mathcal{M}_{0,n}\left(\mathbb{R}^2,\Delta \right)$ satisfies the cross-ratio constraint $\lambda'$ if $C\in\ft^*_{\lambda'}\left(|\lambda'| \right)\cdot \mathcal{M}_{0,n}\left(\mathbb{R}^2,\Delta \right)$, where $|\lambda'|$ is the canonical local coordinate of the ray $\left(\beta_1\beta_2|\beta_3\beta_4\right)$ in $\mathcal{M}_{0,4}$.
\end{definition}

\begin{remark}
The definition of tropical cross-ratios given above generalizes the one given by Mikhalkin and Tyomkin since we can find a suitable projektion $\pi:\mathcal{M}_{0,4}\to\mathbb{R}$ shrinking on ray to zero, sending another one to $\mathbb{R}_{>0}$ and the last one to $\mathbb{R}_{<0}$ such that $\pi\circ\ft_{\lambda'}$ coincides with definition \ref{definition:cross-ratios_Ilya}. In particular, theorem \ref{thm:correspondence_thm_CRC} holds for our notion of tropical cross-ratios.
\end{remark}

\begin{definition}[General position I]\label{definition:general_position_I}
Let $p_1,\dots,p_n$ be points in $\mathbb{R}^2$ and $\lambda'_1,\dots,\lambda'_l$ be cross-ratios that have pairwise distinct pairs of unordered numbers. These conditions are in \textit{general position} if $\prod_{j=1}^{l}\ft_{\lambda'_j}^*\left( |\lambda'_j|\right)\cdot\prod_{i=1}^n\ev_i^*\left( p_i\right)\cdot\mathcal{M}_{0,n}\left(\mathbb{R}^2,\Delta \right)$ is a nonempty finite set that is contained in the union of the interiors of top-dimensional polyhedra of $\mathcal{M}_{0,n}\left(\mathbb{R}^2,\Delta \right)$ and $n+l=|\Delta|-1$. We say that $p_1,\dots,p_{n'},\lambda_1,\dots,\lambda_{l'}$ with $n'+l'<|\Delta|-1$ are in general position if there are $p_{n'+1},\dots,p_n,\lambda'_{l'+1},\dots,\lambda'_l$ such that $n+l=|\Delta|-1$ and $p_1,\dots,p_n,\lambda'_1,\dots,\lambda'_l$ are in general position. If $p_1,\dots,p_n,\lambda'_1,\dots,\lambda'_l$ with $n+l=|\Delta|-1$ are in general position, we define
\begin{align}\label{eq:general_pos_1}
N_{0,n}\left(\lambda'_1,\dots,\lambda'_l\right):=\degree\left(\prod_{j=1}^{l}\ft_{\lambda'_j}^*\left( |\lambda'_j|\right)\cdot\prod_{i=1}^n\ev_i^*\left( p_i\right)\cdot\mathcal{M}_{0,n}\left(\mathbb{R}^2,\Delta \right)\right),
\end{align}
the {\it number of rational tropical curves of degree $\Delta$ satisfying the point conditions $p_i$ and the cross-ratio conditions $\lambda'_i$}. Denote by $\mathcal{C}_{0,n}\left( \lambda'_1,\dots,\lambda'_l\right)$ the set of tropical curves contributing to $N_{0,n}\left(\lambda'_1,\dots,\lambda'_l\right)$.
\end{definition}

\begin{remark}\label{Remark:N_(0,n)independent of positions/lengths}
The numbers $N_{0,n}\left(\lambda'_1,\dots,\lambda'_l\right)$ are independent of the exact positions of the points since two sets of $n$ points are rationally equivalent and so their pull-backs are rationally equivalent leading to the same degree (see remark \ref{remark:facts_about_rational_equivalence}). Notice also that all points in $\mathcal{M}_{0,4}$ are rationally equivalent using remark \ref{remark:facts_about_rational_equivalence} since $\mathcal{M}_{0,4}$ can be embedded (cf. Figure \ref{Example_M_0_4}) by a morphism into $\mathbb{R}^2$ and all points of $\mathbb{R}^2$ are rationally equivalent. Hence the numbers $N_{0,n}\left(\lambda'_1,\dots,\lambda'_l\right)$ are independent of the lengths $|\lambda'_i|$ of the cross-ratios. In particular, the lengths can be zero. This observation is crucial and is used extensively later.
\end{remark}

Note that the intersection theoretic definition of tropical cross-ratios automatically assigns a multiplicity to each tropical curve satisfying given point conditions and cross-ratio constraints. In our case, lemma 1.2.9 of \cite{Rau} states that the intersection theoretic multiplicity of a tropical curve $C$ is the absolute value of the determinant of the so called $\ev$-$\ft$-matrix which is given by the locally (around $C$) linear maps $\ev:\mathcal{M}_{0,n}\left(\mathbb{R}^2,\Delta \right)\to\mathbb{R}^{2n}$ and $\ft:\mathcal{M}_{0,n}\left(\mathbb{R}^2,\Delta \right)\to\mathcal{M}_{0,4}$, where the coordinates on $\mathcal{M}_{0,n}\left(\mathbb{R}^2,\Delta \right)$ and $\mathcal{M}_{0,4}$ are the bounded edges' lengths.

Often, tropical intersection theory yields multiplicities needed for correspondence theorems, which enables us to count tropical curves by means of tropical intersection theory on tropical moduli spaces. The same holds true for the counts of curves satisfying cross-ratio conditions we consider here. We prove this in the following proposition, using methods well-known to the experts in the area.

\begin{proposition}
Let $C$ be a tropical curve contributing to \eqref{eq:general_pos_1}. The intersection theoretic multiplicity of $C$ coincides with $m_\mathbb{C}(\Gamma,h)$ defined in \ref{Definition:Ilyas_multiplicities}.
\end{proposition}

\begin{proof}
Let $C=(\Gamma,x_1,\dots,x_N,h)$ be a tropical curve that contributes to $N_{0,n}\left(\lambda_1,\dots,\lambda_l\right)$.
In terms of tropical intersection theory the multiplicity of $C$ is given by $|\det(A)|$, where $A$ is the $\ev$-$\ft$-matrix that is given by the (around $C$) linear maps $\ev,\ft$ and the lengths of the edges as coordinates on the moduli space. We want to sketch how to prove that $|\det(A)|$ and $|\det(B)|$ (from definition \ref{Definition:Ilyas_multiplicities}) are equal. For that, we start with the following complex

\begin{align*}
\underbrace{\mathbb{Z}^2\oplus\bigoplus_{\gamma\in E^b(\Gamma)} \mathbb{Z}}_{N_1} \overset{A}{\longrightarrow} \underbrace{\bigoplus_{i=1}^n \mathbb{Z}^2\oplus \bigoplus_{j=1}^l \mathbb{Z}}_{N_2},
\end{align*}
where the first summand on the left belongs to the root vertex defined in \ref{Definition:Ilyas_multiplicities}. There are maps between the complex above and the complex \eqref{eq:complex_Ilya} in the following way: Let $\alpha_2:N_2\to M_2$ be the canonical embedding and let 
\begin{align*}
\alpha_1:N_1\to M_1,\ (a,\underline{e})\mapsto (a,a+\sum\pm e_i u_{e_i},\underline{e})
\end{align*}
be a map where $a$ is the coordinate of the root vertex, $e_i$ is the length of the edge $\gamma_i$ and $u_{e_i}$ is the primitive direction vector of $\gamma_i$. Moreover, we choose $a+\sum\pm e_i u_{e_i}$ in such a way that it is the shortest path between the root vertex and the vertex associated to the $j$-th contracted end depending on which entry of the vector in the image we are considering (the choice of $\pm$ should be consistent with the orientation on $\Gamma$). Note that $\alpha_1,\alpha_2$ are both injective and that the diagram given by the maps $A,B,\alpha_1,\alpha_2$ commutes. This commutative diagram extends to the commutative diagram shown below. By definition
\begin{align*}
\coker \alpha_1 \cong \left(\mathbb{Z}^2\right)^{|V(\Gamma)|-1} \textrm{\quad and \quad} \coker \alpha_2 =\left(\mathbb{Z}^2\right)^{|E^b(\Gamma)|}.
\end{align*}
Considering the definitions of $B,\zeta_2$, we can see that $\zeta_2\circ B$ is surjective. Hence $C$ is surjective. Since $C$ is a surjective morphism of free module of the same rank it is an isomorphism. Therefore $\coker \alpha_3$ vanishes which guarantees that $\alpha_3$ is surjective. The map $\partial$ which we obtain from applying the snake lemma yields that $G$ vanishes. Therefore $\alpha_3$ is an isomorphism. Thus
\begin{align*}
|\det(A)| = |\det(B)|
\end{align*}
follows.

\begin{table}[H]
\begin{tikzpicture}
\matrix[matrix of math nodes,column sep={60pt,between origins},row
sep={60pt,between origins}, nodes={asymmetrical rectangle}] (s)
{
&|[name=ka]| 0 &|[name=kb]| 0 &|[name=kc]| G \\
|[name=03]|0&|[name=A]| N_1 &|[name=B]| N_2 &|[name=C]| \coker A &|[name=01]| 0 \\
|[name=02]| 0 &|[name=A']| M_1 &|[name=B']| M_2 &|[name=C']| \coker B &|[name=04]| 0 \\
&|[name=ca]| \coker \alpha_1 &|[name=cb]| \coker \alpha_2 &|[name=cc]| \coker \alpha_3 \\
};
\draw[->] (ka) edge (A)
          (kb) edge (B)
          (03) edge (A)
          (C') edge (04)
          (kc) edge (C)
          (A) edge node[auto] {\(A\)} (B)
          (B) edge (C)
          (C) edge (01)
          (A) edge node[auto] {\(\alpha_1\)} (A')
          (B) edge node[auto] {\(\alpha_2\)} (B')
          (C) edge node[auto] {\(\alpha_3\)} (C')
          (02) edge (A')
          (A') edge node[auto] {\(B\)} (B')
          (B') edge (C')
          (A') edge node[auto] {\(\zeta_1\)} (ca)
          (B') edge node[auto] {\(\zeta_2\)} (cb)
          (C') edge (cc)
		  (ka) edge (kb)
          (kb) edge (kc)
          (ca) edge node[auto] {\(C\)} (cb)
          (cb) edge (cc)
 (kc) -| node[auto,text=black,pos=.7]
{\(\partial\)} ($(01.east)+(.5,0)$) |- ($(B)!.35!(B')$) -|
($(02.west)+(-.5,0)$) |- (ca);
\end{tikzpicture}
\end{table}
\end{proof}

The strength of our intersection theoretic definition of tropical cross-ratios is that it allows us to \textit{degenerate} tropical cross-ratios easily. For that note that from an intersection theoretic point of view it does not matter if we pull-back $0\in\mathcal{M}_{0,4}$ instead of a nonzero point.

\begin{definition}[Cross-ratios with $|\lambda|=0$]
A \textit{(tropical) cross-ratio $\lambda$ with $|\lambda|=0$} is defined as a set $\lbrace \beta_1,\dots,\beta_4\rbrace$, where $\beta_1,\dots,\beta_4$ are pairwise distinct ends of a tropical curve $\mathcal{M}_{0,n}\left(\mathbb{R}^2,\Delta \right)$. We say that $C\in\mathcal{M}_{0,n}\left(\mathbb{R}^2,\Delta \right)$ satisfies the cross-ratio constraint $\lambda$ (with $|\lambda|=0$) if $C\in\ft^*_\lambda\left(0 \right)\cdot \mathcal{M}_{0,n}\left(\mathbb{R}^2,\Delta \right)$.

Another way to think about a cross-ratio $\lambda$ with $|\lambda|=0$ is that $\lambda$ is the \textit{degeneration} of cross-ratios $\lambda'_j, j\in\mathbb{N}$ which have the same pairs of unordered numbers and $|\lambda'_j|\to 0$ for $j\to \infty$, where the pairs become a set in the limit. Because of remark \ref{Remark:N_(0,n)independent of positions/lengths} it makes sense to refer to $\lambda$ as the degeneration of $\lambda'_j$ for some $j\in\mathbb{N}$.
\end{definition}

\begin{definition}[General position II]\label{definition:general_position_II}
Let $\lambda_1,\dots,\lambda_{l'}$ be cross-ratios with $|\lambda_j|=0$ for $j=1,\dots,l'$ (i.e. \textit{degenerated} cross-ratios). These cross-ratios are in \textit{general position} if there are general positioned cross-ratios $\lambda'_1,\dots,\lambda'_{l'}$ such that $\lambda_j$ is the degeneration of $\lambda'_j$ for $j=1,\dots,l'$. More precisely, points $p_1,\dots,p_n$ in $\mathbb{R}^2$, cross-ratios $\lambda_1,\dots,\lambda_{l'},\lambda_{l'+1},\dots,\lambda_l$ with $|\lambda_j|=0$ for $j=1,\dots,l'$ and $|\lambda_j|>0$ otherwise are in general position if $p_1,\dots,p_n,\lambda'_1,\dots,\lambda'_{l'},\lambda_{l'+1},\dots,\lambda_l$ are in general position where $\lambda_j$ is the degeneration of $\lambda'_j$ for $j=1,\dots,l'$.
\end{definition}

\begin{notation}\label{notation}
We want to fix the following conventions. If we mention a set of conditions, then we assume that these conditions are in general position and that the cross-ratio constraints are totally ordered by their lengths, i.e. $|\lambda_1|>|\lambda_2|>\dots$. Point conditions are always denoted by $p_1,\dots,p_n$. Cross-ratios are denoted by $\lambda_i'$, where we have $l'$ of these cross-ratios if the intersection defined by the conditions $p_1,\dots,p_n,\lambda_1',\dots,\lambda_{l'}'$ is not a $0$-dimensional cycle, and we have $l$ cross-ratios if the intersection defined by the conditions $p_1,\dots,p_n,\lambda_1',\dots,\lambda_{l}'$ is $0$-dimensional. If we write $\lambda_i$, then $\lambda_i$ is the degeneration of $\lambda_i'$. It may also happen that we need classical (i.e. non-tropical) cross-ratios. A classical cross-ratio is denoted by $\mu_i$ and its tropical counterpart obtained from applying the valuation map of the ground field is denoted by $\lambda_i'$.
\end{notation}

Our next aim is to describe the multiplicity of a curve that satisfies point conditions and degenerated cross-ratio conditions. For that we observe that degenerating a cross-ratio means to shrink an edge, i.e. degenerating the tropical curve satisfying it as well. Therefore the multiplicity of such a degenerated tropical curve $C$ can be described in terms of the number of tropical curves degenerating to $C$.

\begin{definition}[Resolving vertices w.r.t. a cross-ratio with $|\lambda|=0$]\label{definition:lambda_v_vertex_trop_curve}
The combinatorial type of a polyhedron $\tau\subset \mathcal{M}_{0,n}\left(\mathbb{R}^2,\Delta \right)$ (resp.\ $\mathcal{M}_{0,m}$) is denoted by $\mathfrak{c}(\tau)$. Let $\lambda_1,\dots,\lambda_{l'}$ be degenerated cross-ratios and let $\tau\subset \mathcal{M}_{0,n}\left(\mathbb{R}^2,\Delta \right)$ be some polyhedron. The set $\lambda_v$ of cross-ratios associated to a vertex $v$ of $\mathfrak{c}(\tau)$ consists of the cross-ratios $\lambda_j$ such that the image of $v$ under $\ft_{\lambda_j}$ is $4$-valent. If $$\val(v)=3+\#\lambda_v$$ holds, then we say that $v$ is \textit{resolved according to $\lambda'_i$} (we use notation \ref{notation}) if we replace $v$ by two vertices $v_1,v_2$ that are connected by a new edge such that
\begin{align*}
\lambda_v=\lbrace\lambda_i\rbrace\cup\lambda_{v_1}\cup\lambda_{v_2}
\end{align*}
is a union of pairwise disjoint sets and
\begin{align*}
\val(v_k)=3+\#\lambda_{v_k}
\end{align*}
holds for $k=1,2$.
\end{definition}

\begin{example}
In this example we want to point out that resolving a vertex according to a cross-ratio is not unique. It is neither unique in the sense (A) that the edges adjacent to $v_1,v_2$ are uniquely determined nor in the (weaker) sense (B) that the $\lambda_{v_i}$ are uniquely determined.

Let $\tau$ be the $0$-dimensional cell of $\mathcal{M}_{0,6}$, that is $\mathfrak{c}(\tau)$ has only one vertex $v$ to which all ends are adjacent to. We choose the following cross-ratios:
\begin{align*}
\lambda_1=\lbrace 1,2,3,4 \rbrace,\quad &\lambda'_1=(12|34)\\
\lambda_2=\lbrace 3,4,5,6 \rbrace,\quad &\lambda'_2=(34|56)\\
\lambda_3=\lbrace 1,2,5,6 \rbrace,\quad &\lambda'_3=(12|56)\\
\end{align*}
\begin{itemize}
\item[(A)]
If we resolve $v$ according to $\lambda'_3$, we have at least two choices shown in the Figure below.
\begin{figure}[H]
\centering
\def\svgwidth{240pt}
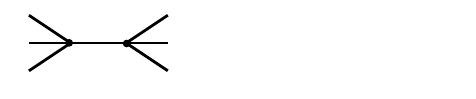
\label{Example_resolving_not_unique}
\end{figure}
\item[(B)]
If we choose another $\lambda'_3$, namely $\lambda'_3=(15|26)$, we also have at least two choices shown in the Figure below.
\begin{figure}[H]
\centering
\def\svgwidth{240pt}
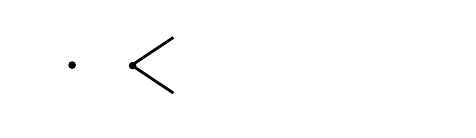
\label{Example_resolving_not_unique_2}
\end{figure}
\end{itemize}
\end{example}

\begin{lemma}\label{lemma:description_of_X}
For notation, see \ref{notation}. The intersection product $X:=\prod_{j=1}^{l'}\ft_{\lambda_j}^*\left( 0\right)\cdot\mathcal{M}_{0,n}\left(\mathbb{R}^2,\Delta \right)$ lies in $\mathcal{M}_{0,n}\left(\mathbb{R}^2,\Delta \right)^{(l')}$ and its top-dimensional polyhedra are top-dimensional polyhedra $\tau$ of $\mathcal{M}_{0,n}\left(\mathbb{R}^2,\Delta \right)^{(l')}$ such that for all vertices $v$ of $\mathfrak{c}(\tau)$
\begin{align*}
\val(v)=3+\#\lambda_v
\end{align*}  
holds and the weight of a top-dimensional polyhedron $\tau$ of $X$ is given recursively by
\begin{align*}
\omega(\tau)=
\begin{cases}
1 & \textrm{, if $l'=1$}\\
\sum_\sigma \omega(\sigma) & \textrm{, otherwise}
\end{cases}
\end{align*}
where the sum runs over all top-dimensional polyhedra of $\prod_{j=2}^{l'}\ft_{\lambda_j}^*\left( 0\right)\cdot\mathcal{M}_{0,n}\left(\mathbb{R}^2,\Delta \right)$ such that $\mathfrak{c}(\sigma)$ is given by resolving the vertex $v\in\mathfrak{c}(\tau)$, that is defined by $\lambda_1\in\lambda_v$, according to $\lambda'_1$. In particular, all weights of $X$ are non-negative.
\end{lemma}

Note that the intersection product $X$ above does not depend on $\lambda'_1,\dots,\lambda'_{l'}$. We consider $X$ up to rational equivalence. We use $\lambda'_1,\dots,\lambda'_{l'}$ to describe a representative of $X$ under this equivalence relation.
\begin{proof}
Let $\lambda'_1,\dots,\lambda'_{l'}$ be cross-ratios such that $\lambda_j$ is the degeneration of $\lambda'_j$ for $j=1,\dots,l'$. The pull-back of $0$ along $\ft_{\lambda_j}$ is given by a Cartier divisior $\ft_{\lambda_j}\left(\max(\star,\star,0)\right)$ (see  example \ref{example:pull_back_0_M_0,4}), where $\max(\star,\star,0):\left(x,y\right)\mapsto\max(x,y,0)$ is a Cartier divisor on $\mathcal{M}_{0,4}\subset \mathbb{R}^2$ (see figure \ref{Example_M_0_4}). Note that $\ft_{\lambda_j}\left(\max(\star,\star,0)\right)$ is a linear function on every cell of $\mathcal{M}_{0,n}\left(\mathbb{R}^2,\Delta \right)$ for $j=1,\dots,l'$. Therefore no refinement of $\prod_{j\neq i}\ft_{\lambda_j}^*\left( 0\right)\cdot\mathcal{M}_{0,n}\left(\mathbb{R}^2,\Delta \right)$ is necessary when intersecting with some $\ft_{\lambda_i}^*\left( 0\right)$. Hence $X$ lies in the codimension-$l'$-skeleton of $\mathcal{M}_{0,n}\left(\mathbb{R}^2,\Delta \right)$. Moreover, every intersection with a Cartier divisor lowers the dimension by one, so the dimension of $X$ is exactly the dimension of top-dimensional cells of the codimension-$l'$-skeleton of $\mathcal{M}_{0,n}\left(\mathbb{R}^2,\Delta \right)$.

To prove the last part of the lemma, we set $m=n+|\Delta|$ and identify
\begin{align*}
\mathcal{M}_{0,n}\left(\mathbb{R}^2,\Delta \right)\cong\mathcal{M}_{0,m}\times\mathbb{R}^2
\end{align*}
as in remark \ref{remark:identification:stable_maps_abstract_maps} such that it is sufficient to prove the statements for $\mathcal{M}_{0,m}$ because cross-ratio constraints only fix a tropical curve up to translation in $\mathbb{R}^2$. To do so, we use induction on the number of cross-ratio constraints. Let $m\in\mathbb{N}_{>3}$.

We start with one cross-ratio $\lambda_1=\lbrace\beta_1,\dots,\beta_4\rbrace$ with $|\lambda_1|=0$. Obviously, a top-dimensional polyhedron $\tau$ of $\ft^*_{\lambda_1}\left( 0 \right)\cdot\mathcal{M}_{0,m}$ is a top-dimensional polyhedron $\mathcal{M}_{0,m}^{(1)}$ such that $\val(v)=3+\#\lambda_v$ holds for the only $4$-valent vertex $v$ of $\mathfrak{c}(\tau)$ since $\#\lambda_v=\#\lbrace\lambda_1\rbrace=1$. Note that the three resolutions of $v$ correspond to three top-dimensional polyhedra $\sigma_{\beta_1\beta_2},\sigma_{\beta_1\beta_3},\sigma_{\beta_1\beta_4}$ of $\mathcal{M}_{0,m}$. On two of these polyhedra the map $\ft_{\lambda_j}\left(\max(\star,\star,0)\right)$ is the zero function and on one of that polyhedra it maps each point to the length of the edge that was obtained from resolving the vertex $v$. Which of the $\sigma_{\beta_1\beta_2},\sigma_{\beta_1\beta_3},\sigma_{\beta_1\beta_4}$ are mapped to zero depends on the choice of coordinates of $\mathcal{M}_{0,4}\subset\mathbb{R}^2$. Let $v_{\beta_1\beta_2}$ denote the direction vector in $\mathcal{M}_{0,m}$ associated to a tropical curve that has only one edge of length one that separates the ends $\beta_1,\beta_2$ from $\beta_3,\beta_4$ (see the following Figure) and define $v_{\beta_1\beta_3},v_{\beta_1\beta_4}$, respectively.

\begin{figure}[H]
\centering
\def\svgwidth{400pt}
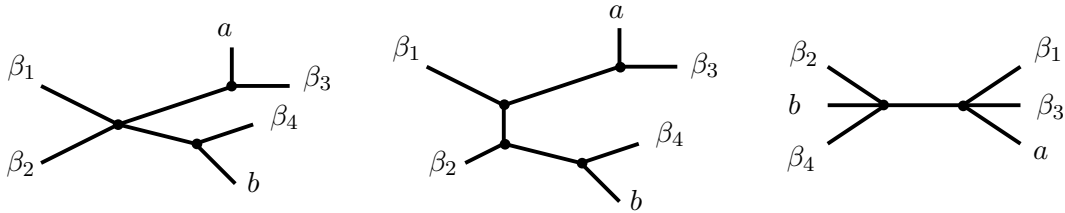
\caption{From left to right: an arbitrary $\tau$ with its $\sigma_{\beta_1\beta_3}$ and the curve associated to $v_{\beta_1\beta_3}$.}
\label{Proof_structure_of_X_notation}
\end{figure}

\noindent
We assume without loss of generality that $\sigma_{\beta_1\beta_2}$ is not mapped to zero under $\ft_{\lambda_j}\left(\max(\star,\star,0)\right)$. Therefore $v_{\beta_1\beta_2}$ is mapped to $1$ under $\ft_{\lambda_j}\left(\max(\star,\star,0)\right)$ and $v_{\beta_1\beta_3},v_{\beta_1\beta_4}$ are mapped to zero. We write $\varphi:=\ft_{\lambda_j}\left(\max(\star,\star,0)\right)$. The weight $\omega_\varphi(\tau)$ is 
\begin{align*}
\omega_\varphi(\tau)=\sum_{\sigma=\sigma_{\beta_1\beta_2},\sigma_{\beta_1\beta_3},\sigma_{\beta_1\beta_4}}\varphi_\sigma\left(\omega(\sigma)\cdot v_{\sigma/\tau}\right)-
\varphi_{\tau}\left( \sum_{\sigma=\sigma_{\beta_1\beta_2},\sigma_{\beta_1\beta_3},\sigma_{\beta_1\beta_4}} \omega(\sigma)\cdot v_{\sigma/\tau} \right),
\end{align*}
where $\varphi_\sigma,\varphi_\tau$ denote the linear parts of $\varphi$ on $\sigma,\tau$, $\omega(\sigma)=1$ denotes the weight of $\sigma$ in $\mathcal{M}_{0,m}$ and $v_{\sigma/\tau}$ denotes an arbitrary representative of the normal vector $u_{\sigma/\tau}$. Moreover, $v_{\sigma_{\beta_1\beta_2}/\tau}=v_{\beta_1\beta_2}$ and $v_{\sigma_{\beta_1\beta_3}/\tau}$, $v_{\sigma_{\beta_1\beta_4}/\tau}$, respectively. Note that the second sum is in $\tau$ as $\mathcal{M}_{0,m}$ is balanced and because of $\tau\subset\ft^{-1}_{\lambda_1}\left( 0 \right)$ this second sum vanishes under $\varphi_\tau$. As discussed above only one summand of the first sum is nonzero, namely $\sigma=\sigma_{\beta_1\beta_2}$. Hence $\omega_\varphi(\tau)=1$.

Next, we will perform the induction step from $l'-1$ to $l'$. We denote the elements of $\lambda_1$ as above, that is $\lambda_1=\lbrace\beta_1,\dots,\beta_4\rbrace$ with $|\lambda_1|=0$. We use the fact that
\begin{align*}
\prod_{j=1}^{l'}\ft_{\lambda_j}^*\left( 0\right)\cdot\mathcal{M}_{0,m}
=
\ft_{\lambda_1}^*\left( 0\right)\cdot\left(\prod_{j=2}^{l'}\ft_{\lambda_j}^*\left( 0\right)\cdot\mathcal{M}_{0,m}\right)
\end{align*}
and that use the induction hypothesis for $\prod_{j=2}^{l'}\ft_{\lambda_j}^*\left( 0\right)\cdot\mathcal{M}_{0,m}$. A top-dimensional polyhedron $\tau$ of $\ft_{\lambda_1}^*\left( 0\right)\cdot\left(\prod_{j=2}^{l'}\ft_{\lambda_j}^*\left( 0\right)\cdot\mathcal{M}_{0,m}\right)$ is a top-dimensional polyhedron of $\mathcal{M}_{0,m}^{(l')}$ such that there is a vertex $v$ of $\mathfrak{c}(\tau)$ with $\lambda_1\in\lambda_v$. Since the interior of $\tau$ is in the codimension-1-boundary of $\prod_{j=2}^{l'}\ft_{\lambda_j}^*\left( 0\right)\cdot\mathcal{M}_{0,m}$ and the cross-ratio lengths are without loss of generality small, the vertex $v$ is obtained by shrinking an edge connecting two vertices $v_1,v_2$ in the combinatorial type of a top-dimensional polyhedron of $\prod_{j=2}^{l'}\ft_{\lambda_j}^*\left( 0\right)\cdot\mathcal{M}_{0,m}$ such that
\begin{align*}
\val(v)&=3+\# \lambda_{v_1}+3+\# \lambda_{v_2} -2\\
&=4+\#\left( \lambda_{v_1}\cup\lambda_{v_2}\right)\\
&=3+\#\left( \lambda_{v_1}\cup\lambda_{v_2}\cup\lbrace \lambda_1\rbrace\right)\\
&=3+\#\lambda_v.
\end{align*}
 Again there are three resolutions of $v$ and we choose the coordinates on $\mathcal{M}_{0,4}$ such that the top-dimensional polyhedra of $\prod_{j=2}^{l'}\ft_{\lambda_j}^*\left( 0\right)\cdot\mathcal{M}_{0,n}\left(\mathbb{R}^2,\Delta \right)$  given by resolving the vertex $v$ according to the pairs of unordered numbers of $\lambda'_1$ are not mapped to zero.  The weight $\omega_\varphi(\tau)$ is 
\begin{align*}
\omega_\varphi(\tau)=\sum_{\sigma}\varphi_\sigma\left(\omega(\sigma)\cdot v_{\sigma/\tau}\right)-
\varphi_{\tau}\left( \sum_{\sigma} \omega(\sigma)\cdot v_{\sigma/\tau} \right),
\end{align*}
where the sums run over all top-dimensional polyhedra of $\prod_{j=2}^{l'}\ft_{\lambda_j}^*\left( 0\right)\cdot\mathcal{M}_{0,n}\left(\mathbb{R}^2,\Delta \right)$ that have $\tau$ in their boundaries. Since $\prod_{j=2}^{l'}\ft_{\lambda_j}^*\left( 0\right)\cdot\mathcal{M}_{0,n}\left(\mathbb{R}^2,\Delta \right)$ is balanced, the second sum is in $\tau$ and vanishes. Moreover, the arguments above yield that $\varphi_\sigma\left(v_{\sigma/\tau}\right)$ is zero if and only if $v$ is not resolved according to $\lambda'_1$. By definition $\varphi_\sigma\left(v_{\sigma/\tau}\right)=1$ otherwise.
\end{proof}

\begin{definition}[Local description of the weights of $X$]\label{definition:resolution_weights_local}
Let $\tau$ be a top-dimensional polyhedron of $X$ (for notation, see lemma \ref{lemma:description_of_X}) of weight $\omega(\tau)$. Let $\mathfrak{c}(\tau)$ be the combinatorial type of $\tau$ such that $\mathfrak{c}(\tau)$ satisfies all given degenerated cross-ratios $\lambda_1,\dots,\lambda_l$. That is, the disjoint union over all $\lambda_v$ of $\mathfrak{c}(\tau)$ is exactly $\lambda_1,\dots,\lambda_l$ and each vertex $v$ of $\mathfrak{c}(\tau)$ satisfies $\val(v)=3+\#\lambda_v$. If $v\in\mathfrak{c}(\tau)$ is a vertex with $\val(v)>3$, then cut all adjacent bounded edges of $v$, stretch the remaining edges to infinity and denote the component that contains $v$ by $C_v$. If $\lambda=\lbrace \beta_1,\dots,\beta_4\rbrace\in\lambda_v$ is a given cross-ratio and $\beta_i$ is not adjacent to $v$ after cutting some bounded edges, then replace $\beta_i$ by the label of the edge adjacent to $v$ that is contained in the shortest path from $v$ to $\beta_i$ in $\mathfrak{c}(\tau)$. Let $\tilde{\lambda}_1,\dots,\tilde{\lambda}_r$ be the cross-ratios obtained this way such that $\lbrace \tilde{\lambda}_1,\dots,\tilde{\lambda}_r\rbrace=\lambda_v$ in $C_v$ and let $\Delta'$ be the degree associated to $C_v$. The component of $v$ is by definition the $0$-dimensional cell of $\prod_{j=1}^r \ft_{\mu_j}^*\left( 0\right)\cdot\mathcal{M}_{0,n}\left(\mathbb{R}^2,\Delta' \right)$. We call its weight the \textit{local weight of $v$} and denote it by $\omega_v(\tau)$. 
\end{definition}

Using the proof of lemma \ref{lemma:description_of_X}, we can deduce the following corollary:
\begin{corollary}\label{corollary:resolution_weights_local}
Under the same assumption as lemma \ref{lemma:description_of_X}, we have that
\begin{align*}
\omega(\tau)=\prod_v \omega_v(\tau),
\end{align*}
where the product runs over all vertices of $\mathfrak{c}(\tau)$ and $\omega_v(\tau)$ is the local weight of $v$.
\end{corollary}

Corollary \ref{corollary:resolution_weights_local} allows us to deduce the following:

\begin{lemma}\label{lemma:nice_property}
For notation, see \ref{notation}. Let $C$ be a point in the interior of a top-dimensional polyhedron $\tau$ of $X:=\prod_{j=1}^{l}\ft_{\lambda_j}^*\left( 0\right)\cdot\mathcal{M}_{0,n}\left(\mathbb{R}^2,\Delta \right)$ such that its multiplicity $\omega(\tau)$ is nonzero. Let $v\in C$ be a vertex of $C$ such that $\val(v)>3$. Then for every edge $e$ adjacent to $v$ in $C$ there is a $\beta_i$ in some $\lambda_j\in\lambda_v$ such that $e$ is in the shortest path from $v$ to $\beta_i$. 
\end{lemma}

\begin{proof}
We use the notation from definition \ref{definition:resolution_weights_local}: Let $C_v$ be the component of $v$ in $C$ and let $\mu_1,\dots,\mu_r$ be the cross-ratios associated to $v$ in $C_v$. Then $\val(v)=3+r$ by lemma \ref{lemma:description_of_X}. Denote the ends adjacent to $v$ by $e_1,\dots, e_{3+r}$ suppose that there is an end $e_i$ adjacent to $v$ in $C_v$ such that there is no $\mu_j$ with $e_i\in\mu_j$. Since the multiplicity of $\tau$ is nonzero, corollary \ref{corollary:resolution_weights_local} guarantees that there is a \textit{total resolution} of $v$, that is there is a tropical curve $C_v'$ and cross-ratios $\mu'_1,\dots,\mu'_r$ such that $C_v'$ is $3$-valent and $C_v'$ arises from resolving $\mu_1,\dots,\mu_r$ in $C_v$ according to $\mu_1',\dots,\mu_r'$. The end $e_i$ does not appear in any $\mu_j$ and therefore it does not appear in any $\mu_j'$ for $j=1,\dots,r$. Let $v_i$ be the vertex of $C_v'$ to which $e_i$ is adjacent to. Note that there is a bounded edge $b$ adjacent to $v_i$ that is shrunk first when degenerating $\mu_1',\dots,\mu_r'$ step by step. Therefore there is a cross-ratio $\mu_j'$ shrinking exactly $b$. Hence $e_i$ appears in $\mu_j'$ as $v_i$ is $3$-valent. This is a contradiction.
\end{proof}

\begin{remark}\label{corollary:mult_ev_well-def}
Let $\Delta$ be a degree. Let $p_1,\dots,p_n$ be points in $\mathbb{R}^2$ and let $\lambda'_1,\dots,\lambda'_{l'},\lambda_{l'+1},\dots,\lambda_l$ be cross-ratios such that $p_1,\dots,p_n,\lambda'_1,\dots,\lambda'_l,\lambda_{l'+1},\dots,\lambda_l$ are in general position and $n+l=|\Delta|-1$ holds. Let
$$X:=\prod_{k=1}^{l'}\ft_{\lambda_k}^*\left( 0\right)\cdot\prod_{j=l'+1}^{l}\ft_{\lambda_j}^*\left( |\lambda_j|\right)\cdot\mathcal{M}_{0,n}\left(\mathbb{R}^2,\Delta \right)$$ be an intersection product, where $\lambda_1,\dots,\lambda_{l'}$ are the degenerations of $\lambda'_1,\dots,\lambda'_{l'}$. Then, using general position, the curves $\prod_{i=1}^n\ev_i^*\left( p_i\right)\cdot X$ are in the interior of top-dimensional cells of $X$.
\end{remark}

\begin{proposition}\label{prop:mult_X}
Let $\Delta$ be a degree, let $p_1,\dots,p_n,\lambda_1,\dots,\lambda_{l'},\lambda'_{l'+1},\dots,\lambda'_l$ be conditions as in \ref{notation} such that $$n+l=|\Delta|-1$$ and let $$X:=\prod_{k=1}^{l'}\ft_{\lambda_k}^*\left( 0\right)\cdot\prod_{j=l'+1}^{l}\ft_{\lambda'_j}^*\left( |\lambda'_j|\right)\cdot\mathcal{M}_{0,n}\left(\mathbb{R}^2,\Delta \right).$$ Then the multiplicity $\mult(C)$ with which a curve $C$ in $\prod_{i=1}^n\ev_i^*\left( p_i\right)\cdot X$ contributes to the degree of this $0$-dimensional cycle is 
\begin{align*}
\mult(C)=\mult_{\ev}(C)\cdot\omega(\sigma_C),
\end{align*}
where $\omega(\sigma_C)$ is the weight of the top-dimensional cell $\sigma_C$ of $X$ that contains $C$ and $\mult_{\ev}(C)$ is the absolute value of the determinant of the locally (around $C$) linear map $\ev:X\to\mathbb{R}^{2n}$.
\end{proposition}

\begin{proof}
This follows from lemma \ref{lemma:description_of_X}, remark \ref{corollary:mult_ev_well-def} and lemma 1.2.9 of \cite{Rau}.
\end{proof}

Having expressed $\omega(\sigma_C)$ locally already (see corollary \ref{corollary:resolution_weights_local}), our next goal is to express $\mult_{\ev}(C)$ locally.

\begin{definition}[Free and fixed components]\label{definition:fixed_free_components}
Let $C$ be a rational tropical curve (possibly with vertices of higher valence) that is fixed by general positioned points $p_1,\dots,p_n$. Let $v$ be an $m$-valent vertex of $C$ such that there is no point lying on $v$ and denote adjacent edges of $v$ by $e_1,\dots,e_m$. Fix $i\in\lbrace 1,\dots,m\rbrace$, cut the edge $e_i$ and stretch it to infinity. Now there are two tropical curves, namely one that contains $v$ and one that does not. The tropical curve $C_i$ that does not contain $v$ is called a \textit{component of $v$}. A component of $v$ is called a \textit{fixed component of $v$} if it is fixed by the points on it (if this component is only a line, then this line is considered fixed if there is a point on it). Otherwise it is called a \textit{free component of $v$}.
\end{definition}

Note that there are exactly two fixed components of $v$: It is clear that every vertex has at least two fixed components, otherwise it could be moved. On the other hand general positioned points do not allow the number of fixed components to be greater than two. Hence the following multiplicities that generalize the well-know local $\ev$-multiplicities for $3$-valent vertices are well-defined.

\begin{definition}[Local multiplicities]\label{definition:local_multiplicities}
Let $C$ be a rational tropical curve (possibly with vertices of higher valence) that is fixed by general positioned points $p_1,\dots,p_n$. Let $v$ be a vertex of $C$. If there is a point on $v$, then define $\mult(v)=1$. Otherwise let $v$ be a vertex of $C$ with fixed components $C_1,C_2$ associated to the edges $e_1,e_2$ adjacent to $v$. Let $v_1$ denote the weighted primitive vector of $e_1$ and $v_2$, respectively. The multiplicity of $v$ is defined as
\begin{align*}
\mult_{\ev}(v):=|\det\left( v_1,v_2 \right)|.
\end{align*}
\end{definition}

Another way to think about the multiplicity of a higher-valent vertex is to add up edges of free components, to be more precise, consider the following example: 
\begin{figure}[H]
\centering
\def\svgwidth{250pt}
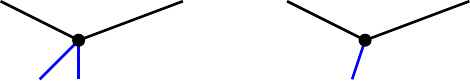
\end{figure}
\noindent On the left there is a $4$-valent vertex whose black edges belong to fixed components and its blue edges belong to free components. The multiplicity of this vertex is completely determined by its black edges. If we ``add" these blue edges (add their direction vectors), we obtain the $3$-valent vertex on the right whose multiplicity is again completely determined by its black edges.

\begin{lemma}\label{lemma:local_multiplicities}
Let $p_1,\dots,p_n,\lambda_1,\dots,\lambda_l$ be in general position, where $|\lambda_j|=0$ for $j=1,\dots,l$ and let $C$ be a rational tropical curve of some degree such that $C$ is fixed by $p_1,\dots,p_n,\lambda_1,\dots,\lambda_l$, then
\begin{align*}
\mult_{\ev}(C)=\prod_{v\mid v \textrm{ vertex of }C}\mult_{\ev}(v)
\end{align*}
\end{lemma}

\begin{proof}
We prove this by induction on the number of vertices of $C$ which is denoted by $k$. Let $k=1$ and denote the vertex of $C$ by $v$.
 There are two choices of general positioned conditions that fix this curve:
\begin{itemize}
\item[1.]
If there is no point on $v$ and $v$ is at least $3$-valent, then we have $n+2$ parameters of $C$ that need to be fixed. On the other hand each point $p_i$ for $i=1,\dots,n$ is in $\mathbb{R}^2$ and therefore $2n=n+2$ for a natural number $n>0$. Hence $n=2$, so there are two ends $e_1,e_2$ that are equipped with points. Denote the weighted primitive vector of $e_1$ (pointing away from $v$) by $u=\left(u_1,u_2\right)$ and the vector of $e_2$ by $w$, respectively. If we choose $p_1$ as the base point of the ev-matrix $M(C)$ of $C$, then
\begin{align*}
M(C)=
\left(
\begin{array}{cccc}
1&0&0&0\\
0&1&0&0\\
1&0&-u_1&w_1\\
0&1&-u_2&w_2
\end{array}
\right)
\end{align*}
has determinant $\mult(v)$.
\item[2.]
If there is a point on $v$ and this point fixes the position of $C$, then
$\mult_{\ev}(C)=1$ since it is the determinant of the $2\times 2$ identity matrix.
\end{itemize}
Let $k> 1$. In order to use induction and lower the number of vertices, we have to split off components. This has been done in the case where all vertices are $3$-valent, see prop 3.8 of \cite{KontsevichPaper}. Let $v$ be a vertex of $C$ and let $C_1$ be a component of $v$ that contains at least on vertex. Denote by $C'$ the tropical curve after cutting $e_1$ that belongs to $v$. Introduce a new point $p$ on $e'_1\in C'$, where $e_1'$ denotes the cut and stretched edge $e_1$ in $C'$ and denote $C'$ with its new point by $C''$. The proof of proposition 3.8 in \cite{KontsevichPaper} given by Gathmann and Markwig can easily be adapted to our situation, such that
\begin{align*}
\mult_{\ev}(C)=\mult_{\ev}(C_1)\cdot \mult_{\ev}(C'')
\end{align*}
holds and the induction hypothesis can be applied.
\end{proof}

We finish this section by summing up the most important results of this section in a theorem.

\begin{theorem}\label{thm:ZSFSSG_section_2}
Let $\Delta$ be a degree and let $p_1,\dots,p_n,\lambda'_1,\dots,\lambda'_l$ be conditions as defined in \ref{notation}. Let $\lambda_1,\dots,\lambda_{l}$ denote the degenerations of $\lambda'_1,\dots,\lambda'_{l}$ and define
\begin{align*}
N_{0,n}\left(\lambda_1,\dots,\lambda_l\right):=
\degree\left(\prod_{j=1}^{l}\ft_{\lambda_j}^*\left( 0\right)\cdot\prod_{i=1}^n\ev_i^*\left( p_i\right)\cdot\mathcal{M}_{0,n}\left(\mathbb{R}^2,\Delta\right)\right).
\end{align*} 
Then
\begin{align*}
N_{0,n}\left(\lambda'_1,\dots,\lambda'_l\right)
=
N_{0,n}\left(\lambda_1,\dots,\lambda_l\right)
\end{align*}
holds, where $N_{0,n}\left(\lambda'_1,\dots,\lambda'_l\right)$ is defined in definition \ref{definition:general_position_I}. Moreover, the multiplicity of a tropical curve contributing to the right side can be expressed locally as
\begin{align*}
\mult(C)=\prod_{v\mid v \textrm{ vertex of }C}\mult_{\ev}(v)\cdot \omega_v(\sigma_C),
\end{align*}
where $\omega_v(\sigma_C)$ is the local weight of the top-dimensional cell $\sigma_C$ of $X$ that contains $C$ (see definition \ref{definition:resolution_weights_local}) and $\mult_{\ev}(v)$ is defined in \ref{definition:local_multiplicities}.
\end{theorem}

\begin{proof}
The first part is a consequence of remark \ref{remark:facts_about_rational_equivalence}. For the second part, note that if $C$ is a tropical curve corresponding to a point in $\prod_{i=1}^n\ev_i^*\left( p_i\right)\cdot X$ such that 
$$X=\prod_{j=1}^{l}\ft_{\lambda_j}^*\left( 0\right)\cdot \mathcal{M}_{0,n}\left(\mathbb{R}^2,\Delta \right),$$
then the contribution of $C$ to $N_{0,n}\left(\lambda_1,\dots,\lambda_l\right)$ is
\begin{align*}
\mult(C)=\prod_{v\mid v \textrm{ vertex of }C}\mult_{\ev}(v)\cdot \omega_v(\sigma_C)
\end{align*}
due to proposition \ref{prop:mult_X}, lemma \ref{lemma:local_multiplicities} and corollary \ref{corollary:resolution_weights_local}.
\end{proof}

Combining the correspondence theorem \ref{thm:correspondence_thm_CRC} and theorem \ref{thm:ZSFSSG_section_2} enables us to enumerate classical curves satisfying point and classical cross-ratio conditions using degenerated cross-ratios. We state this in the following corollary, which is used to obtain a cross-ratio lattice path algorithm in the next section.

\begin{corollary}\label{corollary:ZSFSSG_section}
Use the same notations/assumptions as in the correspondence theorem \ref{thm:correspondence_thm_CRC} and denote the degenerations of $\lambda'_1,\dots,\lambda'_l$ by $\lambda_1,\dots,\lambda_l$. Then
\begin{align*}
N^{\textrm{class}}_{0,n}\left( \mu_1,\dots,\mu_l \right)
=
N_{0,n}\left(\lambda_1,\dots,\lambda_l\right)
\end{align*}
holds.
\end{corollary}

The results of this section can be generalized to counts of curves satisfying tangency conditions to the toric boundary, point conditions and cross-ratio conditions in a straight-forward way.
We make use of this in section \ref{section:special_case:floor_diagrams_cross-ratios} when dealing with floor diagrams. Here, we sum up the relevant notations.

\begin{lemma}[Evaluation of horizontal ends]\label{lemma:ev_horizontal_ends}
The pull-backs of the maps
\begin{align*}
\partial \ev_k:
\mathcal{M}_{0,n}\left(\mathbb{R}^2,\Delta\left(\alpha,\beta\right) \right)
&\to
\mathbb{R} \\
(\Gamma,x_1,\dots,x_N,h)
&\mapsto
\left( h\mid_{x_k} \right)_y
\end{align*}
are well-defined for $k=1,\dots,l(\alpha)+l(\beta)$.
\end{lemma}

\begin{proof}
This follows immediately from
\begin{align}\label{eq:delta_ev_equals_projection_ev}
\partial\ev_k=\pi_y\circ \ev_h
\end{align}
for some label $h$ of an ending, where $\pi_y$ is the projection on the $y$-coordinate of $\mathbb{R}^2$ and proposition 1.12 of \cite{JohannesIntersectionsonTropModuliSpaces}.
\end{proof}

The pull-back of a map $\partial \ev_k$ for some $k$ imposes a condition on the height of a horizontal end, corresponding to tangency conditions with the toric boundary. General position for point-, end- and cross-ratio conditions can be defined analogously to 
 definitions \ref{definition:general_position_I} and \ref{definition:general_position_II}. The multiplicity of a curve in a $0$-dimensional cycle in the moduli space of rational tropical stable maps corresponding to point-, end- and cross-ratio conditions can be computed as in lemma \ref{lemma:local_multiplicities}.

\section{Cross-ratio lattice path algorithm}\label{section:cross-ratio_lattice_path_algorithm}

In this section we present a generalized lattice path algorithm to determine the number of tropical curves passing through prescribed points and satisfying given degenerated cross-ratio constraints.

\begin{definition}~\vspace{-\baselineskip}
\begin{itemize}
\item
An \textit{edge} $E$ is a $1$-dimensional lattice polytope in $\mathbb{R}^2$ consisting of one $1$-dimensional face and two $0$-dimensional faces. A \textit{labeled edge} is a tuple $\left( E,\tau^E\right)$, where $\tau^E$ is a multiset of $m>0$ elements denoted by $\tau_1,\dots,\tau_m$ in $\mathbb{N}_{>0}$ such that $\sum_i \tau^E_i=|E|$, where $|E|$ denotes the lattice length of $E$. We refer to $\tau^E$ as \textit{labeling} of $E$ and to $\tau_1,\dots,\tau_m$ as \textit{labels} of $E$.
\item
In particular, we call a labeled edge $\left(E,\tau \right)$ where $\tau=\lbrace n\rbrace$ for some $n\in\mathbb{N}_{>0}$ a \textit{segment}.
\item 
Let $P$ be a lattice polytope in $\mathbb{R}^2$ where each of its $e$  facets is a labeled edge. Denote the labeling of an edge $E^j$ of $P$ by $\tau^j$. Then $\left( P,\tau\right)$ with $\tau=\left(\tau^1,\dots,\tau^e \right)$ is called a \textit{labeled polytope}.
\end{itemize}
\end{definition}

\begin{definition}[Minkowski labeled polytopes]\label{definition:Minkowski_labeled_polytopes}
Let $P$ be the Minkowski sum of a labeled polytope $\tilde{P}\subset\mathbb{R}^2$ that is either $0$-dimensional or $2$-dimensional and segments $S_1,\dots,S_r$ such that each segment is parallel to an edge of $\tilde{P}$ and $P$ is $2$-dimensional. Note that if $\tilde{P}$ is a point, then every segment is by definition parallel to it. Moreover, we require that if $\tilde{P}$ is $0$-dimensional, then there are two segments $S_{i_1},S_{i_2}\in\lbrace S_1,\dots, S_r\rbrace$ such that all other Minkowski summands of $P$ are parallel to one of them. Let $E$ be an edge of $P$ and denote by $F_1,\dots,F_k$ edges of the Minkowski summands $\tilde{P},S_1,\dots,S_r$ that contribute to $E$. If $\tau^{F_i}$ is the labeling of $F_i$, then we define $\mathcal{\tau}^E$ to be the multiset
\begin{align*}
\tau^E:=\tau^{F_1}\cupdot\dots\cupdot\tau^{F_k}.
\end{align*}
A pair $\left( P,\tau \right)$ of such a polytope $P$ with $e$ edges $E^1,\dots,E^e$ and a tuple of multisets $\tau=\left( \tau^{E^1},\dots,\tau^{E^e} \right)$, where $\tau^{E^i}$ is defined above, together with maps that \textit{match} labels to the summands they come from
\begin{align*}
f_P\mid_E:\tau^E \to \lbrace \tilde{P},S_1,\dots,S_r \rbrace
\end{align*}
such that if $f_P\mid_E(t)=A\in\lbrace \tilde{P},S_1,\dots,S_r \rbrace$, then $t\in\tau^{F_i}$ for $F_i\subset A$, is called a \textit{Minkowski labeled polytope}.
\end{definition}

\begin{figure}[H]
\centering
\def\svgwidth{400pt}
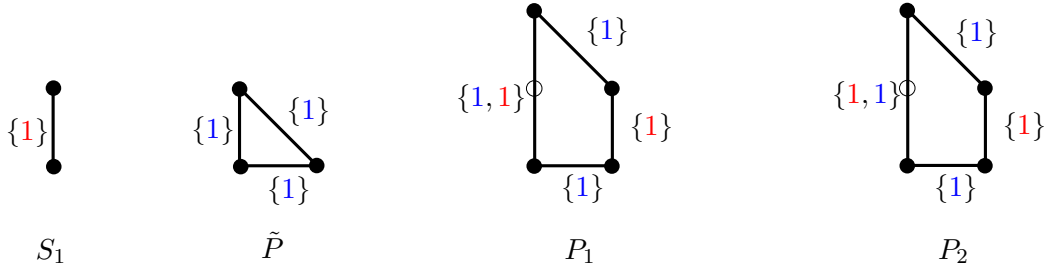
\caption{From left to right: A Segment $S_1$ and a $2$-dimensional labeled polytope $\tilde{P}$ whose Minkowski sum forms the labeled polytopes $P_1,P_2$ on the right. The colors indicate the matching of labelings of $P_1,P_2$ to their Minkowski summands. $P_1$ and $P_2$ just differ by the way the multiset $\tau^E$ on the left edge is listed, they are the same Minkowski labeled polytope.}
\label{Example_uniquely_labeled_polytope}
\end{figure}

We always denote the non-segment Minkowski summand of a Minkowski labeled polytope $P$  by $\tilde{P}$.

\begin{definition}~\vspace{-\baselineskip}
\begin{itemize}
\item
A Minkowski labeled poyltope $P$ is called \textit{$k$-marked} if $\tilde{P}$ has $e$ edges $E^j$ with labelings $\tau^j$ such that $\sum_{j=1}^{e} |\tau^j|=3+k$ holds, where $|\tau^j|\in\mathbb{N}_{>0}$ is the number of entries of $\tau^j$. If $k=0$ or $\tilde{P}$ is $0$-dimensional, then $P$ is called \textit{unmarked}.

\item
A Minkowski labeled polytope is called \textit{valid polytope} if it is either unmarked or $k$-marked. Two valid polytopes that share an edge $E$ are \textit{compatible} if their labelings of $E$ coincide.
\item
Let $\tilde{P}$ be a $1$-dimensional polytope where each side of its edge $E$ is equipped with a labeling. The Minkowski sum of $\tilde{P}$ with segments $S_1,\dots,S_r$ parallel to it, where each summand contributes a label to the two labelings of $E$ as in definition \ref{definition:Minkowski_labeled_polytopes} is called a \textit{pointed segment}. If $\tilde{P}$ is $0$-dimensional, then it is called a non-pointed segment (all $S_i$ are then parallel). The notion of compatibility extends to (non-)pointed segments as well: If a valid polytope and a (non-)pointed segment share an edge, then they are compatible if their labelings on this (side of the) edge coincide. We can refer to a (non-)pointed segment as $k$-marked as above. 
\end{itemize}
\end{definition}

\begin{definition}[Coloring]
A \textit{coloring} of a labeled polytope $P$ is a $2$-coloring of all of its labels on each of its edges. 
The two colors are called \textit{fixed} and \textit{free}. A colored polytope is called \textit{free} (or \textit{fixed}) if it is monochromatic of the color free (or fixed). Given a colored Minkowski labeled polytope $P$, we say that exactly $\tilde{P}$ is fixed if all labels associated to $\tilde{P}$ are colored fixed and the rest is colored free.
\end{definition}

\begin{algorithm}[Adjusting colors of two compatible polytopes.]\label{algorithm:adjusting_colors_two_polytopes}
Let $P_1,P_2$ be two colored polytopes that are compatible and denote their shared edge by $E$ with labelings $\tau_{P_1}^E,\tau_{P_2}^E$. Let $f_{P_1}\mid_E,f_{P_2}\mid_E$ be maps as in definition \ref{definition:Minkowski_labeled_polytopes} and let $g:\tau_{P_1}^E\to\tau_{P_2}^E$ be a bijective map such that $g(t)=t$ for all $t\in\tau^E\cap\mathbb{N}_{>0}$. Let $t\in\tau_{P_1}^E$ be a colored label of $E$ in $P_1$ and let $g(t)$ be its image under $g$ in $\tau_{P_2}^E$. When comparing and adjusting the colors of $t$ and $g(t)$, we follow the slogan ``fixed wins":
\begin{itemize}
\item[(1)] If $t$ is colored fixed and $g(t)$ is colored fixed, we leave the colors the way they are.
\item[(2)] If $t$ is colored fixed and $g(t)$ is colored free, we change $g(t)$ to fixed. When changing $g(t)$ to fixed, we check whether all other labels coming from $f_{P_2}\mid_E(g(t))$ are fixed. If this is not the case, then change them to fixed if $f_{P_2}\mid_E(g(t))$ is a segment. If $f_{P_2}\mid_E$ associates $g(t)$ to $\tilde{P}_2$, then change the labels associated to $\tilde{P}_2$ to fixed if exactly two of the labels associated to $\tilde{P}_2$ are fixed (where $g(t)$ is one of them).
\item[(3)] If $t$ is colored free and $g(t)$ is colored fixed, then do the same as in (2) but with the roles of $t,g(t)$ and $P_1,P_2$ exchanged.
\item[(4)] If $t$ is colored free and so is $g(t)$, then do nothing.
\end{itemize}
We repeat this procedure using different labels in $\tau_{P_1}^E$ until no color of labels of $P_1,P_2$ can be changed according to the rules above. Note that this algorithm terminates since colors can only be changed from free to fixed.
\end{algorithm}

\begin{algorithm}[Adjusting colors of a set of polytopes]\label{algorithm:adjusting_colors_set_polytopes}
Let $P_1,\dots,P_z$ be a finite set of colored polytopes, where two polytopes are compatible if they share an edge. Go through all pairs of compatible polytopes of $P_1,\dots,P_z$ and adjust their colors according to algorithm \ref{algorithm:adjusting_colors_two_polytopes}. Repeat this procedure until no colors can be changed. This algorithm terminates because we only allow changing a color from free to fixed, following the slogan that fixed wins.
\end{algorithm}

Note that the notion of coloring and adjusting colors extends to (non-)pointed segments.

The following definitions can be found in \cite{MikhalkinLatticePaths} and \cite{rag-rug}.

\begin{definition}[Lattice path]\label{definition:lattice_path}
Fix $\theta$ to be a linear map of the form
\begin{align*}
\theta:\mathbb{R}^2\to\mathbb{R}, \left(x,y\right)\mapsto x-\epsilon y,
\end{align*}
where $\epsilon$ is a small irrational number. A path $\gamma:\left[ 0,n\right]\to\mathbb{R}^2$ is called a \textit{lattice path} if $\gamma\mid_{[j-1,j]}$ for $j=1,\dots,n$ is an affine-linear map and $\gamma(j)\in\mathbb{Z}^2$ for all $j=0,\dots,n$. For $j=1,\dots,n$, we call $\gamma\mid_{[j-1,j]}\left([j-1,j] \right)$ a \textit{step} (the $j$-th step) of the lattice path $\gamma$. A lattice path is called \textit{$\theta$-increasing} if $\theta\circ\gamma$ is strictly increasing. If every step in a lattice path is a labeled edge, the lattice path is called \textit{labeled lattice path}.
\end{definition}

\begin{definition}[Cross-ratio lattice path]\label{definition:rag-rug}
Let $\Sigma$ be a polytope in $\mathbb{R}^2$ and let $n\in\mathbb{N}_{>0}$. Let $\mathcal{A}$ be a set $\lbrace P_1,\dots,P_{n+z} \rbrace$ of colored polytopes in $\Sigma$ such that there are polytopes $\lbrace P_{i_1},\dots,P_{i_n}\rbrace\subset\mathcal{A}$ such that $P_{i_j}$ is a pointed segments or a valid polytope such that $\tilde{P}_{i_j}$ is fixed and not $0$-dimensional for $j=1,\dots,n$. The other polytopes in $\mathcal{A}\backslash\lbrace P_{i_1},\dots,P_{i_n}\rbrace$ are non-pointed segments that are colored free. The set $\mathcal{A}$ is called a \textit{cross-ratio lattice path} if the following conditions are satisfied:
\begin{itemize}
\item[(1)]
two polytopes $P_i,P_j$ intersect in at most one point,
\item[(2)]
if an edge $E$ of a polytope $P_i$ lies in the boundary $\partial\Sigma$ of $\Sigma$ it is labeled by $\tau^E=\left( 1,\dots, 1\right)$,
\item[(3)]
 there are sets $\gamma_+,\gamma_-$ of edges of $P_1,\dots,P_{n+z}$ such that $\gamma_+,\gamma_-$ form $\theta$-increasing labeled lattice paths, $\gamma_+\cup\gamma_-$ is the set of all edges of $P_1,\dots,P_{n+z}$ and for all $x\in\pi_x\left( \Sigma\right)$ (where $\pi_x$ is the projection of $\mathbb{R}^2$ to the $x$-axis) and  all $E_+\in\gamma_+,E_-\in\gamma_-$ such that there are points $\left(x,y_+ \right)\in E_+  \subset\mathbb{R}^2,\left(x,y_- \right)\in E_-  \subset\mathbb{R}^2$ the inequality $y_+\geq y_-$ holds (see Figure \ref{Example_upper_and_lower_path}),
\item[(4)]
the order of the polytopes $P_1,\dots,P_{n+z}$ agrees with the obvious order given by $\gamma_+$ and $\gamma_-$, respectively,
\item[(5)]
let $p$ and $q$ be the points in $\Sigma$ where $\theta\mid_\Sigma$ reaches its minimum (resp. maximum), then $p=\gamma_+(0)=\gamma_-(0)$ and $q=\gamma_+(n_+)=\gamma_-(n_-)$, where $\gamma_+:\left[ 0,n_+\right]\to\mathbb{R}^2$ and $\gamma_-:\left[ 0,n_-\right]\to\mathbb{R}^2$ are defined as above.
\end{itemize}
\end{definition}

\begin{figure}[H]
\centering
\def\svgwidth{400pt}
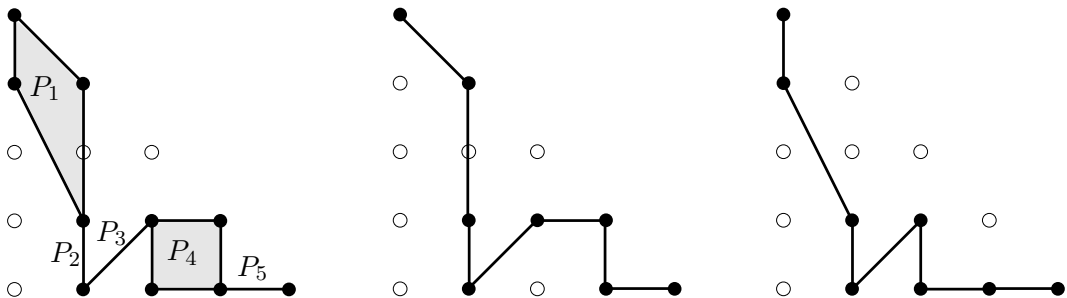
\caption{ Let $\Sigma=\operatorname{conv}\left((0,d),(0,0),(d,0) \right)$. From left to right: $\mathcal{A}=\lbrace P_1,\dots,P_5\rbrace, \gamma_+, \gamma_-$.}
\label{Example_upper_and_lower_path}
\end{figure}

Throughout the following, we fix a degree $\Delta(\Sigma)$ from a polytope $\Sigma\subset\mathbb{R}^2$, see definition \ref{definition:degree}, point conditions $p_1,\dots,p_n$ and degenerated cross-ratio constraints $\lambda_1,\dots,\lambda_l$ in general position.

\begin{construction}[Constructing subdivisions of $\Sigma$ from a cross-ratio lattice path $\mathcal{A}$]\label{construction:subdivisions_from_lattice_path}
Let $\mathcal{A}$ be a cross-ratio lattice path in the polytope $\Sigma$ with $\#\mathcal{A}=n+z$ for some $z\in\mathbb{N}$ such that $z\leq\#\left(\Sigma\cap\mathbb{Z}^2\right)$. Let $\gamma_+$ be the associated labeled lattice path as before. Recall that in the ``standard" lattice path algorithm left (resp. right) turns of a given lattice path are filled up with triangles and parallelograms. In our case we must allow more polytopes than only triangles and parallelograms.

Let $\gamma_+(j)$ and $\gamma_+(j+1)$ be the $j$-th and the $(j+1)$-th labeled edge of $\gamma_+$ that form the first left turn. Fill up this left turn with a valid polytope $P\subset\Sigma$ that is colored free, whose edges that equal $\gamma_+(j)$ and $\gamma_+(j+1)$ are compatible with $\gamma_+(j)$ and $\gamma_+(j+1)$ and if $P$ shares other edges with our polytopes, it should there be compatible, too. Whenever two compatible labeled edges with labelings $\tau^E$ come together, we choose a bijective map $g:\tau^E\to\tau^E$ such that $g(t)=t$ for all $t\in\tau^E\cap\mathbb{N}_{>0}$. Moreover, we use algorithm \ref{algorithm:adjusting_colors_set_polytopes} to adjust the colors of the set of polytopes we have so far. If $P$ shares an edge $E$ with $\partial \Sigma$, then we require $\tau^E=\left( 1,\dots, 1\right)$ and we choose a bijective map $g':\tau^E\to M$, where $M$ is a submultiset of the labels of the degree $\Delta(\Sigma)$ that are associated to vectors orthogonal (and pointing away from $\Sigma$) to $E$ (see definition \ref{definition:degree}). When another polytope $P'$ shares an edge with $\partial \Sigma$, then we choose $M'$ in the set of labels of $\Delta(\Sigma)$ minus $M$ and so on. In the same way the right turns of $\gamma_-$ can be filled up.

Repeating these steps, we obtain subdivisions of $\Sigma$ if and only if $\Sigma=\mathcal{A}\cup\bigcup \lbrace P\rbrace$, where the union runs over all valid polytopes $P$ used to fill up turns during the process described above. The cells of such a subdivision are valid polytopes which are compatible and connected via maps called $g$ above. Such a subdivision is called a \textit{lattice path subdivision of $\mathcal{A}$} if all polytopes are fixed. The set of all lattice path subdivisions of $\mathcal{A}$ is denoted by $\mathcal{S}_0(\mathcal{A})$.
\end{construction}

\begin{figure}
\centering
\def\svgwidth{200pt}
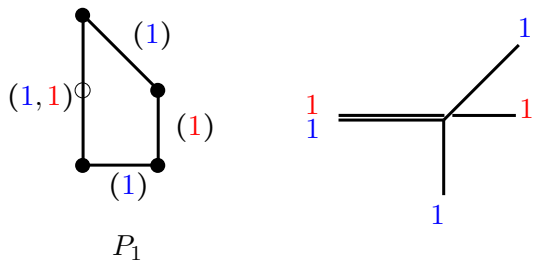
\caption{On the left is the Minkowski labeled polytope $P_1$ introduced in Figure \ref{Example_uniquely_labeled_polytope} and on the right is its dual tropical curve.}
\label{Example_dual_tropical_curve}
\end{figure}

\begin{construction}[Dual tropical curve]\label{construction:dual_tropical_curve}
Let $\mathcal{S}\in\mathcal{S}_0(\mathcal{A})$ be a lattice path subdivision. We want to construct the dual tropical curve $C_\mathcal{S}\in\mathcal{M}_{0,n}\left(\mathbb{R}^2,\Delta(\Sigma) \right)$ to $\mathcal{S}$. For that draw a $k$-valent vertex $v$ for every $k$-marked ($k>0$) polytope $P$ in $\mathcal{S}$ and an edge passing through this vertex for every segment of $P$. An edge $e$ adjacent to $v$ is dual to an edge $E$ of $\tilde{P}$, that is the weight of $e$ is given by an entry of the labeling $\tau^E$ of $E$. The weight of an edge passing through $v$ is given by the label of its associated segment that is dual to this edge.
If two polytopes $P,Q\in\mathcal{S}_0(\mathcal{A})$ share an edge $E$ with labeling $\tau^E$, we connect the edge associated to $\tau^E_i$ in $P$ with the edge associated to $g\left( \tau^E_i\right)$ in $Q$ for all $i$, where $g$ is a map as in construction \ref{construction:subdivisions_from_lattice_path}. Moreover, if $P\in\mathcal{A}$ and $P$ is neither a pointed segment nor a non-pointed segment, then add a point (a contracted end) to the vertex dual to $\tilde{P}$. If $P\in\mathcal{A}$ and $P$ is a pointed segment, then the edges dual to the labelings associated to $\tilde{P}$ meet in one vertex which is in addition adjacent to a point. In this way, we obtain the combinatorial type of $C_\mathcal{S}$. From the general construction of tropical curves dual to lattice paths (see \cite{MikhalkinFundamental}) and the fact that all polytopes are fixed, it follows that for given points $p_1,\dots,p_n$ in general position linearly ordered on a line with a small negative slope such that distances grow ($|p_{i}-p_{i-1}| <<|p_{i+1}-p_i|$) there is exactly one curve of type $C_\mathcal{S}$ that satisfies the point conditions.
\end{construction}

Since we are only interested in genus zero curves, we need to remove  subdivisions whose dual tropical curves are reducible. We denote the  set of lattice path subdivisions for a given cross-ratio lattice path $\mathcal{A}$ which are dual to irreducible tropical curves by $\mathcal{S}_1(\mathcal{A})$.

\begin{definition}\label{definition:CR_fit}
Let $\Lambda:=\bigcup_{j=1}^l \lambda_j$ the union of all given degenerated cross-ratio constraints. Let $\mathcal{S}$ be a lattice path subdivision in $\mathcal{S}_1(\mathcal{A})$ and let $P$ be a valid polytope or a pointed segment in $\mathcal{S}$. Consider the summand $\tilde{P}$ of $P$ and define for all entries $\tau_1,\dots,\tau_m$ of labelings of edges of $P$ associated to $\tilde{P}$  the sets $\Lambda(P,i)\subset\Lambda$ of points and ends appearing in the cross-ratios $\lambda_1,\dots,\lambda_l$ that can be reached from $P$ via $\tau_i$. That is, we obtain the elements of $\Lambda(P,i)$ with the following procedure:
\begin{itemize}
\item
If the edge $E$ of $P$ where $\tau_i$ appears is contained in $\partial\Sigma$, then its dual edge is a labeled end determined by $g(\tau_i)$ (construction \ref{construction:subdivisions_from_lattice_path}), and we add it to $\Lambda(P,i)$.
\item
Else there is a valid polytope (or a pointed segment) $Q$ in $\mathcal{S}$ such that $Q\neq P$ and $P,Q$ share an edge $E$ such that $\tau_i$ appears in $\tau^E$. Then either:
\begin{itemize}
\item
$\tau_i$ is mapped to $\tilde{Q}$ (via the map $f_Q\mid_E$ from definition \ref{definition:Minkowski_labeled_polytopes}) and $Q\nin\mathcal{A}$, then continue with all other labels mapped to $\tilde{Q}$ instead of $\tau_i$. 
\item
$\tau_i$ is mapped to $\tilde{Q}$ and $Q=P_j\in\mathcal{A}$, then  add the marked point $x_j$ to $\Lambda(P,i)$ and continue with all other labels mapped to $\tilde{Q}$ instead of $\tau_i$
\item
$\tau_i$ is mapped to a segment of $Q$, then there is exactly one $\tau_i'$ in another edge $E'$ of $Q$ that is mapped to the same segment. We continue with this.
\end{itemize}
In each case, we follow all appearing edges until we reach edges in $\partial\Sigma$ for which we add the labels of the dual ends to $\Lambda(P,i)$.
\end{itemize}

Furthermore, if $P$ is a polytope appearing in the lattice path $\mathcal{A}$ itself as $j$-th step, then we set $\Lambda(P,0):=\lbrace
 x_j \rbrace$, the $j$-th marked point. Otherwise, we set $\Lambda(P,0):=\emptyset$.

Moreover, we define 
\begin{align*}
\Lambda(P)&:=\lbrace \lambda_j=\lbrace \beta_{j_1},\dots,\beta_{j_4}\rbrace\mid \beta_{j_i}\in\Lambda(P,k_i) \textrm{ for $i=1,\dots,4$ and $k_i\neq k_{i'}$ if $i\neq i'$} \rbrace,
\end{align*}
and we say that the lattice path subdivision $\mathcal{S}$ \textit{fits} the cross-ratios $\lambda_1,\dots,\lambda_l$ if
\begin{align*}
\sum_{P} \#\Lambda(P)=l,
\end{align*}
where the sum goes over all valid polytopes and pointed segments in $\mathcal{S}$ and
\begin{align*}
\#\Lambda(P)=
\begin{cases}
k & \textrm{, if $P$ is $k$-marked}\\
0 & \textrm{, otherwise.}
\end{cases}
\end{align*}
\end{definition}

For a cross-ratio lattice path $\mathcal{A}$, the subset of $\mathcal{S}_1(\mathcal{A})$ of subdivisions which fit the given cross-ratios is denoted  by $\mathcal{S}_2(\mathcal{A})$.

\begin{definition}[Multiplicity of a subdivision]
In order to associate a multiplicity to a lattice path subdivision $\mathcal{S}$ in $\mathcal{S}_2(\mathcal{A})$, define
\begin{align*}
\mult_{\ev}(\mathcal{S}):=\prod_{P} \mult_{\ev}(P),
\end{align*}
where the product goes over all valid polytopes and pointed segments in $\mathcal{S}$, and $\mult(P)$ is defined as follows: If $\tilde{P}$ is $0$-dimensional or $P\in\mathcal{A}$, then $\mult(P):=1$. Otherwise let $\tau_1,\dots,\tau_m$ denote the entries of labelings of edges of $P$ associated to $\tilde{P}$, let $\mathcal{E}_i$ be the number of ends that can be reached from $P$ via $\tau_i$ and let $\mathcal{C}_i$ be the number of constraints that can be reached from $P$ via $\tau_i$ (using the procedure from definition \ref{definition:CR_fit}), that is
\begin{align*}
\mathcal{C}_i&:=\mathcal{C}_i^{\left(\textrm{points}\right)}+\mathcal{C}_i^{\left(\textrm{cross-ratios}\right)},\\
\mathcal{C}_i^{\left(\textrm{cross-ratios}\right)}&:= \sum_{P'} \#\Lambda(P'),
\end{align*}
where the sum goes over all valid polytopes and pointed segments in $\mathcal{S}$ that can be reached from $P$ via $\tau_i$, $\Lambda(P')$ is defined in \ref{definition:CR_fit} and $\mathcal{C}_i^{\left(\textrm{points}\right)}$ is the number of points that can be reached from $P$ via $\tau_i$. We have either $\mathcal{E}_i-1=\mathcal{C}_i$ or
$\mathcal{E}_i-2=\mathcal{C}_i$: in the first case, the edge dual to $\tau_i$ in the tropical curve leads to a fixed component, in the second to a free component (see definition \ref{definition:fixed_free_components}). Every vertex of the dual tropical curve has exactly two fixed components, we use the indices $i_0$ and $i_1$ for those labels corresponding to edges in the dual tropical curve that lead to a fixed component. Then we set
\begin{align*}
\mult_{\ev}(P):=|\det\left(\tau_{i_0}\cdot v_0, \tau_{i_1}\cdot v_1 \right)|,
\end{align*}
where $v_0$ is the primitive vector of the edge $E_0$ of $P$ that belongs to $\tau_{i_0}$ and $v_1$, respectively.

Furthermore let $C_\mathcal{S}$ be the dual tropical curve of $\mathcal{S}$ (see construction \ref{construction:dual_tropical_curve}). Let $X:=\prod_{j=1}^{l}\ft_{\lambda_j}^*\left( 0\right)\cdot\mathcal{M}_{0,n}\left(\mathbb{R}^2,\Delta(\Sigma) \right)$. Note that $C_\mathcal{S}\in X$ since the lattice path subdivision $\mathcal{S}$ fits the cross-ratios $\lambda_1,\dots,\lambda_l$. Moreover, $C_\mathcal{S}$ passes through the points $p_1,\dots,p_n$ by construction \ref{construction:dual_tropical_curve}. Using remark \ref{corollary:mult_ev_well-def}, we know that $C_\mathcal{S}$ lies in the interior of a top-dimensional cell of $X$. Denote this top-dimensional cell by $\sigma_\mathcal{S}$ and define $\omega (\sigma_\mathcal{S})$ to be its weight. Recall that this weight has a local structure, see corollary \ref{corollary:resolution_weights_local}.

We define the \textit{multiplicity $\mult(\mathcal{S})$ of $\mathcal{S}$} as
\begin{align*}
\mult(\mathcal{S}):=\mult_{\ev}(\mathcal{S})\cdot \omega (\sigma_\mathcal{S}).
\end{align*}
\end{definition}

By definition, we have $\mult(\mathcal{S}) = \mult(C_\mathcal{S})$ for all $\mathcal{S}\in\mathcal{S}_2(\mathcal{A})$.

\begin{definition}\label{definition:number_lpa}
Given cross-ratio constraints $\lambda_1,\dots,\lambda_l$, we denote the sum over all $\mathcal{S}\in\mathcal{S}_2(\mathcal{A})$ (counted with multiplicity) for all cross-ratio lattice paths $\mathcal{A}$ with $n+z$ steps for all $z$ by $N_{0,n}^{\textrm{lpa}}\left(\lambda_1,\dots,\lambda_l\right)$.
\end{definition}

\begin{remark}[Arbitrary degree]
Note that we do not need to restrict to a degree $\Delta$ coming from a polytope where all entries of all partitions are one (see definition \ref{definition:degree}). We restrict to $\Delta(\Sigma)$ here to keep notation as simple as possible. The cross-ratio lattice path algorithm can be extended to arbitrary degrees.
\end{remark}

\begin{figure}[H]
\centering
\def\svgwidth{380pt}
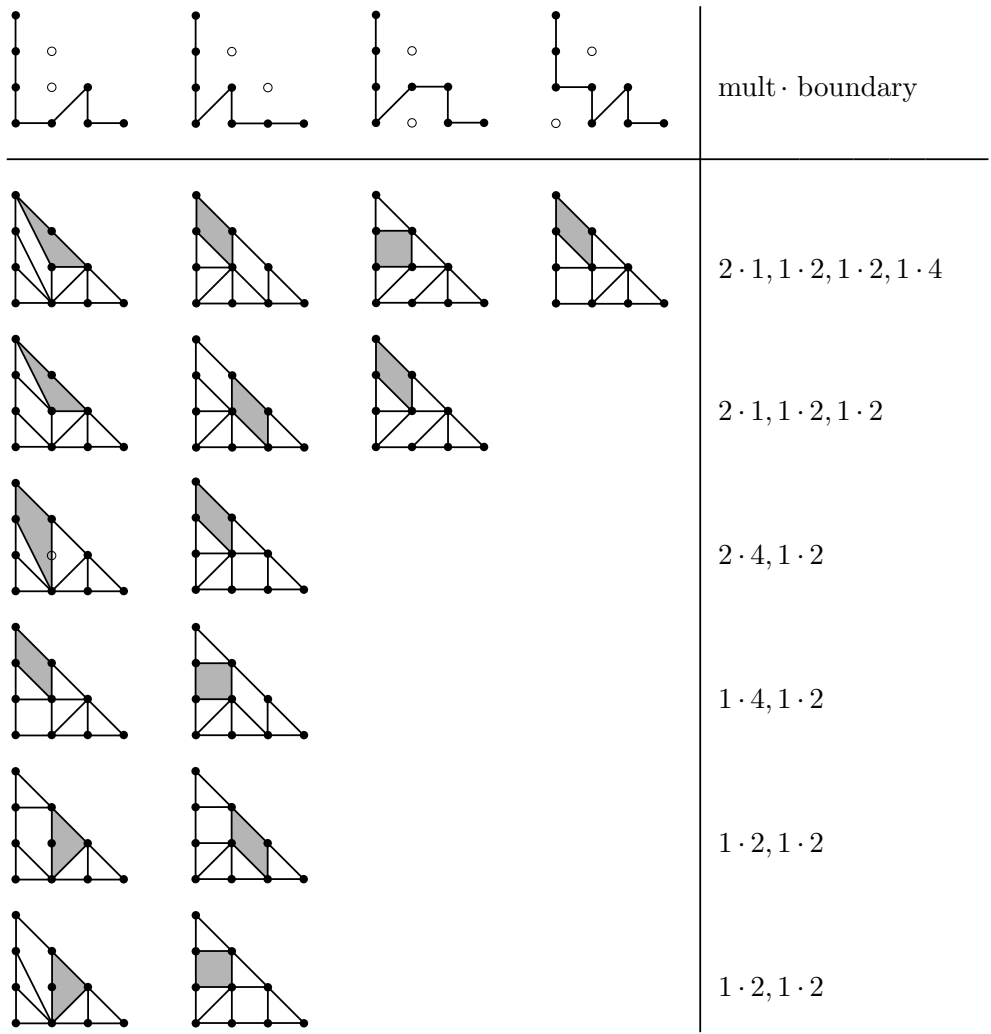
\caption{A complete example of lattice paths, subdivisions and their multiplicities.}
\label{Example_complete_lattice_path}
\end{figure}

\begin{example}\label{example:Kontsevich_CR_2}
We want to give an example of the lattice path algorithm. Fix the degree $\Delta_d$ for $d=3$ (cf. definition \ref{definition:degree}). We choose points $p_1,\dots,p_7$ and a degenerated cross-ratio $\lambda=\lbrace x_1,x_2,7,8\rbrace$. It turns out that all cross-ratio lattice paths we need to consider have $7$ steps. The top row of Figure \ref{Example_complete_lattice_path} shows these cross-ratio lattice paths. There are no labels on polytopes and colors in Figure \ref{Example_complete_lattice_path} because all labels are $1$ and all labels are colored fixed. The column under each of these cross-ratio lattice paths shows the subdivisions arising from these lattice paths. The maps that glue together the polytopes in a subdivision (maps like $g$ from construction \ref{construction:subdivisions_from_lattice_path}) are not mentioned in Figure \ref{Example_complete_lattice_path} since they are the obvious ones. However, the glueing maps that connect the polytopes in the subdivsion to the boundary of $\Delta_3$ are not unique since we labeled ends of tropical curves (we come back to this later). The grey polytopes are $1$-marked, that is $\lambda$ sits at these polytopes. Note that all subdivions fit the cross-ratio $\lambda$ for some choice of glueing the polytopes to the boundary.

The numbers in the rightmost column correspond to subdivisions shown on the left. Each of these numbers is a product, where the first factor is the multiplicity $\mult(\mathcal{S})$ of its associated subdivision $\mathcal{S}$. Note that $\omega (\sigma_\mathcal{S})=1$ for all subdivisions since there is only one way of resolving the $4$-valent vertex dual to each $1$-marked polytope according to some $\lambda'$ degenerating to $\lambda$. The second factor comes from different glueings of polytopes to the boundary of $\Delta_3$ and can easily be seen from an example, see Figure \ref{Example_complete_lattice_path_2}.

The total sum of the numbers in the right column is $40$, which is the number of unlabeled tropical curves satisfying the given point conditions and the cross-ratio constraint. Since the second factor of each product in the rightmost column equals the number of ways to label ends parallel to the vector $(1,1)\in\mathbb{R}^2$, we obtain the number of labeled tropical curves satisfying our given conditions by multiplying $40$ with $\left( 3!\right)^2$, which is $1440$ as we would expect considering example \ref{example:Kontsevich_CR}. Thus we checked that we are not missing any subdivisions.

\begin{figure}[H]
\centering
\def\svgwidth{380pt}
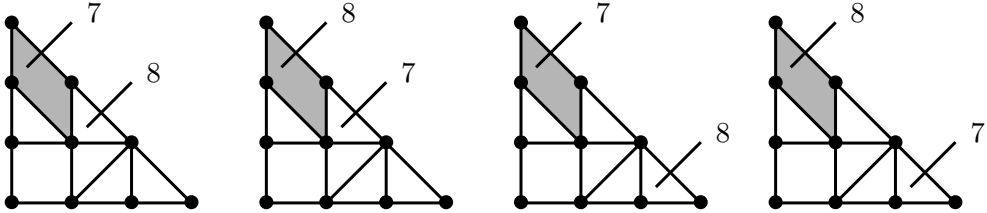
\caption{The subdivision in the right top corner of Figure \ref{Example_complete_lattice_path} and the $4$ different choices of labels of ends in $\lambda$ such that the subdivision still fits $\lambda$.}
\label{Example_complete_lattice_path_2}
\end{figure}

\end{example}

\section{Duality: tropical curves \& subdivisions} \label{section:duality:tropical_curves_subdivisions}

In this section we want to prove theorem \ref{theorem:same_numbers} that relates the numbers obtained from the cross-ratio lattice path algorithm to the enumerative numbers $N_{0,n}\left(\lambda'_1,\dots,\lambda'_l\right)$ of tropical curves satisfying point conditions and cross-ratio constraints. Moreover, it makes $N_{0,n}\left(\lambda'_1,\dots,\lambda'_l\right)$ computable using the cross-ratio lattice path algorithm. As a consequence the numbers $N^{\textrm{class}}_{0,n}\left( \mu_1,\dots,\mu_l \right)$ (we use notation \ref{notation}) become computable too.

\begin{definition}[Simple tropical curves]\label{definition:simple_curves}
An element $\left( \Gamma,x_1,\dots,x_n,h \right)$ in $\mathcal{M}_{0,n}\left(\mathbb{R}^2,\Delta \right)$ is called \textit{simple} if is satisfies:
\begin{itemize}
\item
the map $h$ that embeds $\Gamma$ in $\mathbb{R}^2$ is injective on vertices,
\item
if $h(v)\in h(e)$ for a vertex $v$ and an edge $e$, then there is an edge $e'$ adjacent to $v$ such that $h(e)$ and $h(e')$ intersect in infinitely many points and then there are a vertex $v'$ and finite sequences $(e_i)^r_i,(e'_j)^{r'}_j$ of edges (with $e_0=e,e'_0=e'$) that lie in $\operatorname{span}(e)$ such that two consecutive elements in a sequence meet in a vertex and such that $h(e_r)$ and $h(e_{r'})$ are adjacent to $h(v')$,
\item
assume $p\in\mathbb{R}^2$ is a point through which more than two edges pass. Divide these edges into equivalence classes depending on the slope of the line they are mapped to. Then there are at most two equivalence classes.
\end{itemize}
\end{definition}

\begin{theorem}\label{theorem:same_numbers}
For notation, see \ref{notation}. The number of rational tropical curves satisfying point and cross-ratio conditions (see definition \ref{definition:general_position_I}) equals the number obtained from the cross-ratio lattice path algorithm (see definition \ref{definition:number_lpa}) if the input data of the algorithm are the number of point conditions and the degenerated cross-ratios. More precisely, the equality
\begin{align*}
N_{0,n}\left(\lambda'_1,\dots,\lambda'_l\right)
=
N_{0,n}^{\textrm{lpa}}\left(\lambda_1,\dots,\lambda_l\right)
\end{align*}
holds.
\end{theorem}

\begin{proof}
Using theorem \ref{thm:ZSFSSG_section_2}, we deduce that $N_{0,n}\left(\lambda'_1,\dots,\lambda'_l\right)$ equals the number of tropical curves satisfying the degenerated cross-ratio conditions $\lambda_1,\dots,\lambda_l$.

Let $\mathcal{S}_{0,n}\left( \lambda_1,\dots,\lambda_l\right)$ denote the set of elements that contribute to $N_{0,n}^{\textrm{lpa}}\left(\lambda_1,\dots,\lambda_l\right)$
As before, we pick points $p_1,\dots,p_n$ in general position linearly ordered on a line with a small negative slope such that distances grow ($|p_{i}-p_{i-1}| <<|p_{i+1}-p_i|$), and we let $\mathcal{R}_{0,n}\left(\lambda_1,\dots,\lambda_l\right)$ denote the set of degenerated tropical curves satisfying degenerated cross-ratio constraints, that is $\mathcal{R}_{0,n}\left(\lambda_1,\dots,\lambda_l\right)$ denotes the set of elements that contribute to $N_{0,n}\left(\lambda_1,\dots,\lambda_l\right)$. Consider the map
\begin{align*}
\phi:\mathcal{S}_{0,n}\left( \lambda_1,\dots,\lambda_l\right)
&\to
\mathcal{R}_{0,n}\left(\lambda_1,\dots,\lambda_l\right)\\
\mathcal{S}&\mapsto C_\mathcal{S}
\end{align*}
that maps a lattice path subdivision $\mathcal{S}$ to its dual tropical curve $C_\mathcal{S}$ given by construction \ref{construction:dual_tropical_curve}. This map is obviously well-defined because we only have subdivisions where all polytopes are fixed and the map is injective because curves with different combinatorial types are different. To see that $\phi$ is surjective, we need to construct a preimage for a given curve $C=(\Gamma,x_1,\dots,x_n,h)$ in $\mathcal{R}_{0,n}\left(\lambda_1,\dots,\lambda_l\right)$. Note that $C$ carries two different graph structures, namely one induced by $\Gamma$ and one induced by $h(\Gamma)$ in the canonical way. If we refer to a vertex in $h(\Gamma)$, we mean the graph structure induced by $h$ and if we refer to a vertex in $\Gamma$, we mean the graph structure of $\Gamma$.

First of all, associate a valid polytope (resp. a pointed segment) to every vertex $v\in h(\Gamma)$: Let $v$ be a vertex of $h(\Gamma)$ and consider its dual polytope $P_v$. The polytope $P_v$ can be turned into a labeled polyotpe (resp. a pointed segment) if we label its edges $E_i$ with weights of its dual edges $e_{i_1},\dots,e_{i_m}\in \Gamma$. Moreover, denote by $\tilde{P}_v$ the dual polytope of $v\in \Gamma$ and label its edges as before. Note that $P_v$ is a Minkowski sum of $\tilde{P}_v$ and segments $S_1,\dots, S_r$ that correspond to edges of $v\in h(\Gamma)$ that are no edges of $v\in \Gamma$. We can choose the points $p_1,\dots,p_n$ in such a way that $C$ is a simple tropical curve. Then, edges of $v\in h(\Gamma)$ that are no edges of $v\in \Gamma$ can only be parallel to edges of $v\in \Gamma$. Furthermore, if $\tilde{P}$ is $0$-dimensional, then there are two segments $S_{i_1},S_{i_2}\in\lbrace S_1,\dots, S_r\rbrace$ such that all other Minkowski summands of $P$ are parallel to one of them. Note also that there are mappings of entries of labeled edges of $P_v$ to its Minkowski summands. In addition $P_v$ is unique because permuting parallel edges of $v\in h(\Gamma)$ leads to the same dual polytope. In this way, we can assign a valid polytope (resp.\ pointed segment) to every vertex $v\in h(\Gamma)$.

The second step is to associate a subdivision $\mathcal{S}_C\in \mathcal{S}_{0,n}\left( \lambda_1,\dots,\lambda_l\right)$ to $C$: The tropical curve $h(\Gamma)$ determines how to glue the polytopes $P_v$ (via maps called $g$ in construction \ref{construction:subdivisions_from_lattice_path}) for all vertices $v\in h(\Gamma)$ together. Note that if two vertices $v,v'\in h(\Gamma)$ are adjacent, then their dual valid polytopes $P_v,P_{v'}$ are compatible. Denote the subdivision obtained this way by $\mathcal{S}_C$. The dual polytopes resp.\ segments associated to the vertices and edges of $h(\Gamma)$  meeting the points $p_1,\dots,p_n$ and non-pointed segment we associate in the obvious way to the edges of $C$ intersecting the line the points $p_1,\dots,p_n$ lie on form a cross-ratio lattice path $\mathcal{A}$. Hence $\mathcal{S}_C$ is a lattice path subdivision whose dual tropical curve is $C$, the genus of $C$ is zero, all polytopes of $\mathcal{S}_C$ are fixed and $\mathcal{S}_C$ fits to the given cross-ratios by definition. Therefore $\mathcal{S}_C\in\mathcal{S}_2(\mathcal{A})$ for some cross-ratio lattice path $\mathcal{A}$.
Thus $\phi$ is bijective and preserves weights.
\end{proof}

Now that we established theorem \ref{theorem:same_numbers}, we can apply corollary \ref{corollary:ZSFSSG_section} and in particular the correspondence theorem shown by Tyomkin in \cite{IlyaCRC} such that the next corollary follows immediately.

\begin{corollary}
We use the notation from \ref{notation}. Under the same assumptions as in theorem \ref{thm:correspondence_thm_CRC} the equality
\begin{align*}
N^{\textrm{class}}_{0,n}\left( \mu_1,\dots,\mu_l \right)
=
N_{0,n}^{\textrm{lpa}}\left(\lambda_1,\dots,\lambda_l\right)
\end{align*}
holds.
\end{corollary}

\section{Floor diagrams for cross-ratio counts}\label{section:special_case:floor_diagrams_cross-ratios}

In this section, we want to impose some restrictions on the degree $\Delta$ and the cross-ratios such that we can work with simple combinatorial objects called floor diagrams. Let $\Sigma_d$ be the convex hull of $\lbrace (0,0),(d,0),(d,0)\rbrace\in\mathbb{R}^2$ for some $d\in\mathbb{N}_{>0}$ and $\Delta_d:=\Delta\left(\Sigma_d\right)$ (see definition \ref{definition:degree}).

\begin{definition}[Cross-ratio floor diagrams]
Let $d\in\mathbb{N}_{>0}$ and let $\mathcal{F}$ be a tree on a totally ordered set of vertices $v_1,\dots,v_n$, then $\mathcal{F}$ is called a \textit{cross-ratio floor diagram of degree $\Delta_d$} if:
\begin{itemize}
\item[(1)] Each edge of $\mathcal{F}$ consists of two half-edges. There are two types of half-edges, \textit{thin} and \textit{thick} ones. A thin half-edge can only be completed to an edge with a thick half-edge and vice versa.
\item[(2)]
Each vertex $v$ is labeled with $s_v,|\lambda_v|\in\mathbb{N}$ and a set $\delta_v$ of labels that appear in $\Delta_d$, where $|\lambda_v|$ is called the \textit{number of cross-ratios of $v$} and $s_v$ is called the \textit{size of $v$} such that 
\begin{align*}
s_v=\lbrace x\in\delta_v\mid d+1\leq x\leq 2d \rbrace = \lbrace x\in\delta_v\mid 2d+1\leq x\leq 3d \rbrace
\end{align*}
and $\emptyset=\delta_v\cap\delta_{v'}$ for all $v\neq v'$ and $\bigcup_v \delta_v$  is the set of all labels appearing in $\Delta_d$.
\item[(3)]
The number of thick edges adjacent to a vertex $v$ is $2-2s_v+|\lambda_v|$.
\item[(4)]
The total ordering on the vertices induces directions on the edges in the following way: we order the vertices on a line starting with the smallest vertex $v_1$ on the left and direct the edges from smaller to larger vertices. Each edge $e$ of the graph is equipped with a \textit{weight} $\omega(e)\in\mathbb{N}$ such that the \textit{balancing condition}
\begin{align*}
s_v-\left(\#\delta_v-2s_v \right)+\sum\pm \omega(e)=0
\end{align*}
holds for all vertices $v$, where the sign is $+$ for outgoing edges and $-$ for incoming edges of $v$.
\end{itemize}
\end{definition}

\begin{definition}\label{definition:floor_diag_satisfies_CRC}
Let $\lambda=\lbrace \beta_1,\dots,\beta_4 \rbrace$ be a degenerated cross-ratio on $\mathcal{M}_{0,n}\left(\mathbb{R}^2,\Delta_d \right)$. Let $\mathcal{F}$ be a floor diagram of degree $\Delta_d$. Each element $\beta_i$ of $\lambda$ is associated to a vertex of $\mathcal{F}$ the following way: 
\begin{itemize}
\item[(1)]
If $\beta_i$ is the label $t\in\lbrace 1,\dots,3d\rbrace$ of an end, then $\beta
_i$ is associated to the unique vertex $v\in\mathcal{F}$ such that $t\in\delta_v$.
\item[(2)]
If $\beta_i$ is a point $x_j\in\lbrace x_1,\dots,x_n\rbrace$, then $\beta_i$ is associated to $v_j$.
\end{itemize}
Hence a pair $\left( \beta_i,\beta_j\right)$ induces a unique path in $\mathcal{F}$. If the paths associated to $\left( \beta_{i_1},\beta_{i_2} \right)$ and $\left( \beta_{i_3},\beta_{i_4} \right)$ intersect in exactly one vertex $v$ of $\mathcal{F}$ for all pairwise different choices of $i_1,\dots,i_4$ such that $\lbrace i_1,\dots,i_4 \rbrace =\lbrace 1,\dots,4 \rbrace$, then the cross-ratio $\lambda$ is \textit{satisfied at $v$}. A cross-ratio floor diagram \textit{satisfies} the degenerated cross-ratios $\lambda_1,\dots,\lambda_l$ if for each cross-ratio there is a vertex of $\mathcal{F}$ satisfying it and $|\lambda_v|$ is exactly the total number of cross-ratios that are satisfied at a vertex $v$ for each vertex.
\end{definition}

\begin{remark}
Note that the condition `all choices of $i_1,\dots,i_4$ lead to exactly one vertex in the intersection of the paths' is equivalent to `one choice of $i_1,\dots,i_4$ leads to exactly one vertex in the intersection of the paths'. This makes it easier to check if $\mathcal{F}$ satisfies a cross-ratio.
\end{remark}

\begin{example}\label{Example:cross_ratio_floor_diagram}
The figure below shows a cross-ratio floor diagram, where all weights on the edges are $1$ and where thick edges are drawn thick. Note that we have $d=3$ and this cross-ratio floor diagram satisfies the cross-ratio $\lambda=\lbrace x_1x_4x_5x_6 \rbrace$.

\begin{figure}[H]
\centering\hspace*{1.5cm}
\def\svgwidth{300pt}
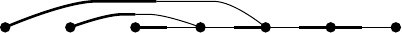
\label{Example_cross_ratio_floor_diagram}

\setlength{\tabcolsep}{1.5em}
\renewcommand{\arraystretch}{1.3} 
\begin{table}[H]
  \centering
      \centering
      \begin{tabular}{*{8}{|c}|}\hline
        $i$ & $1$&$2$&$3$&$4$&$5$&$6$&$7$ \\\hline
        $s_{v_i}$ & $0$ & $0$ & $0$ & $1$ & $1$& $0$& $1$ \\\hline
        $|\lambda_{v_i}|$& $0$ & $0$ & $0$ & $0$ & $1$& $0$& $0$ \\\hline
        $\delta_{v_i}$ & $\lbrace 1 \rbrace$ & $\lbrace 2 \rbrace$ & $\lbrace 3 \rbrace$ & $\lbrace 4,9 \rbrace$ & $\lbrace 5,8 \rbrace$ & $\emptyset$ & $\lbrace 6,7 \rbrace$ \\\hline
      \end{tabular}
\end{table}
\end{figure}
\end{example}

\begin{definition}[$i$-th piece of $\mathcal{F}$]\label{definition:piece_of_floor_diag}
Let $\mathcal{F}$ be a cross-ratio floor diagram of degree $\Delta_d$ on the ordered set of vertices $v_1,\dots,v_n$ corresponding to given point conditions $p_1,\dots,p_n$ such that $\mathcal{F} $satisfies the degenerated cross-ratios $\lambda_1,\dots,\lambda_l$. The \textit{ $i$-th piece} $\left(\mathcal{F}_i,\delta_{v_i},s_{v_i},|\lambda_{v_i}|,\tilde{\lambda}_{i_1},\dots,\tilde{\lambda}_{i_{|\lambda_{v_i}|}}\right)$ (for $i=1,\dots,n$) of $\mathcal{F}$ is obtained from $\mathcal{F}$ in the following way: Cut all edges that connect the vertex $v_i$ to other vertices of $\mathcal{F}$ into (thick or thin) half-edges, and call the connected component containing $v_i$ now $\mathcal{F}_i$, equip the cut edges with the labels indicating the vertices that they used to be connected to.
Moreover, we want to adapt the cross-ratios that are satisfied at $v_i$: If $\lambda=\lbrace \beta_1,\dots,\beta_4 \rbrace$ is a degenerated cross-ratio which is satisfied at $v_i$, the paths associated to $\lambda$ in $\mathcal{F}$ (see definition \ref{definition:floor_diag_satisfies_CRC}) might have been cut by cutting the edges connecting $v_i$ to the rest of $\mathcal{F}$. Let $\beta_j\in\lambda$ be such that the path from the vertex associated to $\beta_j$ to $v_i$ is cut. Replace $\beta_j$ by the label of the edge in the path that has been cut and denote the cross-ratio obtained that way by $\tilde{\lambda}$. We shorten the notation to $\mathcal{F}_i$ if the additional data $\left(\mathcal{F}_i,\delta_{v_i},s_{v_i},|\lambda_{v_i}|,\tilde{\lambda}_{i_1},\dots,\tilde{\lambda}_{i_{|\lambda_{v_i}|}}\right)$ is obvious from the context.
\end{definition}

\begin{definition}[Multiplicities of cross-ratio floor diagrams]\label{definition:multiplicities_cross-ratio_floor_diagrams}
Let $\mathcal{F}$ be a cross-ratio floor diagram of degree $\Delta_d$ on the ordered set of vertices $v_1,\dots,v_n$ that satisfies the degenerated cross-ratios $\lambda_1,\dots,\lambda_l$ and let $p_1,\dots,p_n$ be points in a stretched configuration. Let $\mathcal{F}_i$ be a piece of a floor diagram $\mathcal{F}$ like above. The weighted incoming edges of $\mathcal{F}_i$ induce a partition $\alpha$ of the sum of all weights of incoming edges of $\mathcal{F}_i$ in a natural way and the weighted outgoing edges induce a partition $\beta$, respectively. Let $\kappa$ be the set of labels of thin edges adjacent to $v_i\in\mathcal{F}_i$. The \textit{multiplicity} of the piece $\mathcal{F}_i$ is defined as
\begin{align*}
\mult(\mathcal{F}_i):=\degree \left( \ev^*_i(p_i) \cdot \prod_{k\in\kappa}\partial\ev_k^*\left( y_k \right)\cdot\prod_{j=1}^{|\lambda_{v_i}|}\ft_{\tilde{\lambda}_{i_j}}^*\left( 0\right) \cdot \mathcal{M}_{0,n}\left(\mathbb{R}^2,\Delta\left(\alpha,\beta\right) \right)\right),
\end{align*}
where $\degree$ is the degree of a cycle and $p_i,\lambda_{i_1},\dots,\lambda_{i_{|\lambda_{v_i}|}},\lbrace y_k\mid k\in\kappa\rbrace$ are in general position (cf.\ lemma \ref{lemma:ev_horizontal_ends}). The \textit{multiplicity} of $\mathcal{F}$ is defined as
\begin{align*}
\mult(\mathcal{F}):=\prod_{i=1}^n \mult(\mathcal{F}_i).
\end{align*}
\end{definition}

\begin{definition}[Floors and elevators]
An \textit{elevator} of a tropical curve of degree $\Delta_d$ is an edge that is parallel to $(-1,0)\in\Delta_d$. A connected component of a tropical curve that remains if the interiors of the elevators are removed is called \textit{floor of size $s$} if there are exactly $s$ ends that are in this connected component and that are parallel to $(1,1)\in\Delta_d$. The case $s=0$ is possible for floors consisting of a single contracted marked point. A tropical curve that is fixed by points and cross-ratios is called \textit{floor decomposed} if each point lies on its own floor.
\end{definition}

\begin{definition}\label{definition:4_points_CR}
A cross-ratio $\left(\beta_1\beta_2|\beta_3\beta_4\right)$ is said to have \textit{$t$ points} if the number of $\beta_i$ that are points is $t$. A set of cross-ratios $\lambda_1,\dots,\lambda_l$ has \textit{$t$ points} if each cross-ratio in the set does.
\end{definition}

\begin{lemma}\label{lemma:floor_decomposed_4_points}
A tropical curve $C$ of degree $\Delta_d$ that is fixed by general positioned point conditions $p_1,\dots,p_n$ and degenerated cross-ratio constraints $\lambda_1,\dots,\lambda_l$ that have $4$ points such that the $y$-coordinates of the points $p_1,\dots,p_n$ are contained in a small interval while the $x$-coordinates have large distances is floor decomposed.
\end{lemma}

\begin{proof}
A \textit{string} is a path in a tropical curve connecting two non-contracted ends such that no point lies on that path. A string gives rise to a $1$-dimensional family of tropical curves.
Let $I\subset \mathbb{R}$ be a compact interval such that $p_1,\dots,p_n$ lie in the stripe $\mathbb{R}\times I$ of $\mathbb{R}^2$. Assume there is a vertex $v$ of $C$ whose $y$-coordinate (among all vertices of $C$) is (without loss of generality) maximal and $v$ lies above the stripe. There are two cases.
\begin{itemize}
\item[(1)]
Assume $v$ has valency greater $3$, that is there are cross-ratios such that $\val(v)=3+\#\lambda_v$ (see definition \ref{definition:lambda_v_vertex_trop_curve}). By the balancing condition there is an edge adjacent to $v$ whose direction vector has $y$-coordinate greater zero. But this edge cannot lead to a point since all points lie beneath $v$ and $v$ has maximal $y$-coordinate. This contradicts lemma \ref{lemma:nice_property} since all cross-ratios have $4$ points.
\item[(2)]
Assume $v$ is $3$-valent. We follow the proof of proposition 5.3 of \cite{MikhalkinBrugalle}: Since the $y$-coordinate of $v$ is maximal there is an edge $e_1$ that is an end with direction vector $u_1$ adjacent to $v$. The given degree $\Delta_d$ guarantees that $u_1=(\alpha,1)$ for some $\alpha$. Denote the two other direction vectors by $u_2,u_3$. Using the balancing condition, we can (without loss of generality) write $u_2=(\gamma,\beta)$ and $u_3=(\epsilon,\delta)$ for some integers $\beta\geq 0,\delta<0$. Note that the edge $e_2$ associated to $u_2$ is an end if $\beta>0$ and this leads to a string from $e_1$ to $e_2$ which is a contradiction. Therefore $\beta=0$ and $e_2$ is no end. Let $v'$ be the vertex to which $v$ is connected to via $e_2$. By case (1) $v'$ is also $3$-valent, and $v'$ is (by the balancing condition) adjacent to an end denoted by $e_1'$. Thus there is a string from $e_1$ to $e_1'$ which is a contradiction.
\end{itemize}

Since no vertex of $C$ lies outside the stripe $\mathbb{R}\times I$, corollary 5.4 of \cite{MikhalkinBrugalle} can be applied, which yields that $C$ is floor decomposed.
\end{proof}

Assume in the following that all cross-ratios have $4$ points.

\begin{construction}[Floor decomposed curve $\mapsto$ cross-ratio floor diagram]\label{construction:floor_decomposed_curve_to_CR_floor_diag}
Let $\Delta_d$ be a degree, let $p_1,\dots,p_n,\lambda_1,\dots,\lambda_l$ be in general position, where $p_1,\dots,p_n\in\mathbb{R}^2$ are points in a stretched configuration, $\lambda_1,\dots,\lambda_l$ are degenerated cross-ratios with  $4$ points such that $3d-1=n+l$ holds. Curves satisfying these conditions are floor decomposed by lemma \ref{lemma:floor_decomposed_4_points}. We obtain a cross-ratio floor diagram $\mathcal{F}_C$ the following way: Cut all elevators of $C$, that is cut all edges parallel to $(1,0)\in\mathbb{R}^2$ such that each remaining component contains exactly one point. Shrinking these components to points $v_i$ we get the vertices of $\mathcal{F}_C$. We connect  $v_i,v_j\in\mathcal{F}_C$ if and only if the components obtained from $p_i,p_j$ are connected by an elevator. Distribute the conditions $\lambda_1,\dots,\lambda_l$ to the components analogous to definition \ref{definition:piece_of_floor_diag}.
We draw half-edges thin if they lead to a fixed component, and thick if they lead to a free component (see definition \ref{definition:fixed_free_components}).
We set
\begin{align*}
|\lambda_{v_i}|:=\sum_u \#\lambda_u,
\end{align*}
where the sum runs over all vertices $u$ in the component  of $p_i$ where $\lambda_u$ is introduced in definition \ref{definition:lambda_v_vertex_trop_curve}, $s_{v_i}$ is the size of the component associated to $p_i$ and $\delta_{v_i}$ is the set of labels of ends in $\Delta_d$ that are adjacent to the component associated to $p_i$ by cutting. Finally, the balancing condition of $C$ turns $\mathcal{F}_C$ into a cross-ratio floor diagram.
\end{construction}

\begin{example}
In order to illustrate construction \ref{construction:floor_decomposed_curve_to_CR_floor_diag}, a tropical curve (see Figure \ref{Example_construction_cross_ratio_floor_diagram2}) of degree $d=3$ through points $p_1,\dots,p_7$ in a stretched configuration satisfying the cross-ratio $\lambda=\lbrace x_1x_4x_5x_6\rbrace$ is given such that this curve is by construction \ref{construction:floor_decomposed_curve_to_CR_floor_diag} associated to the cross-ratio floor diagram of example \ref{Example:cross_ratio_floor_diagram}. The floors of the curve are indicated by dotted lines.
\end{example}

\begin{figure}
\centering
\def\svgwidth{300pt}
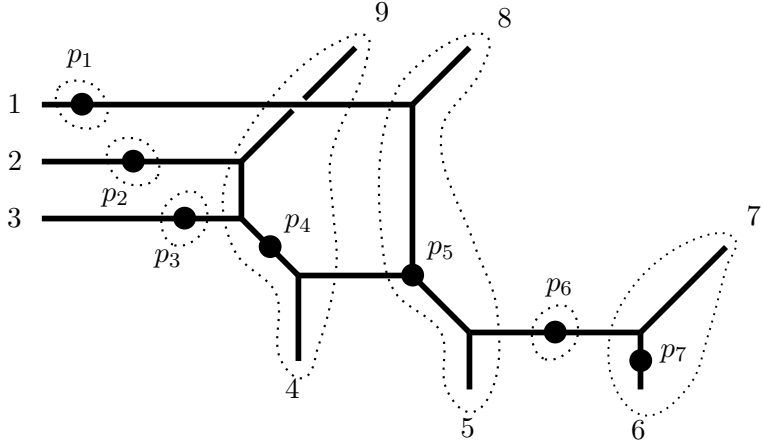
\caption{A floor decomposed curve.}
\label{Example_construction_cross_ratio_floor_diagram2}
\end{figure}

\begin{definition}\label{definition:number_floor}
Let $d\in\mathbb{R}_{>0}$ and let $\lambda_1,\dots,\lambda_l$ be general positioned degenerated cross-ratios with $4$ points. We define
\begin{align*}
N_{0,n}^{\textrm{floor}}\left(\lambda_1,\dots,\lambda_l\right)
:=\sum_{\mathcal{F}} \mult(\mathcal{F}),
\end{align*}
where the sum runs over all cross-ratio floor diagrams of degree $\Delta_d$ on an ordered set of vertices $v_1,\dots,v_n$ that satisfy $\lambda_1,\dots,\lambda_l$.
\end{definition}

\begin{lemma}\label{lemma:vertex_only_thick_edges}
Let $G$ be a tree such that each edge of $G$ consists of two half-edges and there are two types of half-edges, thin and thick ones. A thin half-edge can only be completed to an edge with a thick half-edge and vice versa. There is a vertex of $G$ that is only adjacent to thick half-edges.
\end{lemma}

\begin{proof}
This can be shown by induction over the number $n$ of vertices of $G$. For $n=2$ it is obviously true. If $n>2$, there is a $1$-valent vertex $v$ of $G$ since $G$ is a tree. There are two cases: either $v$ is adjacent to a thick half-edge, then we are done or $v$ is adjacent to a thin half-edge. If $v$ is adjacent to a thin half-edge, then remove this edge and $v$ from $G$. The graph $G'$ obtained this way has one vertex less than $G$ such that there is a vertex $v'\in G'$ that is only adjacent to thick half-edges. Again there are two cases: if $v'$ is not connected to $v$ in $G$, then we are done. Otherwise, the edge connecting $v'$ to $v$ in $G$ is thick at $v'$ since it is thin at $v$.
\end{proof}

\begin{theorem}\label{theorem:CR_count=floor_diag_count}
For notation, see \ref{notation}. Let $d\in\mathbb{N}_{>0}$ and let $\Delta_d$ be its associated degree. The number of rational tropical curves satisfying point and cross-ratio conditions (see definition \ref{definition:general_position_I}) equals the number obtained from counting floor diagrams (see definition \ref{definition:number_floor}). More precisely, the equality
\begin{align*}
N_{0,n}\left(\lambda'_1,\dots,\lambda'_l\right)
=
N_{0,n}^{\textrm{floor}}\left(\lambda_1,\dots,\lambda_l\right)
\end{align*}
holds.
\end{theorem}
\begin{proof}
We use theorem \ref{thm:ZSFSSG_section_2} showing that $N_{0,n}\left(\lambda'_1,\dots,\lambda'_l\right)$ equals the number of tropical curves satisfying the degenerated cross-ratio conditions $\lambda_1,\dots,\lambda_l$.

Let $p_1,\dots,p_n\in\mathbb{R}^2$ be points as in lemma \ref{lemma:floor_decomposed_4_points}. Let $\mathcal{R}_{0,n}\left(\lambda_1,\dots,\lambda_l\right)$ denote the set of degenerated tropical curves satisfying degenerated cross-ratio constraints, that is $\mathcal{R}_{0,n}\left(\lambda_1,\dots,\lambda_l\right)$ denotes the set of elements that contribute to $N_{0,n}\left(\lambda_1,\dots,\lambda_l\right)$.
Then all curves in $\mathcal{R}_{0,n}\left(\lambda_1,\dots,\lambda_l\right)$ are floor decomposed. Let $C$ be such a curve. By construction \ref{construction:floor_decomposed_curve_to_CR_floor_diag} there is a cross-ratio floor diagram $\mathcal{F}_C$ associated to $C$. Recall that all weights are local (see theorem \ref{thm:ZSFSSG_section_2}), hence $\mathcal{F}_C$ contributes to $N_{0,n}^{\textrm{floor}}\left(\lambda_1,\dots,\lambda_l\right)$ since cutting $C$ along its elevators yields $\mult({\mathcal{F}_C}_i)\neq 0$ for all pieces of $\mathcal{F}_C$.

Let $\mathcal{F}_{0,n}\left(\lambda_1,\dots,\lambda_l\right)$ denote the set of elements that contribute to $N_{0,n}^{\textrm{floor}}\left(\lambda_1,\dots,\lambda_l\right)$. The arguments above show that
\begin{align*}
\phi:\mathcal{R}_{0,n}\left(\lambda_1,\dots,\lambda_l\right)&\to \mathcal{F}_{0,n}\left(\lambda_1,\dots,\lambda_l\right)\\
C&\mapsto \mathcal{F}_C
\end{align*}
is a well-defined map. We want to show that $\phi$ is onto by constructing preimages. Let $\mathcal{F}\in\mathcal{F}_{0,n}\left(\lambda_1,\dots,\lambda_l\right)$. Using lemma \ref{lemma:vertex_only_thick_edges}, there is a vertex $v_i$ of $\mathcal{F}$ such that $v_i$ is only adjacent to thick half-edges. Let $\left(\mathcal{F}_i,\delta_{v_i},s_{v_i},|\lambda_{v_i}|,\tilde{\lambda}_{i_1},\dots,\tilde{\lambda}_{i_{|\lambda_{v_i}|}}\right)$ be the piece of $\mathcal{F}$ that includes $v_i$. The weighted incoming elevators and ends of $\mathcal{F}_i$ induce an unordered partition $\alpha^{(i)}$ and the weighted outgoing elevators and ends of $\mathcal{F}_i$ induce $\beta^{(i)}$, respectively. Since $\mult(\mathcal{F}_i)\neq 0$ there is a curve $C_i\in\ev^*_i(p_i) \cdot \prod_{j=1}^{|\lambda_{v_i}|}\ft_{\tilde{\lambda}_{i_j}}^*\left( 0\right) \cdot \mathcal{M}_{0,n}\left(\mathbb{R}^2,\Delta\left(\alpha^{(i)},\beta^{(i)}\right) \right)$ (see definition \ref{definition:degree}) that is fixed by $p_i,\tilde{\lambda}_{i_1},\dots,\tilde{\lambda}_{i_{|\lambda_{v_i}|}}$. Remove $v_i$ and its adjacent edges from $\mathcal{F}$. The resulting graph might be disconnected. Let $K$ be a component of this graph. Using lemma \ref{lemma:vertex_only_thick_edges}, there is a vertex $v_j$ of $K$ such that $v_j$ is only adjacent to thick half-edges. There are two cases:
\begin{itemize}
\item[(1)]
If $v_j\in\mathcal{F}$ is only adjacent to thick half-edges, then associated a curve $C_j$ to $v_j$ like we did before for $v_i$.
\item[(2)]
There is an edge $e$ in $\mathcal{F}$ that connects $v_i$ and $v_j$ such that the thick half-edge of $e$ is adjacent to $v_i$. Let $y_e\in\mathbb{R}$ be the height of the horizontal end associated to $e$ in $C_i$. Now that we fixed that height, we can argue like before: Let $\left(\mathcal{F}_j,\delta_{v_j},s_{v_j},|\lambda_{v_j}|,\tilde{\lambda}_{j_1},\dots,\tilde{\lambda}_{j_{|\lambda_{v_i}|}}\right)$ be the piece of $\mathcal{F}$ that includes $v_j$. The weighted incoming elevators and ends of $\mathcal{F}_j$ induce $\alpha^{(j)}$ and $\beta^{(j)}$ as before. Since $\mult(\mathcal{F}_j)\neq 0$ there is a curve $C_j\in\ev^*_j(p_j) \cdot \partial\ev_e^*\left( y_e \right) \cdot\prod_{z=1}^{|\lambda_{v_j}|}\ft_{\tilde{\lambda}_{z_j}}^*\left( 0\right) \cdot \mathcal{M}_{0,n}\left(\mathbb{R}^2,\Delta\left(\alpha^{(j)},\beta^{(j)}\right) \right)$ that is fixed by $p_i,\tilde{\lambda}_{i_1},\dots,\tilde{\lambda}_{i_{|\lambda_{v_i}|}}$.
\end{itemize}
Iterating this procedure gives us a curve $C_t$ for each piece $\mathcal{F}_t$ of $\mathcal{F}$ such that $C_1,\dots,C_n$ can be glued together by construction. Denote the curve obtained from this glueing by $C$. The multiplicity of $C$ is given by
\begin{align*}
\mult(C)=\prod_{t=1}^n \mult(C_t)
\end{align*}
because of theorem \ref{thm:ZSFSSG_section_2}. Therefore $C\in\phi^{-1}(\mathcal{F})$.

Note that the procedure above does not depend on the choice of $C_t$ we associated to each $\mathcal{F}_t$. Hence
\begin{align*}
\mult(\mathcal{F})=\sum_{C\in\phi^{-1}(\mathcal{F})} \mult(C)
\end{align*}
holds.
\end{proof}

We can now apply corollary \ref{corollary:ZSFSSG_section} and  the correspondence theorem \ref{thm:correspondence_thm_CRC} such that the next corollary follows immediately.

\begin{corollary}
We use the notation from \ref{notation}. If we require in addition to the assumptions of theorem \ref{thm:correspondence_thm_CRC} that every cross-ratio has $4$ points (see definition \ref{definition:4_points_CR}) and that the given degree is $\Delta_d$, the equality
\begin{align*}
N^{\textrm{class}}_{0,n}\left( \mu_1,\dots,\mu_l \right)
=
N_{0,n}^{\textrm{floor}}\left(\lambda_1,\dots,\lambda_l\right)
\end{align*}
holds.
\end{corollary}

\begin{remark}
The results of this section are not restricted to degree $\Delta_d$ curves and can be generalized to Hirzebruch surfaces or other  surfaces with $h$-transverse polytopes (see \cite{hTransverse}) since the cross-ratio floor diagram techniques can be extended to these degrees in a straight-forward way.
\end{remark}

\begin{example}
Fix the degree $\Delta_3$ (i.e. $d=3$), let $p_1,\dots,p_7$ be points and let $\lambda=\lbrace x_1,\dots,x_4\rbrace$ be a degenerated cross-ratio. We want to determine the number $N_{0,7}\left(\lambda\right)$ using floor diagrams. For that draw all floor diagrams of degree $\Delta_3$ on $7$ vertices that satisfy the cross-ratio $\lambda$. Since we have $7$ points, there are no floors of size $3$ or $2$. Figure \ref{Example_complete_flor_diagram} shows all possible floor diagrams. Note that in this example we do not need all discrete data a floor diagram is equipped with, i.e. floors of size $1$ are drawn white and floors of size $0$ are drawn black (instead of specifying $s_{v_i}$ for each floor), the number of cross-ratios satisfied at each floor is obvious (we only have one cross-ratio) and the labels of ends adjacent to each floor are dropped here, so we need to add a factor of $(d!)^3$ to the final count. By considering the multiplicities of each piece $\mathcal{F}_i$ of a floor diagram $\mathcal{F}$ in Figure \ref{Example_complete_flor_diagram}, we end up with multiplicity $1$ for all floor diagrams shown in Figure \ref{Example_complete_flor_diagram}. Hence 
\begin{align*}
N_{0,7}\left(\lambda\right)&=4*(3!)^3=864.
\end{align*}
Note that this number is not the same as in example \ref{example:Kontsevich_CR_2} because we considered a cross-ratio with $4$ points here, whereas we considered a cross-ratio with $2$ points in example \ref{example:Kontsevich_CR_2}.

\begin{figure}[H]
\centering
\def\svgwidth{230pt}
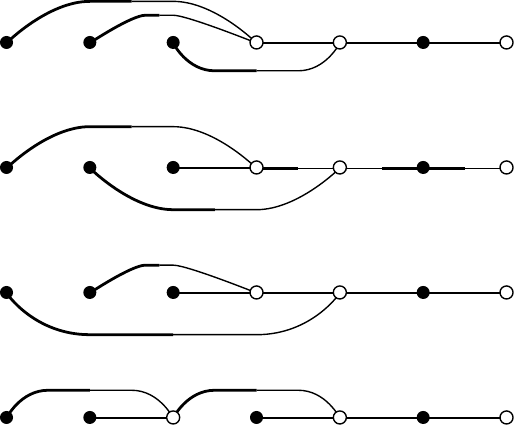
\caption{Floor diagrams with floors of size $0$ (black) and $1$ (white).}
\label{Example_complete_flor_diagram}
\end{figure}

\end{example}

\hyphenation{Kaisers-lautern} 
\bibliographystyle{alpha}
\bibliography{literatur}

\end{document}

%% file: Example_Introduction.pdf_tex
\begingroup%
  \makeatletter%
  \providecommand\color[2][]{%
    \errmessage{(Inkscape) Color is used for the text in Inkscape, but the package 'color.sty' is not loaded}%
    \renewcommand\color[2][]{}%
  }%
  \providecommand\transparent[1]{%
    \errmessage{(Inkscape) Transparency is used (non-zero) for the text in Inkscape, but the package 'transparent.sty' is not loaded}%
    \renewcommand\transparent[1]{}%
  }%
  \providecommand\rotatebox[2]{#2}%
  \ifx\svgwidth\undefined%
    \setlength{\unitlength}{177.30455504bp}%
    \ifx\svgscale\undefined%
      \relax%
    \else%
      \setlength{\unitlength}{\unitlength * \real{\svgscale}}%
    \fi%
  \else%
    \setlength{\unitlength}{\svgwidth}%
  \fi%
  \global\let\svgwidth\undefined%
  \global\let\svgscale\undefined%
  \makeatother%
  \begin{picture}(1,0.33941192)%
    \put(0,0){\includegraphics[width=\unitlength,page=1]{Example_Introduction.pdf}}%
    \put(-0.0437806,0.29609246){\color[rgb]{0,0,0}\makebox(0,0)[lt]{\begin{minipage}{0.06345015\unitlength}\raggedright \end{minipage}}}%
    \put(-0.00170743,0.24094979){\color[rgb]{0,0,0}\makebox(0,0)[lb]{\smash{$1$}}}%
    \put(-0.00121331,0.15672568){\color[rgb]{0,0,0}\makebox(0,0)[lb]{\smash{$2$}}}%
    \put(0.22116379,0.00397665){\color[rgb]{0,0,0}\makebox(0,0)[lb]{\smash{$3$}}}%
    \put(0.39178193,0.00435276){\color[rgb]{0,0,0}\makebox(0,0)[lb]{\smash{$4$}}}%
    \put(0.4926425,0.19498676){\color[rgb]{0,0,0}\makebox(0,0)[lb]{\smash{$5$}}}%
    \put(0.23855247,0.32167524){\color[rgb]{0,0,0}\makebox(0,0)[lb]{\smash{$6$}}}%
    \put(0.63288082,0.24019751){\color[rgb]{0,0,0}\makebox(0,0)[lb]{\smash{$1$}}}%
    \put(0.63291205,0.15710179){\color[rgb]{0,0,0}\makebox(0,0)[lb]{\smash{$2$}}}%
    \put(0.84316604,0.00435276){\color[rgb]{0,0,0}\makebox(0,0)[lb]{\smash{$3$}}}%
    \put(0.87108197,0.00435276){\color[rgb]{0,0,0}\makebox(0,0)[lb]{\smash{$4$}}}%
    \put(0.95849303,0.19498676){\color[rgb]{0,0,0}\makebox(0,0)[lb]{\smash{$5$}}}%
    \put(0.87371918,0.32205126){\color[rgb]{0,0,0}\makebox(0,0)[lb]{\smash{$6$}}}%
    \put(0.08335568,0.27532433){\color[rgb]{0,0,0}\makebox(0,0)[lb]{\smash{$q_1$}}}%
    \put(0.16347108,0.30842506){\color[rgb]{0,0,0}\makebox(0,0)[lb]{\smash{$q_2$}}}%
    \put(0.71759905,0.27533339){\color[rgb]{0,0,0}\makebox(0,0)[lb]{\smash{$q_1$}}}%
    \put(0.79749197,0.30914935){\color[rgb]{0,0,0}\makebox(0,0)[lb]{\smash{$q_2$}}}%
  \end{picture}%
\endgroup%

%% file: Example_M_0_4.pdf_tex
\begingroup%
  \makeatletter%
  \providecommand\color[2][]{%
    \errmessage{(Inkscape) Color is used for the text in Inkscape, but the package 'color.sty' is not loaded}%
    \renewcommand\color[2][]{}%
  }%
  \providecommand\transparent[1]{%
    \errmessage{(Inkscape) Transparency is used (non-zero) for the text in Inkscape, but the package 'transparent.sty' is not loaded}%
    \renewcommand\transparent[1]{}%
  }%
  \providecommand\rotatebox[2]{#2}%
  \ifx\svgwidth\undefined%
    \setlength{\unitlength}{71.24637467bp}%
    \ifx\svgscale\undefined%
      \relax%
    \else%
      \setlength{\unitlength}{\unitlength * \real{\svgscale}}%
    \fi%
  \else%
    \setlength{\unitlength}{\svgwidth}%
  \fi%
  \global\let\svgwidth\undefined%
  \global\let\svgscale\undefined%
  \makeatother%
  \begin{picture}(1,0.65521016)%
    \put(0,0){\includegraphics[width=\unitlength,page=1]{Example_M_0_4.pdf}}%
    \put(0.01077173,0.65921715){\color[rgb]{0,0,0}\makebox(0,0)[lt]{\begin{minipage}{0.06316111\unitlength}\raggedright $1$\end{minipage}}}%
    \put(0.01216852,0.53125327){\color[rgb]{0,0,0}\makebox(0,0)[lt]{\begin{minipage}{0.06316111\unitlength}\raggedright $2$\end{minipage}}}%
    \put(0.27346772,0.64869034){\color[rgb]{0,0,0}\makebox(0,0)[lt]{\begin{minipage}{0.06316111\unitlength}\raggedright $3$\end{minipage}}}%
    \put(0.27369811,0.54131935){\color[rgb]{0,0,0}\makebox(0,0)[lt]{\begin{minipage}{0.06316111\unitlength}\raggedright $4$\end{minipage}}}%
    \put(0,0){\includegraphics[width=\unitlength,page=2]{Example_M_0_4.pdf}}%
    \put(0.68331168,0.38365498){\color[rgb]{0,0,0}\makebox(0,0)[lt]{\begin{minipage}{0.06316111\unitlength}\raggedright $1$\end{minipage}}}%
    \put(0.68372122,0.25739465){\color[rgb]{0,0,0}\makebox(0,0)[lt]{\begin{minipage}{0.06316111\unitlength}\raggedright $3$\end{minipage}}}%
    \put(0.94625016,0.3815446){\color[rgb]{0,0,0}\makebox(0,0)[lt]{\begin{minipage}{0.06316111\unitlength}\raggedright $2$\end{minipage}}}%
    \put(0.94835263,0.26668528){\color[rgb]{0,0,0}\makebox(0,0)[lt]{\begin{minipage}{0.06316111\unitlength}\raggedright $4$\end{minipage}}}%
    \put(0,0){\includegraphics[width=\unitlength,page=3]{Example_M_0_4.pdf}}%
    \put(0.11766743,0.19603578){\color[rgb]{0,0,0}\makebox(0,0)[lt]{\begin{minipage}{0.06316111\unitlength}\raggedright $1$\end{minipage}}}%
    \put(0.11720666,0.06807187){\color[rgb]{0,0,0}\makebox(0,0)[lt]{\begin{minipage}{0.06316111\unitlength}\raggedright $2$\end{minipage}}}%
    \put(0.27440377,0.19486935){\color[rgb]{0,0,0}\makebox(0,0)[lt]{\begin{minipage}{0.06316111\unitlength}\raggedright $3$\end{minipage}}}%
    \put(0.27276207,0.0772019){\color[rgb]{0,0,0}\makebox(0,0)[lt]{\begin{minipage}{0.06316111\unitlength}\raggedright $4$\end{minipage}}}%
  \end{picture}%
\endgroup%

%% file: Example_resolving_not_unique.pdf_tex
\begingroup%
  \makeatletter%
  \providecommand\color[2][]{%
    \errmessage{(Inkscape) Color is used for the text in Inkscape, but the package 'color.sty' is not loaded}%
    \renewcommand\color[2][]{}%
  }%
  \providecommand\transparent[1]{%
    \errmessage{(Inkscape) Transparency is used (non-zero) for the text in Inkscape, but the package 'transparent.sty' is not loaded}%
    \renewcommand\transparent[1]{}%
  }%
  \providecommand\rotatebox[2]{#2}%
  \ifx\svgwidth\undefined%
    \setlength{\unitlength}{136.83423923bp}%
    \ifx\svgscale\undefined%
      \relax%
    \else%
      \setlength{\unitlength}{\unitlength * \real{\svgscale}}%
    \fi%
  \else%
    \setlength{\unitlength}{\svgwidth}%
  \fi%
  \global\let\svgwidth\undefined%
  \global\let\svgscale\undefined%
  \makeatother%
  \begin{picture}(1,0.20822813)%
    \put(0,0){\includegraphics[width=\unitlength,page=1]{Example_resolving_not_unique.pdf}}%
    \put(0.37782489,0.13440524){\color[rgb]{0,0,0}\makebox(0,0)[lt]{\begin{minipage}{0.04513679\unitlength}\raggedright $5$\end{minipage}}}%
    \put(-0.00087253,0.05564732){\color[rgb]{0,0,0}\makebox(0,0)[lt]{\begin{minipage}{0.05465959\unitlength}\raggedright $3$\end{minipage}}}%
    \put(0.37426789,0.22202865){\color[rgb]{0,0,0}\makebox(0,0)[lt]{\begin{minipage}{0.04672392\unitlength}\raggedright $4$\end{minipage}}}%
    \put(0.00310807,0.21494474){\color[rgb]{0,0,0}\makebox(0,0)[lt]{\begin{minipage}{0.06418239\unitlength}\raggedright $1$\end{minipage}}}%
    \put(0.37279041,0.05564586){\color[rgb]{0,0,0}\makebox(0,0)[lt]{\begin{minipage}{0.08387885\unitlength}\raggedright $6$\end{minipage}}}%
    \put(0.00117898,0.13532267){\color[rgb]{0,0,0}\makebox(0,0)[lt]{\begin{minipage}{0.09403069\unitlength}\raggedright $2$\end{minipage}}}%
    \put(0,0){\includegraphics[width=\unitlength,page=2]{Example_resolving_not_unique.pdf}}%
    \put(0.96247388,0.13440524){\color[rgb]{0,0,0}\makebox(0,0)[lt]{\begin{minipage}{0.04513679\unitlength}\raggedright $5$\end{minipage}}}%
    \put(0.58377646,0.05564732){\color[rgb]{0,0,0}\makebox(0,0)[lt]{\begin{minipage}{0.05465959\unitlength}\raggedright $4$\end{minipage}}}%
    \put(0.95891688,0.22202865){\color[rgb]{0,0,0}\makebox(0,0)[lt]{\begin{minipage}{0.04672392\unitlength}\raggedright $3$\end{minipage}}}%
    \put(0.58775706,0.21494474){\color[rgb]{0,0,0}\makebox(0,0)[lt]{\begin{minipage}{0.06418239\unitlength}\raggedright $1$\end{minipage}}}%
    \put(0.9574394,0.05564586){\color[rgb]{0,0,0}\makebox(0,0)[lt]{\begin{minipage}{0.08387885\unitlength}\raggedright $6$\end{minipage}}}%
    \put(0.58582797,0.13532267){\color[rgb]{0,0,0}\makebox(0,0)[lt]{\begin{minipage}{0.09403069\unitlength}\raggedright $2$\end{minipage}}}%
  \end{picture}%
\endgroup%

%% file: Example_resolving_not_unique_2.pdf_tex
\begingroup%
  \makeatletter%
  \providecommand\color[2][]{%
    \errmessage{(Inkscape) Color is used for the text in Inkscape, but the package 'color.sty' is not loaded}%
    \renewcommand\color[2][]{}%
  }%
  \providecommand\transparent[1]{%
    \errmessage{(Inkscape) Transparency is used (non-zero) for the text in Inkscape, but the package 'transparent.sty' is not loaded}%
    \renewcommand\transparent[1]{}%
  }%
  \providecommand\rotatebox[2]{#2}%
  \ifx\svgwidth\undefined%
    \setlength{\unitlength}{136.83423923bp}%
    \ifx\svgscale\undefined%
      \relax%
    \else%
      \setlength{\unitlength}{\unitlength * \real{\svgscale}}%
    \fi%
  \else%
    \setlength{\unitlength}{\svgwidth}%
  \fi%
  \global\let\svgwidth\undefined%
  \global\let\svgscale\undefined%
  \makeatother%
  \begin{picture}(1,0.29007899)%
    \put(0,0){\includegraphics[width=\unitlength,page=1]{Example_resolving_not_unique_2.pdf}}%
    \put(-0.00087253,0.05564732){\color[rgb]{0,0,0}\makebox(0,0)[lt]{\begin{minipage}{0.05465959\unitlength}\raggedright $5$\end{minipage}}}%
    \put(0.38596087,0.25710759){\color[rgb]{0,0,0}\makebox(0,0)[lt]{\begin{minipage}{0.04672392\unitlength}\raggedright $2$\end{minipage}}}%
    \put(0.00310807,0.28510262){\color[rgb]{0,0,0}\makebox(0,0)[lt]{\begin{minipage}{0.06418239\unitlength}\raggedright $3$\end{minipage}}}%
    \put(0.38448339,0.0907248){\color[rgb]{0,0,0}\makebox(0,0)[lt]{\begin{minipage}{0.08387885\unitlength}\raggedright $6$\end{minipage}}}%
    \put(0.00117898,0.13532267){\color[rgb]{0,0,0}\makebox(0,0)[lt]{\begin{minipage}{0.09403069\unitlength}\raggedright $1$\end{minipage}}}%
    \put(0,0){\includegraphics[width=\unitlength,page=2]{Example_resolving_not_unique_2.pdf}}%
    \put(0.96247388,0.13440524){\color[rgb]{0,0,0}\makebox(0,0)[lt]{\begin{minipage}{0.04513679\unitlength}\raggedright $2$\end{minipage}}}%
    \put(0.58377646,0.09072626){\color[rgb]{0,0,0}\makebox(0,0)[lt]{\begin{minipage}{0.05465959\unitlength}\raggedright $5$\end{minipage}}}%
    \put(0.95891688,0.21033567){\color[rgb]{0,0,0}\makebox(0,0)[lt]{\begin{minipage}{0.04672392\unitlength}\raggedright $4$\end{minipage}}}%
    \put(0.58775706,0.25002368){\color[rgb]{0,0,0}\makebox(0,0)[lt]{\begin{minipage}{0.06418239\unitlength}\raggedright $1$\end{minipage}}}%
    \put(0.9574394,0.05564586){\color[rgb]{0,0,0}\makebox(0,0)[lt]{\begin{minipage}{0.08387885\unitlength}\raggedright $6$\end{minipage}}}%
    \put(0,0){\includegraphics[width=\unitlength,page=3]{Example_resolving_not_unique_2.pdf}}%
    \put(0.00310807,0.21494474){\color[rgb]{0,0,0}\makebox(0,0)[lt]{\begin{minipage}{0.06418239\unitlength}\raggedright $4$\end{minipage}}}%
    \put(0.95891688,0.30387951){\color[rgb]{0,0,0}\makebox(0,0)[lt]{\begin{minipage}{0.04672392\unitlength}\raggedright $3$\end{minipage}}}%
  \end{picture}%
\endgroup%

%% file: Proof_structure_of_X_notation.pdf_tex
\begingroup%
  \makeatletter%
  \providecommand\color[2][]{%
    \errmessage{(Inkscape) Color is used for the text in Inkscape, but the package 'color.sty' is not loaded}%
    \renewcommand\color[2][]{}%
  }%
  \providecommand\transparent[1]{%
    \errmessage{(Inkscape) Transparency is used (non-zero) for the text in Inkscape, but the package 'transparent.sty' is not loaded}%
    \renewcommand\transparent[1]{}%
  }%
  \providecommand\rotatebox[2]{#2}%
  \ifx\svgwidth\undefined%
    \setlength{\unitlength}{219.41869246bp}%
    \ifx\svgscale\undefined%
      \relax%
    \else%
      \setlength{\unitlength}{\unitlength * \real{\svgscale}}%
    \fi%
  \else%
    \setlength{\unitlength}{\svgwidth}%
  \fi%
  \global\let\svgwidth\undefined%
  \global\let\svgscale\undefined%
  \makeatother%
  \begin{picture}(1,0.20705087)%
    \put(0,0){\includegraphics[width=\unitlength,page=1]{Proof_structure_of_X_notation.pdf}}%
    \put(0.00027984,0.16733373){\color[rgb]{0,0,0}\makebox(0,0)[lt]{\begin{minipage}{0.02913805\unitlength}\raggedright $\beta_1$\end{minipage}}}%
    \put(-0.00063893,0.0761478){\color[rgb]{0,0,0}\makebox(0,0)[lt]{\begin{minipage}{0.04002552\unitlength}\raggedright $\beta_2$\end{minipage}}}%
    \put(0.27995935,0.1571681){\color[rgb]{0,0,0}\makebox(0,0)[lt]{\begin{minipage}{0.02814828\unitlength}\raggedright $\beta_3$\end{minipage}}}%
    \put(0.3641439,0.18556372){\color[rgb]{0,0,0}\makebox(0,0)[lt]{\begin{minipage}{0.02913805\unitlength}\raggedright $\beta_1$\end{minipage}}}%
    \put(0.39876438,0.07633286){\color[rgb]{0,0,0}\makebox(0,0)[lt]{\begin{minipage}{0.04002552\unitlength}\raggedright $\beta_2$\end{minipage}}}%
    \put(0.64603511,0.16903625){\color[rgb]{0,0,0}\makebox(0,0)[lt]{\begin{minipage}{0.02814828\unitlength}\raggedright $\beta_3$\end{minipage}}}%
    \put(0.97336518,0.12999909){\color[rgb]{0,0,0}\makebox(0,0)[lt]{\begin{minipage}{0.02814828\unitlength}\raggedright $\beta_3$\end{minipage}}}%
    \put(0.2485636,0.12006918){\color[rgb]{0,0,0}\makebox(0,0)[lt]{\begin{minipage}{0.0340869\unitlength}\raggedright $\beta_4$\end{minipage}}}%
    \put(0.61316335,0.1036807){\color[rgb]{0,0,0}\makebox(0,0)[lt]{\begin{minipage}{0.0340869\unitlength}\raggedright $\beta_4$\end{minipage}}}%
    \put(0.73720128,0.08088396){\color[rgb]{0,0,0}\makebox(0,0)[lt]{\begin{minipage}{0.0340869\unitlength}\raggedright $\beta_4$\end{minipage}}}%
    \put(0.97114695,0.18464297){\color[rgb]{0,0,0}\makebox(0,0)[lt]{\begin{minipage}{0.02913805\unitlength}\raggedright $\beta_1$\end{minipage}}}%
    \put(0.73968367,0.18022527){\color[rgb]{0,0,0}\makebox(0,0)[lt]{\begin{minipage}{0.04002552\unitlength}\raggedright $\beta_2$\end{minipage}}}%
    \put(0.19862509,0.19693032){\color[rgb]{0,0,0}\makebox(0,0)[lt]{\begin{minipage}{0.05230866\unitlength}\raggedright $a$\end{minipage}}}%
    \put(0.56874936,0.21423956){\color[rgb]{0,0,0}\makebox(0,0)[lt]{\begin{minipage}{0.05230866\unitlength}\raggedright $a$\end{minipage}}}%
    \put(0.97022557,0.08088303){\color[rgb]{0,0,0}\makebox(0,0)[lt]{\begin{minipage}{0.05230866\unitlength}\raggedright $a$\end{minipage}}}%
    \put(0.22700918,0.0566565){\color[rgb]{0,0,0}\makebox(0,0)[lt]{\begin{minipage}{0.05863957\unitlength}\raggedright $b$\end{minipage}}}%
    \put(0.58884669,0.03934727){\color[rgb]{0,0,0}\makebox(0,0)[lt]{\begin{minipage}{0.05863957\unitlength}\raggedright $b$\end{minipage}}}%
    \put(0.73848065,0.13057123){\color[rgb]{0,0,0}\makebox(0,0)[lt]{\begin{minipage}{0.05863957\unitlength}\raggedright $b$\end{minipage}}}%
  \end{picture}%
\endgroup%

%% file: Example_alternative_characterization_local_mult.pdf_tex
\begingroup%
  \makeatletter%
  \providecommand\color[2][]{%
    \errmessage{(Inkscape) Color is used for the text in Inkscape, but the package 'color.sty' is not loaded}%
    \renewcommand\color[2][]{}%
  }%
  \providecommand\transparent[1]{%
    \errmessage{(Inkscape) Transparency is used (non-zero) for the text in Inkscape, but the package 'transparent.sty' is not loaded}%
    \renewcommand\transparent[1]{}%
  }%
  \providecommand\rotatebox[2]{#2}%
  \ifx\svgwidth\undefined%
    \setlength{\unitlength}{135.29939614bp}%
    \ifx\svgscale\undefined%
      \relax%
    \else%
      \setlength{\unitlength}{\unitlength * \real{\svgscale}}%
    \fi%
  \else%
    \setlength{\unitlength}{\svgwidth}%
  \fi%
  \global\let\svgwidth\undefined%
  \global\let\svgscale\undefined%
  \makeatother%
  \begin{picture}(1,0.17085286)%
    \put(0,0){\includegraphics[width=\unitlength,page=1]{Example_alternative_characterization_local_mult.pdf}}%
  \end{picture}%
\endgroup%

%% file: Example_uniquely_labeled_polytope.pdf_tex
\begingroup%
  \makeatletter%
  \providecommand\color[2][]{%
    \errmessage{(Inkscape) Color is used for the text in Inkscape, but the package 'color.sty' is not loaded}%
    \renewcommand\color[2][]{}%
  }%
  \providecommand\transparent[1]{%
    \errmessage{(Inkscape) Transparency is used (non-zero) for the text in Inkscape, but the package 'transparent.sty' is not loaded}%
    \renewcommand\transparent[1]{}%
  }%
  \providecommand\rotatebox[2]{#2}%
  \ifx\svgwidth\undefined%
    \setlength{\unitlength}{272.29303491bp}%
    \ifx\svgscale\undefined%
      \relax%
    \else%
      \setlength{\unitlength}{\unitlength * \real{\svgscale}}%
    \fi%
  \else%
    \setlength{\unitlength}{\svgwidth}%
  \fi%
  \global\let\svgwidth\undefined%
  \global\let\svgscale\undefined%
  \makeatother%
  \begin{picture}(1,0.31087316)%
    \put(0,0){\includegraphics[width=\unitlength,page=1]{Example_uniquely_labeled_polytope.pdf}}%
    \put(0.42536425,0.15132107){\color[rgb]{0,0,0}\makebox(0,0)[lb]{\smash{$\{\textcolor{blue}{1},\textcolor{red}{1}\}$}}}%
    \put(0.54378206,0.34437604){\color[rgb]{0,0,0}\makebox(0,0)[lt]{\begin{minipage}{0.24973096\unitlength}\raggedright  \end{minipage}}}%
    \put(0.54770781,0.21427846){\color[rgb]{0,0,0}\makebox(0,0)[lb]{\smash{$\{\textcolor{blue}{1}\}$}}}%
    \put(0.58958561,0.12403962){\color[rgb]{0,0,0}\makebox(0,0)[lb]{\smash{$\{\textcolor{red}{1}\}$}}}%
    \put(0.5246973,0.0660431){\color[rgb]{0,0,0}\makebox(0,0)[lb]{\smash{$\{\textcolor{blue}{1}\}$}}}%
    \put(-0.0004727,0.11942276){\color[rgb]{0,0,0}\makebox(0,0)[lb]{\smash{$\{\textcolor{red}{1}\}$}}}%
    \put(0.17517721,0.12110151){\color[rgb]{0,0,0}\makebox(0,0)[lb]{\smash{$\{\textcolor{blue}{1}\}$}}}%
    \put(0.26541616,0.1387296){\color[rgb]{0,0,0}\makebox(0,0)[lb]{\smash{$\{\textcolor{blue}{1}\}$}}}%
    \put(0.24673881,0.06423003){\color[rgb]{0,0,0}\makebox(0,0)[lb]{\smash{$\{\textcolor{blue}{1}\}$}}}%
    \put(0.78094174,0.15395488){\color[rgb]{0,0,0}\makebox(0,0)[lb]{\smash{$\{\textcolor{red}{1},\textcolor{blue}{1}\}$}}}%
    \put(0.87755682,0.06619226){\color[rgb]{0,0,0}\makebox(0,0)[lb]{\smash{$\{\textcolor{blue}{1}\}$}}}%
    \put(0.94242346,0.12463244){\color[rgb]{0,0,0}\makebox(0,0)[lb]{\smash{$\{\textcolor{red}{1}\}$}}}%
    \put(0.89704965,0.21279455){\color[rgb]{0,0,0}\makebox(0,0)[lb]{\smash{$\{\textcolor{blue}{1}\}$}}}%
    \put(0.02805615,0.00650477){\color[rgb]{0,0,0}\makebox(0,0)[lb]{\smash{$S_1$}}}%
    \put(0.24106197,0.00702942){\color[rgb]{0,0,0}\makebox(0,0)[lb]{\smash{$\tilde{P}$}}}%
    \put(0.5280427,0.00650477){\color[rgb]{0,0,0}\makebox(0,0)[lb]{\smash{$P_1$}}}%
    \put(0.88112866,0.00650477){\color[rgb]{0,0,0}\makebox(0,0)[lb]{\smash{$P_2$}}}%
  \end{picture}%
\endgroup%

%% file: Example_upper_and_lower_path.pdf_tex
\begingroup%
  \makeatletter%
  \providecommand\color[2][]{%
    \errmessage{(Inkscape) Color is used for the text in Inkscape, but the package 'color.sty' is not loaded}%
    \renewcommand\color[2][]{}%
  }%
  \providecommand\transparent[1]{%
    \errmessage{(Inkscape) Transparency is used (non-zero) for the text in Inkscape, but the package 'transparent.sty' is not loaded}%
    \renewcommand\transparent[1]{}%
  }%
  \providecommand\rotatebox[2]{#2}%
  \ifx\svgwidth\undefined%
    \setlength{\unitlength}{308.21250823bp}%
    \ifx\svgscale\undefined%
      \relax%
    \else%
      \setlength{\unitlength}{\unitlength * \real{\svgscale}}%
    \fi%
  \else%
    \setlength{\unitlength}{\svgwidth}%
  \fi%
  \global\let\svgwidth\undefined%
  \global\let\svgscale\undefined%
  \makeatother%
  \begin{picture}(1,0.27340707)%
    \put(0,0){\includegraphics[width=\unitlength,page=1]{Example_upper_and_lower_path.pdf}}%
    \put(0.0209662,0.19007623){\color[rgb]{0,0,0}\makebox(0,0)[lb]{\smash{$P_1$}}}%
    \put(0.04055187,0.03302853){\color[rgb]{0,0,0}\makebox(0,0)[lb]{\smash{$P_2$}}}%
    \put(0.08407451,0.05308523){\color[rgb]{0,0,0}\makebox(0,0)[lb]{\smash{$P_3$}}}%
    \put(0.15087855,0.03433951){\color[rgb]{0,0,0}\makebox(0,0)[lb]{\smash{$P_4$}}}%
    \put(0.21735487,0.01988656){\color[rgb]{0,0,0}\makebox(0,0)[lb]{\smash{$P_5$}}}%
    \put(0,0){\includegraphics[width=\unitlength,page=2]{Example_upper_and_lower_path.pdf}}%
  \end{picture}%
\endgroup%

%% file: Example_dual_tropical_curve.pdf_tex
\begingroup%
  \makeatletter%
  \providecommand\color[2][]{%
    \errmessage{(Inkscape) Color is used for the text in Inkscape, but the package 'color.sty' is not loaded}%
    \renewcommand\color[2][]{}%
  }%
  \providecommand\transparent[1]{%
    \errmessage{(Inkscape) Transparency is used (non-zero) for the text in Inkscape, but the package 'transparent.sty' is not loaded}%
    \renewcommand\transparent[1]{}%
  }%
  \providecommand\rotatebox[2]{#2}%
  \ifx\svgwidth\undefined%
    \setlength{\unitlength}{140.75829091bp}%
    \ifx\svgscale\undefined%
      \relax%
    \else%
      \setlength{\unitlength}{\unitlength * \real{\svgscale}}%
    \fi%
  \else%
    \setlength{\unitlength}{\svgwidth}%
  \fi%
  \global\let\svgwidth\undefined%
  \global\let\svgscale\undefined%
  \makeatother%
  \begin{picture}(1,0.59954981)%
    \put(0,0){\includegraphics[width=\unitlength,page=1]{Example_dual_tropical_curve.pdf}}%
    \put(-0.00057118,0.3019725){\color[rgb]{0,0,0}\makebox(0,0)[lb]{\smash{$(\textcolor{blue}{1},\textcolor{red}{1})$}}}%
    \put(0.22793801,0.67035765){\color[rgb]{0,0,0}\makebox(0,0)[lt]{\begin{minipage}{0.48309765\unitlength}\raggedright  \end{minipage}}}%
    \put(0.23553234,0.41868758){\color[rgb]{0,0,0}\makebox(0,0)[lb]{\smash{$(\textcolor{blue}{1})$}}}%
    \put(0.31654371,0.24412281){\color[rgb]{0,0,0}\makebox(0,0)[lb]{\smash{$(\textcolor{red}{1})$}}}%
    \put(0.19101906,0.13193024){\color[rgb]{0,0,0}\makebox(0,0)[lb]{\smash{$(\textcolor{blue}{1})$}}}%
    \put(0.19749066,0.01675495){\color[rgb]{0,0,0}\makebox(0,0)[lb]{\smash{$P_1$}}}%
    \put(0,0){\includegraphics[width=\unitlength,page=2]{Example_dual_tropical_curve.pdf}}%
    \put(0.56268127,0.28336257){\color[rgb]{0,0,0}\makebox(0,0)[lb]{\smash{$\textcolor{red}{1}$}}}%
    \put(0.96788524,0.27677808){\color[rgb]{0,0,0}\makebox(0,0)[lb]{\smash{$\textcolor{red}{1}$}}}%
    \put(0.56335626,0.23842401){\color[rgb]{0,0,0}\makebox(0,0)[lb]{\smash{$\textcolor{blue}{1}$}}}%
    \put(0.79976289,0.07509526){\color[rgb]{0,0,0}\makebox(0,0)[lb]{\smash{$\textcolor{blue}{1}$}}}%
    \put(0.96476189,0.42929909){\color[rgb]{0,0,0}\makebox(0,0)[lb]{\smash{$\textcolor{blue}{1}$}}}%
  \end{picture}%
\endgroup%

%% file: Example_complete_lattice_path.pdf_tex
\begingroup%
  \makeatletter%
  \providecommand\color[2][]{%
    \errmessage{(Inkscape) Color is used for the text in Inkscape, but the package 'color.sty' is not loaded}%
    \renewcommand\color[2][]{}%
  }%
  \providecommand\transparent[1]{%
    \errmessage{(Inkscape) Transparency is used (non-zero) for the text in Inkscape, but the package 'transparent.sty' is not loaded}%
    \renewcommand\transparent[1]{}%
  }%
  \providecommand\rotatebox[2]{#2}%
  \ifx\svgwidth\undefined%
    \setlength{\unitlength}{446.1625bp}%
    \ifx\svgscale\undefined%
      \relax%
    \else%
      \setlength{\unitlength}{\unitlength * \real{\svgscale}}%
    \fi%
  \else%
    \setlength{\unitlength}{\svgwidth}%
  \fi%
  \global\let\svgwidth\undefined%
  \global\let\svgscale\undefined%
  \makeatother%
  \begin{picture}(1,1.02204914)%
    \put(0,0){\includegraphics[width=\unitlength]{Example_complete_lattice_path.pdf}}%
    \put(0.70826212,0.93241869){\color[rgb]{0,0,0}\makebox(0,0)[lb]{\smash{$\mult \cdot$ boundary}}}%
    \put(0.70866937,0.75168473){\color[rgb]{0,0,0}\makebox(0,0)[lb]{\smash{$2\cdot 1, 1\cdot 2,1\cdot 2, 1\cdot 4$}}}%
    \put(0.70826212,0.60966634){\color[rgb]{0,0,0}\makebox(0,0)[lb]{\smash{$2\cdot 1, 1\cdot 2,1\cdot 2$}}}%
    \put(0.70826212,0.46622085){\color[rgb]{0,0,0}\makebox(0,0)[lb]{\smash{$2\cdot 4, 1\cdot 2$}}}%
    \put(0.70826212,0.32277536){\color[rgb]{0,0,0}\makebox(0,0)[lb]{\smash{$1\cdot 4, 1\cdot 2$}}}%
    \put(0.70826212,0.17932986){\color[rgb]{0,0,0}\makebox(0,0)[lb]{\smash{$1\cdot 2, 1\cdot 2$}}}%
    \put(0.70826212,0.03588437){\color[rgb]{0,0,0}\makebox(0,0)[lb]{\smash{$1\cdot 2, 1\cdot 2$}}}%
  \end{picture}%
\endgroup%

%% file: Example_complete_lattice_path_2.pdf_tex
\begingroup%
  \makeatletter%
  \providecommand\color[2][]{%
    \errmessage{(Inkscape) Color is used for the text in Inkscape, but the package 'color.sty' is not loaded}%
    \renewcommand\color[2][]{}%
  }%
  \providecommand\transparent[1]{%
    \errmessage{(Inkscape) Transparency is used (non-zero) for the text in Inkscape, but the package 'transparent.sty' is not loaded}%
    \renewcommand\transparent[1]{}%
  }%
  \providecommand\rotatebox[2]{#2}%
  \ifx\svgwidth\undefined%
    \setlength{\unitlength}{268.29389466bp}%
    \ifx\svgscale\undefined%
      \relax%
    \else%
      \setlength{\unitlength}{\unitlength * \real{\svgscale}}%
    \fi%
  \else%
    \setlength{\unitlength}{\svgwidth}%
  \fi%
  \global\let\svgwidth\undefined%
  \global\let\svgscale\undefined%
  \makeatother%
  \begin{picture}(1,0.20472927)%
    \put(0,0){\includegraphics[width=\unitlength]{Example_complete_lattice_path_2.pdf}}%
    \put(0.08166239,0.18608134){\color[rgb]{0,0,0}\makebox(0,0)[lb]{\smash{$7$}}}%
    \put(0.14129848,0.12644525){\color[rgb]{0,0,0}\makebox(0,0)[lb]{\smash{$8$}}}%
    \put(0.33511579,0.18608134){\color[rgb]{0,0,0}\makebox(0,0)[lb]{\smash{$8$}}}%
    \put(0.39475188,0.12644525){\color[rgb]{0,0,0}\makebox(0,0)[lb]{\smash{$7$}}}%
    \put(0.58856919,0.18608134){\color[rgb]{0,0,0}\makebox(0,0)[lb]{\smash{$7$}}}%
    \put(0.70784138,0.06680915){\color[rgb]{0,0,0}\makebox(0,0)[lb]{\smash{$8$}}}%
    \put(0.84202259,0.18608134){\color[rgb]{0,0,0}\makebox(0,0)[lb]{\smash{$8$}}}%
    \put(0.96129478,0.06680915){\color[rgb]{0,0,0}\makebox(0,0)[lb]{\smash{$7$}}}%
  \end{picture}%
\endgroup%

%% file: Example_cross_ratio_floor_diagram.pdf_tex
\begingroup%
  \makeatletter%
  \providecommand\color[2][]{%
    \errmessage{(Inkscape) Color is used for the text in Inkscape, but the package 'color.sty' is not loaded}%
    \renewcommand\color[2][]{}%
  }%
  \providecommand\transparent[1]{%
    \errmessage{(Inkscape) Transparency is used (non-zero) for the text in Inkscape, but the package 'transparent.sty' is not loaded}%
    \renewcommand\transparent[1]{}%
  }%
  \providecommand\rotatebox[2]{#2}%
  \ifx\svgwidth\undefined%
    \setlength{\unitlength}{115.49998903bp}%
    \ifx\svgscale\undefined%
      \relax%
    \else%
      \setlength{\unitlength}{\unitlength * \real{\svgscale}}%
    \fi%
  \else%
    \setlength{\unitlength}{\svgwidth}%
  \fi%
  \global\let\svgwidth\undefined%
  \global\let\svgscale\undefined%
  \makeatother%
  \begin{picture}(1,0.08116876)%
    \put(0,0){\includegraphics[width=\unitlength,page=1]{Example_cross_ratio_floor_diagram.pdf}}%
  \end{picture}%
\endgroup%

%% file: Example_construction_cross_ratio_floor_diagram2.pdf_tex
\begingroup%
  \makeatletter%
  \providecommand\color[2][]{%
    \errmessage{(Inkscape) Color is used for the text in Inkscape, but the package 'color.sty' is not loaded}%
    \renewcommand\color[2][]{}%
  }%
  \providecommand\transparent[1]{%
    \errmessage{(Inkscape) Transparency is used (non-zero) for the text in Inkscape, but the package 'transparent.sty' is not loaded}%
    \renewcommand\transparent[1]{}%
  }%
  \providecommand\rotatebox[2]{#2}%
  \ifx\svgwidth\undefined%
    \setlength{\unitlength}{104.3340878bp}%
    \ifx\svgscale\undefined%
      \relax%
    \else%
      \setlength{\unitlength}{\unitlength * \real{\svgscale}}%
    \fi%
  \else%
    \setlength{\unitlength}{\svgwidth}%
  \fi%
  \global\let\svgwidth\undefined%
  \global\let\svgscale\undefined%
  \makeatother%
  \begin{picture}(1,0.56389948)%
    \put(0.00800454,1.1978132){\color[rgb]{0,0,0}\makebox(0,0)[lt]{\begin{minipage}{0.09587893\unitlength}\raggedright \end{minipage}}}%
    \put(0,0){\includegraphics[width=\unitlength,page=1]{Example_construction_cross_ratio_floor_diagram2.pdf}}%
    \put(-0.00231132,0.45326189){\color[rgb]{0,0,0}\makebox(0,0)[lt]{\begin{minipage}{0.08565194\unitlength}\raggedright $1$\end{minipage}}}%
    \put(-0.00214352,0.38115851){\color[rgb]{0,0,0}\makebox(0,0)[lt]{\begin{minipage}{0.08565194\unitlength}\raggedright $2$\end{minipage}}}%
    \put(-0.00290345,0.3092925){\color[rgb]{0,0,0}\makebox(0,0)[lt]{\begin{minipage}{0.08565194\unitlength}\raggedright $3$\end{minipage}}}%
    \put(0.34785239,0.09285501){\color[rgb]{0,0,0}\makebox(0,0)[lt]{\begin{minipage}{0.08565194\unitlength}\raggedright $4$\end{minipage}}}%
    \put(0.56878163,0.04381433){\color[rgb]{0,0,0}\makebox(0,0)[lt]{\begin{minipage}{0.08565194\unitlength}\raggedright $5$\end{minipage}}}%
    \put(0.78402244,0.04173691){\color[rgb]{0,0,0}\makebox(0,0)[lt]{\begin{minipage}{0.08565194\unitlength}\raggedright $6$\end{minipage}}}%
    \put(0.9294615,0.31199032){\color[rgb]{0,0,0}\makebox(0,0)[lt]{\begin{minipage}{0.08565194\unitlength}\raggedright $7$\end{minipage}}}%
    \put(0.61546253,0.56136123){\color[rgb]{0,0,0}\makebox(0,0)[lt]{\begin{minipage}{0.08565194\unitlength}\raggedright $8$\end{minipage}}}%
    \put(0.46063038,0.56937592){\color[rgb]{0,0,0}\makebox(0,0)[lt]{\begin{minipage}{0.08565194\unitlength}\raggedright $9$\end{minipage}}}%
    \put(0.07264516,0.50912453){\color[rgb]{0,0,0}\makebox(0,0)[lt]{\begin{minipage}{0.13295225\unitlength}\raggedright $p_1$\end{minipage}}}%
    \put(0.1165292,0.33400539){\color[rgb]{0,0,0}\makebox(0,0)[lt]{\begin{minipage}{0.13295225\unitlength}\raggedright $p_2$\end{minipage}}}%
    \put(0.18308664,0.25636174){\color[rgb]{0,0,0}\makebox(0,0)[lt]{\begin{minipage}{0.13295225\unitlength}\raggedright $p_3$\end{minipage}}}%
    \put(0.34715425,0.30058252){\color[rgb]{0,0,0}\makebox(0,0)[lt]{\begin{minipage}{0.13295225\unitlength}\raggedright $p_4$\end{minipage}}}%
    \put(0.52706404,0.26669868){\color[rgb]{0,0,0}\makebox(0,0)[lt]{\begin{minipage}{0.13295225\unitlength}\raggedright $p_5$\end{minipage}}}%
    \put(0.67660548,0.22052935){\color[rgb]{0,0,0}\makebox(0,0)[lt]{\begin{minipage}{0.13295225\unitlength}\raggedright $p_6$\end{minipage}}}%
    \put(0.8208365,0.13762126){\color[rgb]{0,0,0}\makebox(0,0)[lt]{\begin{minipage}{0.13295225\unitlength}\raggedright $p_7$\end{minipage}}}%
    \put(-0.03103787,1.37431445){\color[rgb]{0,0,0}\makebox(0,0)[lt]{\begin{minipage}{1.22203589\unitlength}\raggedright \end{minipage}}}%
    \put(0,0){\includegraphics[width=\unitlength,page=2]{Example_construction_cross_ratio_floor_diagram2.pdf}}%
  \end{picture}%
\endgroup%

%% file: Example_complete_floor_diagram.pdf_tex
\begingroup%
  \makeatletter%
  \providecommand\color[2][]{%
    \errmessage{(Inkscape) Color is used for the text in Inkscape, but the package 'color.sty' is not loaded}%
    \renewcommand\color[2][]{}%
  }%
  \providecommand\transparent[1]{%
    \errmessage{(Inkscape) Transparency is used (non-zero) for the text in Inkscape, but the package 'transparent.sty' is not loaded}%
    \renewcommand\transparent[1]{}%
  }%
  \providecommand\rotatebox[2]{#2}%
  \ifx\svgwidth\undefined%
    \setlength{\unitlength}{147.96098203bp}%
    \ifx\svgscale\undefined%
      \relax%
    \else%
      \setlength{\unitlength}{\unitlength * \real{\svgscale}}%
    \fi%
  \else%
    \setlength{\unitlength}{\svgwidth}%
  \fi%
  \global\let\svgwidth\undefined%
  \global\let\svgscale\undefined%
  \makeatother%
  \begin{picture}(1,0.8268719)%
    \put(0,0){\includegraphics[width=\unitlength]{Example_complete_floor_diagram.pdf}}%
  \end{picture}%
\endgroup%